\documentclass{cedram-aif}

\usepackage[english]{babel}
\usepackage[latin1]{inputenc}
\usepackage[T1]{fontenc}
\usepackage{lmodern}

\allowdisplaybreaks[4] % allows to split mult-line formulas on several pages
\usepackage{eucal}    % modifies the standard caligraphic font
\usepackage{graphicx} % to include images from .eps files
\usepackage{url}      % to include links to webpages

\newcommand{\dd}{\,\mathrm{d}}
\renewcommand{\Tilde}{\widetilde}

\renewcommand{\Bar}{\overline}

\newcommand{\RR}{\mathbb{R}}

\newcommand{\NN}{\mathbb{N}}
\newcommand{\ZZ}{\mathbb{Z}}

\newcommand{\cE}{\mathcal{E}}

\newcommand{\cH}{\mathcal{H}}
\newcommand{\cJ}{\mathcal{J}}

\newcommand{\cL}{\mathcal{L}}

\newcommand{\cO}{\mathcal{O}}
\newcommand{\cP}{\mathcal{P}}
\newcommand{\cQ}{\mathcal{Q}}

\newcommand{\cS}{\mathcal{S}}

\newcommand{\cW}{\mathcal{W}}

\newcommand{\boplus}{\mathop{\oplus}}

\newcommand{\rmi}{\mathrm{i}}

\newcommand{\dom}{\mathop{\mathcal{D}}}
\newcommand{\qdom}{\mathop{\mathcal{Q}}}
\newcommand{\ext}{\mathrm{ext}}
\newcommand{\out}{\mathrm{out}}
\newcommand{\cvx}{\mathrm{cvx}}

\DeclareMathOperator{\spec}{spec}
\DeclareMathOperator{\specdisc}{spec_\mathrm{disc}}
\DeclareMathOperator{\specess}{spec_\mathrm{ess}}
\DeclareMathOperator{\dist}{dist}
\DeclareMathOperator{\coloneqq}{:=}
\DeclareMathOperator{\cotan}{cotan}
\DeclareMathOperator{\vspan}{span}

\title[Robin eigenvalues]{Effective operators for Robin eigenvalues in domains with corners}
\alttitle{Op\'erateurs effectifs pour les valeurs propres de Robin dans des domaines \`a coins}

\author[M. Khalile]{\firstname{Magda} \lastname{Khalile}}
\address{Institut f\"ur Analysis, Leibniz Universit\"at Hannover, Welfengarten~1, 30167 Hannover, Germany}
\email{magda.khalile@math.uni-hannover.de}
\author[T. Ourmi\`eres-Bonafos]{\firstname{Thomas} \lastname{Ourmi\`eres-Bonafos}}
\address{Institut de Math\'ematiques de Marseille, Centre de Math\'ematiques et Informatique, Technop\^ole de Ch\^ateau Gombert, 39 rue Fr\'ed\'eric Joliot Curie, 13453 Marseille Cedex 13, France}
\email{thomas.ourmieres-bonafos@univ-amu.fr}
\author[K. Pankrashkin]{\firstname{Konstantin} \lastname{Pankrashkin}}
\address{Universit\'e Paris-Saclay, CNRS,  Laboratoire de math\'ematiques d'Orsay, 91405, Orsay, France}
\curraddr{(From March 1, 2020) Carl von Ossietzky Universit\"at, Institut f\"ur Mathematik, 26111 Oldenburg, Germany}
\email{konstantin.pankrashkin@uni-oldenburg.de}

\altkeywords{Valeur propre, Laplacien, condition aux limites de Robin, op\'erateur effectif, domaine non-lisse}


\begin{document}

\begin{abstract}
We study the eigenvalues of the Laplacian with a strong attractive Robin boundary condition in curvilinear polygons.
It was known from previous works that the asymptotics of several first eigenvalues
is essentially determined by the corner openings, while only rough estimates were available for the next eigenvalues.
Under some geometric assumptions, we go beyond the critical eigenvalue number
and give a precise asymptotics of any individual eigenvalue
by establishing a link with an effective Schr\"odinger-type operator on the boundary of the domain
with boundary conditions at the corners.
\end{abstract}

\begin{altabstract}
Nous \'etudions les valeurs propres du laplacien avec une condition de Robin fortement attractive dans des polygones curvilignes.
Gr\^ace \`a de pr\'ec\'edents travaux, on sait que le comportement asymptotique de quelques premi\`eres valeurs propres
est essentiellement d\'etermin\'e par les ouvertures des coins, alors que seules quelques estim\'ees grossi\`eres sont disponibles
pour les valeurs propres suivantes. Sous certaines hypoth\`eses g\'eom\'etriques, nous allons au-d\'el\`a du nombre critique
de valeurs propres et nous donnons un d\'eveloppement asymptotique pr\'ecis pour chaque valeur propre individuelle
en \'etablissant un lien avec un op\'erateur effectif de type Schr\"odinger
agissant sur le bord du domaine et muni de conditions aux limites aux coins.
\end{altabstract}

\maketitle

\tableofcontents

\section{Introduction}

\subsection{Problem setting and previous results}

Given a domain $\Omega\subset\RR^d$,  $d\geq 2$, with a suitably regular boundary $\partial\Omega$ and a parameter $\alpha>0$, we 
denote by $R^\Omega_\alpha$ the Laplacian in $L^2(\Omega)$ with the Robin condition
$\partial u/\partial\nu=\alpha u$ at the boundary, where $\nu$ is the outer unit normal.
The operator is rigorously defined using its quadratic form
\[
H^1(\Omega)\ni u\mapsto \int_\Omega |\nabla u|^2\dd x -\alpha \int_{\partial\Omega} u^2 \dd s
\]
with $\dd s$ being the $(d-1)$-dimensional Hausdorff measure, 
provided that the form is lower semibounded and closed.
The spectral properties of the operator $R^\Omega_\alpha$ have attracted a lot of attention during the last years, and
a recent review of various results and open problems can be found in the paper \cite{DFK} by Bucur, Freitas, Kennedy.
In the present paper we will be interested
in the behavior of the eigenvalues $E_n(R^\Omega_\alpha)$ in the asymptotic regime $\alpha\to+\infty$. Let us
recall some available results in this direction.

It seems that the study of the above asymptotic regime was first proposed by Lacey, Ockedon, Sabina \cite{los} when considering
a reaction-diffusion system, and Giorgi and Smits \cite{gs,gs2} obtained a number of estimates with links to the theory of enhanced surface superconductivity. 
Remark that for bounded Lipschitz domains $\Omega$ it follows from the general theory of Sobolev spaces
that there exists $C>0$ with $E_1(R^\Omega_\alpha)\ge -C\alpha^2$ for large $\alpha$ (see Lemma~\ref{rob-scale} below).
Lacey, Ockedon, Sabina in \cite{los} conjectured that under suitable regularity assumptions on $\Omega$
the lower bound can be upgraded to an asymptotics
\begin{equation}
 \label{eq-e11}
E_1(R^\Omega_\alpha)\sim -C_\Omega\alpha^2,
\end{equation}
with some $C_\Omega>0$, and they have shown that $C_\Omega=1$ for $C^4$ smooth domains. Levitin and Parnovski in \cite{LP}
have shown the asymptotics \eqref{eq-e11} for piecewise smooth domains satisfying the interior cone condition,
and they have shown that the constant $C_\Omega$ is explicitly determined through the spectra
of model Robin Laplacians by
\begin{equation}
    \label{ctx}
(-C_\Omega)=\inf_{x\in\partial\Omega} \inf\spec  (R^{T_x}_1),
\end{equation}
where $T_x$ is the tangent cone to $\Omega$ at $x$ and $\spec$ stands for the spectrum of the operator. Bruneau and Popoff in \cite{bp} gave an improved remainder estimate under the slightly stronger assumption that $\Omega$ is a so-called corner domain. We also mention the recent paper~\cite{kop2}
by Kova\v r\'{\i}k and Pankrashkin on non-Lipschitz domains, for which the eigenvalue behavior is completely different.

More precise estimates are available for smooth domains. The lower bound by Lou and Zhu~\cite{luzhu}
and the upper bound due to Daners and Kennedy~\cite{dk} imply that if $\Omega$ is a bounded $C^1$ domain,
then for each fixed $n\in\NN$ one has $E_n(R^\Omega_\alpha)\sim-\alpha^2$.
It seems that a more precise asymptotics was first obtained by Pankrashkin in \cite{nano13}:
it was shown that if $\Omega\subset\RR^2$ is bounded with a $C^3$ boundary, then
$E_1(R^\Omega_\alpha)=-\alpha^2-H_*\alpha+\cO(\alpha^{\frac{2}{3}})$, where $H_*$
is the maximum of the curvature of the boundary. Exner, Minakov and Parnovski in \cite{emp}
show that the asymptotics
\begin{equation}
   \label{hmax}
E_n(R^\Omega_\alpha)=-\alpha^2-H_*\alpha+\cO(\alpha^{\frac{2}{3}})
\end{equation}
holds for any fixed $n\in\NN$, and then Exner and Minakov \cite{em} obtained similar results
for a class of non-compact domains. Helffer and Kachmar \cite{HK} obtained a complete asymptotic expansion
for eigenvalues under the additional assumption that the curvature of the boundary
admits a single non-degenerate maximum. Pankrashkin and Popoff in \cite{pp15}
started the study of the multidimensional case: if $\Omega\subset \RR^d$
is a $C^3$ domain, then the asymptotics \eqref{hmax} holds with $H_*:=\max H$
and $H$ is defined as the sum of the principal curvatures at the boundary, i.e. $H=(d-1)$ times the mean curvature.
An analog of the asymptotics \eqref{hmax} for the first eigenvalue of Robin $p$-Laplacians
was obtained by Kova\v r\'{\i}k and Pankrashkin in~\cite{kop}.
Among possible applications of the asymptotics \eqref{hmax} one may mention 
various optimization issues concerning the eigenvalues of $R^\Omega_\alpha$.
It was conjectured by Bareket \cite{bareket} that among the domains $\Omega$ of fixed volume, for any $\alpha>0$
the quantity $E_1(R^\Omega_\alpha)$ is maximized by the balls. In this most general form, the conjecture was disproved by Freitas and Krej{\v{c}}i{\v{r}}{\'{\i}}k \cite{fk},
but an additional analysis shows that the conjecture may hold in a weaker form under additional restrictions on the geometry of $\Omega$, we refer to the papers
by Antunes, Freitas, Krej{\v{c}}i{\v{r}}{\'{\i}}k~\cite{afk}, Bandle and Wagner \cite{bandle}, Bucur, Ferone, Nitsch, Trombetti~\cite{DFNT},
Ferone, Nitsch, Trombetti~\cite{trombetti0}, Trani~\cite{trani} and Savo~\cite{savo} for domains on manifolds.
As noted by Pankrashkin and Popoff in \cite{pp15}, if the ball is the maximizer of $E_1(R^\Omega_\alpha)$ for all $\alpha>0$ in some class of smooth domains $\Omega$, then it is also the minimizer for the maximum mean curvature $H_*$ in the same class of domains, and this observation leads to some new inequalities for $H_*$, see e.g. Ferone, Nitsch, Trombetti \cite{trombetti},
and it was used to construct a number of counterexamples, for example, the asymptotics \eqref{hmax} was used by Krej{\v{c}}i{\v{r}}{\'{\i}}k and Lotoreichik \cite{kl1,kl2}
in the study of isoperimetric inequalities for Robin laplacians in exterior domains.

In \cite{pp15b} Pankrashkin and Popoff proposed an effective operator
to study the eigenvalues of  $R^\Omega_\alpha$. Namely, it was shown for $C^3$ domains $\Omega$, either bounded
or with a controllable behavior at infinity, that for any fixed $n\in\NN$ one has the asymptotics
\begin{equation}
  \label{eq-eff}
E_n(R^\Omega_\alpha)=-\alpha^2+E_n(L_\alpha)+\cO(1),
\end{equation}
where $L_\alpha$ is the Schr\"odinger operator in $L^2(\partial\Omega)$
acting as $L_\alpha=-\Delta_{\partial\Omega}-\alpha H$ with $\Delta_{\partial\Omega}$
being the Laplace-Beltrami operator on $\partial\Omega$. Kachmar, Keraval, Raymond \cite{kkr}
and Helffer, Kachmar, Raymond \cite{hkr} have shown that the same effective
operator appears in other spectral questions for $R^\Omega_\alpha$, e.g. the Weyl asymptotics
and the tunneling effect for $R^\Omega_\alpha$ are also controlled by those for $L_\alpha$
at the leading orders. Pankrashkin \cite{mmn} and Bruneau, Pankrashkin, Popoff \cite{bpp}
used the effective operator in order to study the accumulation of eigenvalues
for Robin Laplacians on some non-compact domains.

We also mention some related papers going slightly beyond the initial problem setting. Colorado and Garc\'\i a-Meli\'an \cite{colgar}
obtained some results in the same spirit for Laplacians with the boundary condition $\partial u/\partial\nu=\alpha p u$ for variable functions $p$ and $\alpha\to+\infty$.
Filinovskii in \cite{filin} obtained the estimate $\liminf_{\alpha\to+\infty} \alpha^{-1}\partial E_1(R^\Omega_\alpha)/\partial\alpha\le -1$.
Helffer and Pankrashkin \cite{hp} studied the exponential splitting between the first two eigenvalues of $R^\Omega_\alpha$
in a domain $\Omega$ with two congruent corners. Cakoni, Chaulet and Haddar \cite{cch} have shown that, in a sense, the only finite accumulation points of the eigenvalues
of $R^\Omega_\alpha$ for large positive $\alpha$ are the Dirichlet Laplacian eigenvalues of $\Omega$.

\subsection{Main results}

\begin{figure}

\centering

\includegraphics[width=55mm]{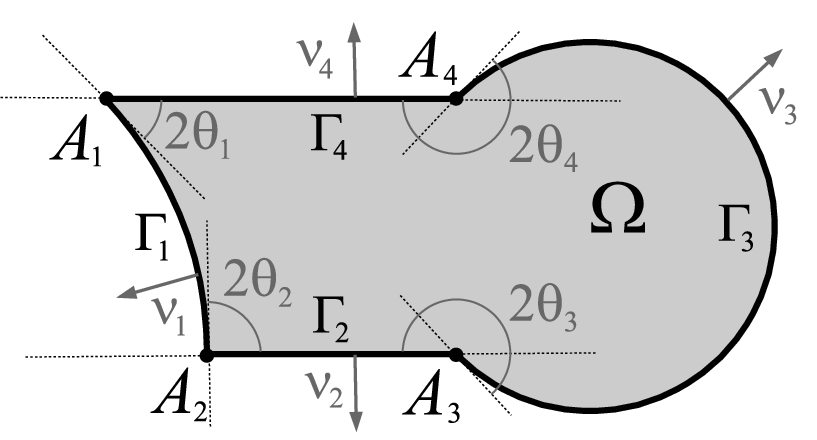}

\caption{An example of a curvilinear polygon $\Omega$ with four vertices and
sides of constant curvature. The vertices $A_1$ and $A_2$ are convex, and the vertices $A_3$
and $A_4$ are concave. One has $H_1<0$, $H_3>0$ and $H_2=H_4=0$.\label{fig-domain}}

\end{figure}

In the present paper, we would like to combine the existing results and techniques in order to study
the eigenvalues of $R^\Omega_\alpha$ for the case of $\Omega\subset \RR^2$ being a \emph{curvilinear polygon}
and to better understand the role of corners in the spectral properties.
A complete definition of curvilinear polygons will be given later in the text (Subsection~\ref{ssec-poly}), and
for the moment we restrict ourselves to a less formal intuitive
definition: one says that a bounded planar domain $\Omega$ is a curvilinear polygon
if its boundary is smooth except near $M$ points (vertices) $A_1,\dots, A_M$, and
if $\Gamma_{j-1}$ and $\Gamma_j$ are two smooth pieces of boundary meeting at $A_j$,
then the \emph{half}-angle $\theta_j$ between them  (measured inside $\Omega$) is non-degenerate and non-trivial, i.e.
$\theta_j\notin\{0,\pi/2,\pi\}$. We say that a vertex $A_j$ is convex if $\theta_j<\pi/2$,
otherwise it is called concave. Furthermore, let $ H_j$
be the curvature defined on $\Gamma_j$, with the convention that $H_j\ge 0$ for convex domain,
and $\ell_j$ denotes the length of $\Gamma_j$.
We refer to Figure~\ref{fig-domain} for an illustration.

\begin{figure}

\centering

\includegraphics[width=65mm]{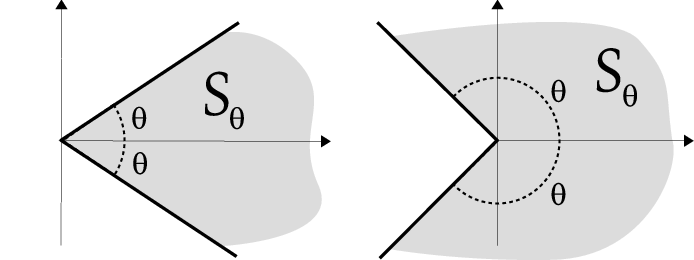}

\caption{The infinite sector $\cS_\theta$ for $\theta<\pi/2$ (left) and $\theta>\pi/2$ (right).\label{fig-sector}}

\end{figure}

Using the general result \eqref{ctx} one is reduced first to the study of Robin Laplacians
in all possible tangent sectors, which have a simple structure in two dimensions. Namely,
consider the infinite planar sectors
$\cS_\theta:=\big\{(x_1,x_2):\, \big|\arg(x_1+\rmi x_2)\big|<\theta\big\}\subset\RR^2$,
see Figure~\ref{fig-sector}, then the tangent sector to $\Omega$ at $A_j$
is a rotated copy of $\cS_{\theta_j}$, while at all other points the tangent sectors are
isometric to $\cS_{\frac{\pi}{2}}$, which is just the half-plane. Denote by
$T_\theta$ the Laplacian in $\cS_\theta$ with the normalized Robin boundary condition $\partial u/\partial\nu=u$.
Its spectral properties were studied in detail by Khalile and Pankrashkin \cite{KP}
and are summarized below in Proposition~\ref{prop-sector}.
For the current presentation we remark
that the essential spectrum is always $[-1,+\infty)$, and, in addition, it has $\kappa(\theta)<\infty$ discrete
eigenvalues $\cE_1(\theta),\dots\cE_{\kappa(\theta)}(\theta)$, while $\kappa(\theta)=0$ for $\theta\ge \pi/2$
(i.e. there are no discrete eigenvalues at all if the sector is concave),
and $\cE_1(\theta)=-1/\sin^2\theta$ for $\theta<\pi/2$. Furthermore,
one has $\kappa(\theta)=1$ for $\frac{\pi}{6}\le \theta<\frac{\pi}{2}$. Hence, with $\Omega$ we  associate
the following objects:
\begin{align*}
K&:=\kappa(\theta_1)+\dots +\kappa(\theta_M),\\
\cE&:=\text{ the disjoint union of } \big\{ \cE_n(\theta_j), \ n=1,\dots,\kappa(\theta_j)\big\}, \  j\in\{1,\dots, M\},\\
\cE_n&:=\text{ the $n$th element of $\cE$ when numbered in the non-decreasing order.}
\end{align*}
Khalile in~\cite{khalile2} gives an improved version of \eqref{ctx} for curvilinear polygons, namely,
for each $n\in\{1,\dots,K\}$ one has $E_n(R^\Omega_\alpha)=\cE_n\alpha^2+\cO(\alpha^{\frac{4}{3}})$,
while the remainder estimate can be improved for polygons with straight sides, and $E_{K+n}(R^\Omega_\alpha)\sim-\alpha^2$ for each $n\in\NN$.
(We remark that paper \cite{khalile2} was in turn motivated by the earlier work by Bonnaillie-No\"el and Dauge~\cite{BND}
on magnetic Neumann Laplacians in corner domains.) Therefore, the behavior of the first $K$ eigenvalues 
at the leading order is determined by the corners only, so one might call them \emph{corner-induced}.
In the present work we would like to understand in greater detail the asymptotics
of the higher eigenvalues  $E_{K+n}(R^\Omega_\alpha)$ with a fixed $n\in\NN$, which will be referred to as \emph{side-induced}.
 As the main term $(-\alpha^2)$
in the asymptotics is the same as in the smooth case, one might expect that their behavior
should take into account the geometry of the boundary away from the corners,
so that a kind of an effective Schr\"odinger-type operator may appear by analogy with \eqref{eq-eff}.
On the other hand, one might expect that the corners should contribute to the effective operator:
due to the singularities at the vertices, some boundary conditions might be needed in order
to make the effective operator self-adjoint. It seems that the only result obtained in this direction
is the one by Pankrashkin \cite{nano15}:  if $\Omega$ is the exterior of a convex polygon with
side lengths $\ell_j$, then for any fixed $n$ one has $E_n(R^\Omega_\alpha)=-\alpha^2+E_n(\oplus_j D_j) +\cO(\alpha^{-\frac{1}{2}})$
as $\alpha\to+\infty$, where $D_j$ is the Dirichlet Laplacian on $(0, \ell_j)$. Remark that this result is in agreement with what precedes:
as all the corners are concave, one simply has $K=0$. We are going to obtain a result in the same spirit
for a more general case, in particular, by allowing the presence of convex corners.

Our analysis will be based 
on the notion of  \emph{non-resonant} convex vertex (it will be seen from the proof
that concave vertices are much easier to deal with), which is formulated in terms of a model Robin eigenvalue problem on a truncated sector.
Namely, for $\theta\in(0,\pi/2)$ and $r>0$ let $A^\pm_r$ be the two points lying on the two boundary rays of the sector $\cS_\theta$
at the distance $r>0$ from the origin $O$, and let $B_r$ be the intersection point of the straight lines
passing through $A^\pm_r$ perpendicular to the boundary, see Figure~\ref{figure-intro}. Denote by $\cS^r_\theta$
the quadrangle $OA^+_r B_r A^-_r$ and by $N^r_{\theta}$ the Laplacian $u\mapsto-\Delta u$ in $\cS^r_\theta$
with the Robin boundary condition $\partial u/\partial\nu =u$ at $OA^\pm_r$
and the Neumann boundary condition at $A^\pm_rB_r$. Using rather standard methods one sees that the first $\kappa(\theta)$
eigenvalues of $N^r_{\theta}$ converge to those of $T_\theta$ as $r\to+\infty$ (Lemma~\ref{nalph}), and the non-resonance
condition is a hypothesis on the behavior of the next eigenvalue.
We say that a half-angle $\theta$ is \emph{non-resonant} if for some $C>0$ one has $E_{\kappa(\theta)+1}(N^r_\theta)\ge-1+ C/r^2$ for large $r$.
One shows in Proposition~\ref{prop-good}, using a combination of a separation of variables with a monotonicity argument 
that  all half-angles $\theta\in\big[\frac{\pi}{4},\frac{\pi}{2}\big)$ are non-resonant.

\begin{figure}

\centering

\includegraphics[width=50mm]{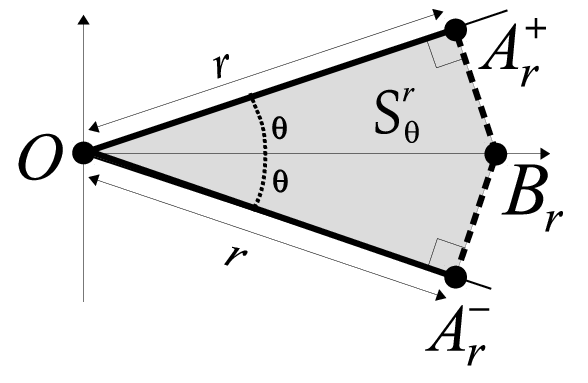}

\caption{The quadrangle $\cS^r_\theta$.\label{figure-intro}}

\end{figure}

In order to concentrate on the contribution of the corners, let us discuss first the case when $\Omega$ is a polygon with straight edges.
We denote
\[
D_j:=\text{the Dirichlet Laplacian on $(0,\ell_j)$}.
\]
Our main result reads as follows:
\begin{theo}\label{thm-poly}
Let $\Omega\subset\RR^2$ be a polygon with $M$ vertices, half-angles $\theta_j$ and side lengths $\ell_j$.
Assume that each $\theta_j$ is either concave or non-resonant, then for any fixed $n\in\NN$ and $\alpha\to+\infty$
there holds
\[
E_{K+n}(R^\Omega_\alpha)=-\alpha^2+E_n\big( \boplus\nolimits_{j=1}^M D_j\big)+\cO\big(\tfrac{\log\alpha}{\sqrt{\alpha}}\big).
\]
\end{theo}
As it will be seen in the proof, using the Dirichlet-Neumann bracketing and the non-resonance condition, 
it is quite elementary to obtain the two-sided estimate
\begin{align*}
-\alpha^2+E_n\big( \boplus\nolimits_{j=1}^M N_j\big)+\cO\big(\tfrac{\log\alpha}{\alpha}\big)
&\le E_{K+n}(R^\Omega_\alpha),\\
E_{K+n}(R^\Omega_\alpha)&\le-\alpha^2+E_n\big( \boplus\nolimits_{j=1}^M D_j\big)+\cO\big(\tfrac{\log\alpha}{\alpha}\big),
\end{align*}
where $N_j$ is the \emph{Neumann} laplacian on $(0,\ell_j)$, and one easily sees that the difference between the lower and upper bounds is of order $1$.
It takes then the most efforts to close this gap and to show that it is the upper bound which gives the main term of the eigenvalue asymptotics,
and this is the main contribution of the present paper.

Using the above observation that all obtuse angles $\theta$ are non-resonant with $\kappa(\theta)=1$, one arrives
at the following  corollary:
\begin{coro}
Let $\Omega\subset\RR^2$ be a  polygon with $M$ vertices, half-angles $\theta_j$ and sides of length $\ell_j$.
Assume that $\theta_j\ge \pi/4$ for all $j$, then for any $n\in\NN$ and $\alpha\to +\infty$ there holds
\[
E_{K+n}(R^\Omega_\alpha)=-\alpha^2+ E_n\big(\boplus\nolimits_{j=1}^M D_j\big) + \cO \big( \tfrac{\log\alpha}{\sqrt{\alpha}}\big),
\]
where $K$ is the number of convex vertices.
\end{coro}
It is an important point that a different eigenvalue asymptotics can arise if no condition is imposed on the corners.
In order to see it, remark first that in the situation of  Theorem~\ref{thm-poly} one has
\begin{equation}
     \label{ekn}
\lim_{\alpha\to+\infty}\big(E_{K+1}(R^\Omega_\alpha)+\alpha^2\big)=E_1\big( \boplus\nolimits_{j=1}^M D_j\big)>0.
\end{equation}
On the other hand, the computations by McCartin~\cite{mcc4} for an explicit configuration (which we review in Subsection~\ref{rem-triangle})
give the following result
\begin{prop}
Let $\Omega$ be an equilateral triangle of side length $\ell>0$. Then for $\alpha\to+\infty$ there holds
$E_n(R^\Omega)=-4\alpha^2+o(1)$ for $n\in\{1,2,3\}$ and
\begin{equation}
   \label{e3n}
E_{3+n}(R^\Omega)=-\alpha^2+E_n(L)+o(1) \text{ for any fixed } n\in\NN,
\end{equation}
where $L$ is the Laplacian on $(0,3\ell)$ with the \emph{periodic} boundary condition.
\end{prop}
For the equilateral triangle one has indeed $K=3$, while $E_1(L)=0$. Eq.~\eqref{e3n} implies $\lim_{\alpha\to+\infty}\big(E_{K+1}(R^\Omega_\alpha)+\alpha^2)=0$,
which contradicts~\eqref{ekn}. This means that the half-angle $\theta=\frac{\pi}{6}$ is resonant (i.e. it does not satisfy the above non-resonance condition).
We remark that the non-resonance condition we use is strictly adapted to our proof method and is not supposed to be optimal,
but we are not aware of any suitable alternative. In fact our choice is strongly motivated by some recent studies of Laplacians
in domains collapsing on graphs, and some analogies with waveguides
and possible reformulations of the non-resonance condition are discussed in Subsection~\ref{ssec-nonres2}.

For the case of curvilinear polygons, a number of additional difficulties arise due to the presence of non-trivial curvatures on the sides, and we were not able to
study the most general case in the present text (the most important technical obstacles are discussed in Subsection~\ref{ssec-var}).
Nevertheless, we were able to consider two important cases.

First, we consider the case when the maximum curvature is not attained at the corners.
\begin{theorem}\label{thm2}
Denote $H_{j,*}:= \max_{s\in[0,\ell_j]}H_j(s)$,
$H_*:=\max_j H_{j,*}$, and assume that
\[
\text{for all $j\in\{1,\dots,M\}$ there holds $H_j(0)\ne H_*$ and $H_j(\ell_j)\ne H_*$,}
\]
and that  each corner of $\Omega$ is either concave or non-resonant.
Then for each $n\in\NN$ and $\alpha\to +\infty$ one has
\begin{align}
E_{K+n}(R^\Omega_\alpha)&=-\alpha^2+E_n\big( \boplus\nolimits_{j:H_{j,*}=H_*}(D_j-\alpha H_j)\big)+\cO(1), \label{eq-dir00}\\
& = -\alpha^2+E_n\big( \boplus\nolimits_{j:H_{j,*}=H_*}(N_j-\alpha H_j)\big)+\cO(1). \nonumber
\end{align}
\end{theorem}
In fact, the proof of Theorem~\ref{thm2} appears to be less involved than the one of Theorem~\ref{thm-poly}:
the main ingredient is that the eigenvalues of $D_j-\alpha H_j$ are exponentially close to those
of $N_j-\alpha H_j$, which is a simple consequence of Agmon-type estimates, hence,
the contribution of the boundary conditions to be imposed at the vertices is very small (as the eigenfunctions
are concentrated near the set on which the curvature attains its maximal value). 
The remainder $\cO(1)$ is the same as for the effective operator in \eqref{eq-eff}
obtained for smooth domains. Furthermore, under suitable geometric assumptions
a complete asymptotic expansion can be obtained, see Subsection~\ref{ssec-thm2}. 
We remark that the non-resonance condition is still used in the proof of Theorem~\ref{thm2},
and we have no intuition on what kind of asymptotics can be expected without additional
conditions on the corners.

The second important case we were able to study is as follows:
\begin{equation}
  \label{h-intro}
\text{the curvatures $H_j$ are constant, and we denote $H_*:=\max H_j$;}
\end{equation}
i.e. each side is either a line segment or a circle arc. We
explicitly mention that $H_j$ can be different for different $j$. Then we obtain the following result, which is in the same spirit as Theorem~\ref{thm-poly}:

\begin{theorem}\label{thm3}
Assume that all corners are concave or non-resonant and that \eqref{h-intro} is satisfied,
then for any fixed $n\in\NN$ and $\alpha\to +\infty$ 
one has the asymptotics
\begin{equation}
   \label{ekn2}
E_{K+n}(R^\Omega_\alpha)= -\alpha^2-H_*\alpha-\tfrac{1}{2}\, H_*^2 + E_n\big(\boplus\nolimits_{j:H_j=H_*} D_j\big) + \cO \big( \tfrac{\log\alpha}{\sqrt{\alpha}}\big).
\end{equation}
\end{theorem}

Remark that \eqref{ekn2} can be formally viewed as a particular case of the asymptotics~\eqref{eq-dir00} as the terms $-\tfrac{1}{2}\, H_*^2+\cO \big( \tfrac{\log\alpha}{\sqrt{\alpha}}\big)$
in~\eqref{ekn2} can be viewed as a resolution of the remainder $\cO(1)$ in~\eqref{eq-dir00}. The presence of the new term $\tfrac{1}{2}\, H_*^2$ was not observed in earlier papers
on Robin eigenvalues.

The text is organized as follows. In Section~\ref{sec-prel} we recall basic tools from the functional analysis (min-max based eigenvalue estimates, distance between subspaces, Sobolev trace theorems) and study or recall the spectral properties
of some model operators (Robin Laplacians on intervals and infinite sectors).
Section~\ref{sec-trunc} is devoted to the study of Robin Laplacians in convex sectors truncated in a special way: we obtain some estimates for the eigenvalues
and decay estimate for the eigenfunctions, then we introduce the new notion of non-resonant angle
and show that it is satisfied by the obtuse angles. 
In Section~\ref{sec-thm-poly} we prove Theorem~\ref{thm-poly}, i.e. the case of polygons with straight sides.
We first decompose the polygon into vertex neighborhoods and side neighborhoods,
and apply Dirichlet-Neumann bracketing in order to give first a rough eigenvalue estimate
in terms of the direct sum of operators in each part. This approach appears to be sufficient for the upper bound.
The proof of the lower bound is much more involved and represents the main contribution of the paper.
Our approach is based on the construction of an identification operator between functions in $\Omega$ and functions
on the boundary satisfying the Dirichlet boundary condition at the vertices. This machinery was initially proposed
by Post \cite{post} for the analysis of thin branching domains, and it was already used by Pankrashkin~\cite{nano15}
to study the Robin Laplacians in the exterior of convex polygons. 
The difference with the present case comes from the fact that we want to obtain estimates on the $K+n$ eigenvalue, namely the identification only applies to the orthogonal complement of the $K$ first  eigenfunctions of $R^\Omega_\alpha$.
The strategy consists in proving that the lowest eigenspaces of the polygon are close in a suitable sense to the ones of the vertex neighborhoods. The non-resonance condition is then used to analyze their orthogonal complement: it allows us to obtain a control on the trace of the eigenfunctions at the boundary of the vertex neighborhoods 
 which gives a necessary input for the eigenvalue estimates.
We explicitly remark that
our analysis is not based on the construction of quasimodes for the operators in play, but on a construction
of test functions which are not in the operator domains. In particular, we do not see
any sufficiently direct way to obtain a complete asymptotic expansion for the eigenvalues.

In section~\ref{sec-thm23} we discuss the case of curvilinear polygons. We still need a special decomposition of the domain
into pieces of a very special form as well as the existence of some diffeomorphisms and
special cut-off functions. The procedure is summarized
 at the beginning of the section while a
 complete justification of the geometric constructions is given in Appendix~\ref{appa},
as we are not aware of suitable constructions in the existing literature.
As in the case of straight polygons we then estimate the portion of the operator
in each piece of the domain, which gives the sought upper bound and some lower bound.
We prove Theorem~\ref{thm2} in Subsection~\ref{ssec-thm2}
by showing that the lower and upper bounds are close enough (when compared with
the order of the eigenvalue) under the geometric condition imposed.
In Subsection~\ref{ssec-thm3} we then prove Theorem~\ref{thm3}. With preparation in the preceding subsection,
the proof scheme is almost identical to the one of Theorem~\ref{thm-poly}, and differences are mostly of a technical nature.

Finally, in Section~\ref{ssec55} we discuss possible extensions of the results, in particular,
we show that some angles do not satisfy the non-resonance condition and
give a different eigenvalue asymptotics, and we explain some links between our study
and the spectral analysis of waveguides. As already mentioned, Appendix~\ref{appa} contains
some geometric constructions in curvilinear sectors, and we believe that they can be of use
for other problems involving differential operators in domains with corners.

\medskip

\noindent {\bf Acknowledgments.}
A large part of this paper was written while Thomas Ourmi\`eres-Bonafos was supported by a public grant as part of the ``Investissement d'avenir'' project, reference ANR-11-LABX-0056-LMH, LabEx LMH. Now, Thomas Ourmi\`eres-Bonafos is supported by the ANR "D\'efi des savoirs (DS10) 2017" programm, reference ANR-17-CE29-0004, project molQED.
The initial version of the paper was significantly reworked following referee's suggestions. We are very grateful to her/him for very constructive criticisms
which resulted in a significant improvement of the text and in the inclusion of several additional results.

\section{Preliminaries}\label{sec-prel}

\subsection{Notation}

For $x=(x_1,x_2)\in\RR^2$ and $y=(y_1,y_2)\in\RR^2$ we will use the length $|x|=\sqrt{x_1^2+x_2^2}$, the scalar product
$x\cdot y=x_1 y_1+x_2y_2$ and the wedge product $x\wedge y=x_1 y_2-x_2 y_1$. 
In this paper we only deal with real-valued operators, so we prefer to work with real Hilbert spaces in order to have a simpler writing.
Let $\cH$  be a Hilbert space and $u,v\in\cH$, then we denote by $\langle u,v\rangle_\cH$ the scalar product of $u$ and $v$.
It will be sometimes shortened to $\langle u,v\rangle$ if there is no ambiguity in the choice of the Hilbert space, and the same applies to the associated norm $\|\cdot\|_\cH$.
For a self-adjoint operator $(A,\dom(A))$ in $\cH$, with $\dom(A)$ being the operator domain,
we denote by $\spec(A)$, $\specdisc(A)$ and $\specess(A)$ the spectrum of $A$, its discrete spectrum and its essential spectrum, respectively. 
For $n\in\NN:=\{1,2,3,\dots\}$, by $E_n(A)$ we denote the $n$th discrete eigenvalue of $A$ (if it exists) when enumerated in the non-decreasing order counting the multiplicities.
If the operator $A$ is semibounded from below, then $\qdom(A)$ denotes the domain of its sesquilinear form, and the value of the sesquilinear form
on two vectors $u,v\in \qdom(A)$ will be denoted by $A[u,v]$.

\subsection{Min-max principle and its consequences}

Let  $\mathcal H$ be an infinite-dimensional Hilbert space and $A$ be a lower semibounded self-adjoint operator in $\cH$, with $A\ge -c$ for some $c\in \mathbb R$.
Recall that $\qdom(A)$ equipped with the scalar product 
\[
\cQ(A)\times\cQ(A)\ni(u,v)\mapsto A[u,v]+(c+1)\langle u,v\rangle_\cH
\]
is a Hilbert space. The following result, giving a variational characterization of eigenvalues (usually referred to as the \emph{min-max} principle), 
is a standard tool in the spectral theory of self-adjoint operators, see e.g. \cite[Section XIII.1]{RS} or \cite[Section 10.2]{bs}:

\begin{prop}\label{minmax}
Let $\Sigma:=\inf \specess (A)$ if $\specess (A)\ne \emptyset$, otherwise set $\Sigma:=+\infty$.
Let $n \in \mathbb N$ and $D$ be a dense subspace of $\qdom(A)$.
Define the $n$th \emph{Rayleigh quotient}  $\Lambda_n(A)$ of $A$ by
\[
\Lambda_n(A):= \inf_{G\subset D:\, \dim G=n} \sup_{u \in G\setminus\{0\}} \frac{A[u,u] }{\|u\|^2_\cH},
\]
then one and only one of the following two assertions is true:
\begin{itemize}
\item $\Lambda_n(A)<\Sigma$ and $E_n(A)=\Lambda_n(A)$.
\item $\Lambda_n(A)=\Sigma$ and $\Lambda_m(A)=\Lambda_n(A)$ for all $m\ge n$.
\end{itemize}
\end{prop}
The following corollary is also well known, see e.g. \cite[Sec.~10.2., Thm.~5]{bs}:
\begin{coro}\label{corol-codim}
Let $A$ and $B$ be lower semibounded self-adjoint operators in an infinite-dimensional Hilbert space $\cH$.
Assume that there exists $d\in\NN$ and a $d$-dimensional subspace $D$ such that $\qdom(A)=\qdom(B)\oplus D$
and that $A[u,u]=B[u,u]$ for all $u\in\qdom(B)$, then
 $\Lambda_n(B)\le \Lambda_{n+d}(A)$ for all $n\in\NN$.
\end{coro}

Furthermore, the following min-max-based eigenvalue estimate will be of use to compare the eigenvalues
of operators acting in different spaces. It was introduced and used by Exner and Post in~\cite[Lemma~2.1]{ep} as well as by Post in~\cite[Lemma~2.2]{post}:

\begin{prop}\label{prop6}
Let $\cH$ and $\cH'$ be infinite-dimensional Hilbert spaces,
$B$ be a non-negative self-adjoint operator with a compact resolvent in $\cH$ and
$B'$ be a lower semibounded self-adjoint operator in $\cH'$.
Pick $n\in\NN$ and assume that there exists a linear map  $J:\qdom(B)\to\qdom (B')$
and constants $\varepsilon_1 \ge 0$ and $\varepsilon_2 \ge 0$ such that
$\varepsilon_1 < 1/\big(1+E_n(B)\big)$ and that for any
$u\in\qdom (B)$ one has
\begin{align*}
\|u\|^2_{\cH}-\|Ju\|^2_{\cH'}&\le \varepsilon_1 \big(  B[u,u]+\|u\|^2_\cH \big),\\ %\label{post1}\\
B'[Ju,Ju]-B[u,u]&\le \varepsilon_2 \big(  B[u,u]+\|u\|^2_\cH \big), %\label{post2}
\end{align*}
then $\Lambda_n(B')\le E_n(B) + \dfrac{\big(E_n(B)\varepsilon_1+\varepsilon_2\big)\big(1+E_n(B)\big)}{1-\big(1+E_n(B)\big)\varepsilon_1}$.
\end{prop}

\subsection{Distance between closed subspaces}\label{sub-dist}

We will use the well-known notion of a distance between two closed subspaces:
\begin{defi}\label{defin4}
Let $E$ and $F$ be closed subspaces of a Hilbert space $\mathcal H$ and denote by $P_E$ and $P_F$ the orthogonal projectors in $\cH$ on $E$ and $F$ respectively.
The \emph{distance $d(E,F)$ between $E$ and $F$} is defined by
\[
d(E,F):=\sup_{x\in E,\, x\ne 0} \frac{\|x-P_F x\|} {\|x\|}\equiv \lVert P_E-P_F P_E\rVert\equiv \lVert P_E -P_E P_F \rVert.
\]
\end{defi}
One easily sees that the distance is not symmetric, i.e.  $d(E,F)\ne d(F,E)$ in general,
but the triangular inequality is satisfied, i.e. $d(E,G) \leq d(E,F) + d(F,G)$ for any closed subspaces $E,F,G$.
Furthermore, we will need the following result due to Helffer and Sj\"ostrand~\cite[Proposition~2.5]{hs}
allowing to estimate the distance between two subspaces in a special case.
\begin{prop}\label{propdist2}
Let $A$ be a self-adjoint operator in a Hilbert space $\cH$ and $I\subset \mathbb R$ be a compact interval.
For some $n\in\NN$ let $\mu_1,...,\mu_n \in I$ and $\psi_1,...,\psi_n \in \dom(A)$ be linearly independent vectors,
then we denote
\begin{align*}
\varepsilon&:= \max_{j\in\{1,...,n\}} \big\| (A-\mu_j) \psi_j\big\|, 
\quad \eta:= \tfrac{1}{2} \dist\big(I, (\spec A)\backslash I \big), \\
\lambda&:=\text{ the smallest eigenvalue of the Gram matrix }\big(\langle \psi_j,\psi_k\rangle\big)_{j,k\in\{1,\dots,n\}}.
\end{align*}
If $\eta>0$, then the distance $d(E,F)$ between the subspaces
\begin{align*}
E&:=\vspan\{\psi_1,...,\psi_n\},\\
F&:=\text{the spectral subspace associated with $A$ and~$I$}
\end{align*}
satisfies $d(E,F) \leq \dfrac{\varepsilon}{\eta} \sqrt{\dfrac{n}{\lambda}}$.
\end{prop}

\subsection{Laplacians with mixed boundary conditions and a trace estimate}

In what follows we will deal with numerous Laplacians with various combinations of boundary conditions.
 In order to simplify the writing, we introduce the following definition:

\begin{defi}[Laplacians with mixed boundary conditions]
Let $U\subset \RR^d$ be an open set and $\Gamma_D$, $\Gamma_N$, $\Gamma_R$ be disjoint
subsets of $\partial U$ such that $\overline{\Gamma_D\cup \Gamma_N\cup \Gamma_R}=\partial U$.
In addition, let $\alpha\in\RR$, then by the Laplacian in $\Omega$ with Dirichlet  condition
at $\Gamma_D$, Neumann condition at $\Gamma_N$ and $\alpha$-Robin condition
at $\Gamma_R$ we mean the self-adjoint operator $A$ in $L^2(U)$ with
\begin{align*}
A[u,u]&=\int_U |\nabla u|^2\dd x -\alpha \int_{\Gamma_R} |u|^2\, \dd s, \\
\qdom (A)&=\big\{u\in H^1(U): \, u=0 \text{ at } \Gamma_D\big\},
\end{align*}
where $\dd s$ is the $(d-1)$-dimensional Hausdorff measure on $\partial U$, provided that
the above expression defines a closed semibounded from below sesquilinear form (which is the case
for  bounded Lipschitz domains $U$).
Informally, the operator $A$ acts then as $u\mapsto -\Delta u$
on suitably regular functions $u$ in $U$ satisfying
$u=0$ at $\Gamma_D$, $\partial_\nu u =0$ at $\Gamma_N$, $\partial_\nu u =\alpha u$ at $\Gamma_R$,
where $\partial_\nu$ stands for the outer normal derivative.
\end{defi}

We will need a variant of the Sobolev trace inequality on scaled domains.

\begin{lemm}\label{sob-scale}\label{rob-scale}
Let $U\subset\RR^d$ be a bounded Lipschitz domain, then there exists $c>0$ such that
\[
\int_{\partial(t U)} f^2\dd s\le c\Big(t\varepsilon\int_{tU} |\nabla f|^2\dd x+ \dfrac{1}{t\varepsilon} \int_{t U} f^2\dd x\Big)
\]
for all $t>0$, $f\in H^1(tU)$, $\varepsilon\in(0,1]$.
In particular, if for $\alpha>0$ one denotes by $R^{tU}_\alpha$ the Laplacian in $tU$ with $\alpha$-Robin condition at the whole boundary,
then there exists $C>0$ such that $R^{tU}_\alpha\ge -C\alpha^2$ for $\alpha t$ sufficiently large.
\end{lemm}

\begin{proof}
The standard trace inequality, see e.g. Grisvard~\cite[Theorem 1.5.1.10]{gris}, implies that there exists $c>0$
such that
\begin{equation}
   \label{grsv1}
\int_{\partial U} u^2\dd s\le c\Big(\varepsilon \int_U |\nabla u|^2\dd x + \dfrac{1}{\varepsilon}\int_U u^2\dd x\Big)
\text{ for all } u\in H^1(U) \text{ and } \varepsilon\in(0,1].
\end{equation}
For $f\in L^2(tU)$ denote by $f_t\in L^2(U)$ the function given by $f_t(x)=f(tx)$, then $f\in H^1(t U)$ if and only if $f_t\in H^1(U)$.
Using \eqref{grsv1} we see that
\begin{equation}
    \label{temp0}
\int_{\partial U} f_t^2\dd s\le c\Big(\varepsilon \int_{U} |\nabla f_t|^2\dd x+ \dfrac{1}{\varepsilon}\int_{U} f_t^2\dd x\Big)
\text{ for all } f\in H^1(tU),\ \varepsilon\in(0,1].
\end{equation}
and using the change of variables $x=y/t$ one easily obtains
\begin{align*}
\int_{\partial U} f_t^2\dd s&=\int_{\partial U} f(tx)^2\dd s=t^{1-d}\int_{\partial(t U)} f(y)^2\dd s,\\
\int_{U} |\nabla f_t|^2\dd x&=\int_{ U} t^2 \big|(\nabla f)(tx)\big|^2\dd x=t^{2-d}\int_{tU} |\nabla f(y)|^2\dd y,\\
\int_{U} f_t^2\dd x&=\int_U f(tx)^2\dd x=t^{-d} \int_{tU} f(y)^2\dd y.
\end{align*}
The substitution of these three equalities into \eqref{temp0} gives the desired trace inequality.
Furthermore, for all $f\in H^1(tU)$ and $\varepsilon\in (0,1]$ one has
\begin{multline*}
R^{tU}_\alpha[f,f]=\int_{tU}|\nabla f|^2\dd x -\alpha\int_{\partial(t U)} f^2\dd s\\
\ge (1-c\alpha t\varepsilon) \int_{tU}|\nabla f|^2\dd x-\dfrac{c\alpha}{t\varepsilon} \int_{tU} f^2\dd x.
\end{multline*}
Hence, taking $\varepsilon:=1/(c\alpha t)$ we arrive at $R^{tU}_\alpha\ge -c^2 \alpha^2$.
\end{proof}

%
%Lemma~\ref{sob-scale}~allows one to give a first generic generic lower bound for Robin Laplacians:
%
%\begin{corol}\label{rob-scale}
%Let $U\subset\RR^d$ be a bounded open set with  Lipschitz boundary and $t>0$.
%For $\alpha>0$, let $R^{tU}_\alpha$ be the Laplacian in $tU$ with the $\alpha$-Robin boundary condition at the whole boundary,
%then there exists $C>0$ such that $R^{tU}_\alpha\ge -C\alpha^2$ for $\alpha t$ sufficiently large.
%\end{corol}
%
%\begin{proof}
%We continue using the notation of Lemma~\ref{sob-scale}, then 
%\end{proof}
%
\subsection{One-dimensional model operators}

Let us recall some eigenvalue estimates for Laplacians on finite intervals with a combination of boundary conditions.

\begin{prop}\label{prop21}
For $\delta>0$ and $\alpha>0$, let $L_D$ be the Laplacian on $(0,\delta)$
with $\alpha$-Robin condition at $0$ and the Dirichlet boundary condition at $\delta$, then
for $\alpha\delta\to +\infty$ there holds
$E_1(L_D)=-\alpha^2\big(1+ \cO(e^{-\delta\alpha})\big)$ and $E_2(L_D)\ge 0$.
\end{prop}
The result is obtained by direct computations, details can be found e.g. in \cite[Lemma 4]{nano13}.

\begin{prop}\label{prop22}
For $\delta>0$, $\alpha>0$ and $\beta\ge 0$, let $L_N$ denote the Laplacian on $(0,\delta)$
with $\alpha$-Robin condition at $0$ and $\beta$-Robin condition
at $\delta$, then for $\alpha\delta\to +\infty$ and $\beta\delta\to 0^+$ one has
$E_1(L_N)=-\alpha^2\big(1+ \cO(e^{-\alpha \delta})\big)$ and $E_2(L_N)\ge 1/\delta^2$.
\end{prop}

\begin{proof}
The estimate for the first eigenvalue was already obtained by direct computations e.g. in \cite[Lemma 3]{nano13}.
To study the second eigenvalue, let $B_\beta$ be the Laplacian on $(0,\delta)$ with the Dirichlet boundary condition
at $0$ and $\beta$-Robin condition at $\delta$, then the sesquilinear form of $B_\beta$ is a restriction of the sesquilinear form of $L_N$, and
$\qdom(L_N)=H^1(0,\delta)$ only differs from $\qdom(B_\beta)=\big\{f\in H^1(0,\delta):\, f(0)=0\big\}$
by a one-dimensional subspace. It follows by Corollary~\ref{corol-codim} that
$E_2(L_N)\ge E_1(B_\beta)$. Now, let us obtain a lower bound for $B_\beta$. For $\beta=0$ one obtains simply the
Laplacian on $(0,\delta)$ with the Dirichlet boundary condition at $0$ and the Neumann boundary condition at $\delta$, and $E_1(B_0)=\pi^2/(4\delta^2)$.
By Lemma~\ref{sob-scale} there is $c>0$ such that
\[
f(\delta)^2\le c\Big(\delta\int_0^\delta (f')^2\dd t +\dfrac{1}{\delta}\int_0^\delta f^2\dd t\Big) \text{ for all } f\in H^1(0,\delta).
\]
It follows that for $f\in \qdom(B_\beta)\equiv\qdom(B_0)$ one has
\[
B_\beta[f,f]=\int_0^\delta (f')^2\dd t -\beta f(\delta)^2\ge (1-c\beta\delta)\int_0^\delta (f')^2\dd t-\dfrac{c\beta}{\delta} \int_0^\delta f^2\dd t,
\]
and the min-max principle implies that for $\beta\delta\to 0^+$ one has
\[
E_1(B_\delta)\ge (1-c\beta\delta)E_1(B_0)-\dfrac{c\beta\delta}{\delta^2}=\dfrac{(1-c\beta\delta)\pi^2-4c\beta\delta}{4\delta^2}\ge \dfrac{1}{\delta^2}. \qedhere
\]
\end{proof}

\subsection{Robin Laplacians in infinite sectors}\label{sec-sectors}

Now, let us recall some basic facts on Robin laplacians in infinite sectors.

\begin{defi}\label{stheta}
For $\theta\in(0,\pi)$ denote by $\cS_\theta$ the following infinite sector of opening angle~$2\theta$:
\[
\cS_\theta=\big\{ (x_1,x_2)\in\mathbb R^2: -\theta<\arg(x_1+\rmi x_2)<\theta\big\}, \quad 0<\theta<\pi,
\]
see Figure~\ref{fig-sector} in the introduction. For $\theta=\pi/2$ one obtains simply the half-plane $\RR_+\times\RR$.
\end{defi}

%
%
%\begin{figure}
%\centering
%\includegraphics[height=30mm]{sector.eps}
%\caption{The infinite sector $\cS_\theta$ for $\theta<\pi/2$ (left) and $\theta>\pi/2$ (right).\label{fig-sector}}
%\end{figure}
%

The following proposition summarizes the basic properties of the associated Robin Laplacians proved in the paper \cite{KP} by Khalile and Pankrashkin:

\begin{prop}\label{prop-sector}
For $\theta\in(0,\pi)$ and $\alpha>0$, let $T_{\theta,\alpha}$ be the Laplacian  on $\cS_\theta$ with $\alpha$-Robin condition
at the whole boundary, then:
\begin{itemize}
\item the operator $T_{\theta,\alpha}$ is well-defined, lower semibounded and is unitarily equivalent to $\alpha^2 T_{\theta,1}$ for all $\theta\in (0,\pi)$ and $\alpha>0$,
\item $\specess(T_{\theta,\alpha})=[-\alpha^2, +\infty)$  for all $\theta\in (0,\pi)$ and $\alpha>0$,
\item the discrete spectrum of $T_{\theta,\alpha}$ is non-empty if and only if $\theta <\frac{\pi}{2}$, in particular,
\[
E_1(T_{\theta,\alpha})=-\alpha^2/\sin^2\theta \text{ for } \theta\in \big(0, \tfrac{\pi}{2}\big).
\]
\item if one denotes
\[
\text{$\kappa(\theta):=\text{the number of discrete eigenvalues of }T_{\theta,\alpha}$,}
\] 
which is independent of $\alpha$, then
\begin{itemize}
\item[$\circ$] $\kappa(\theta) <+\infty$ and $\theta\mapsto \kappa(\theta)$ is non-increasing with $\kappa(0^+)=+\infty$,
\item[$\circ$] for all $\frac{\pi}{6}\le\theta<\frac{\pi}{2}$ one has $\kappa(\theta)=1$,
\end{itemize}
\item there exist $b>0$ and $B>0$
such that if $n\in\big\{1,\dots,\kappa(\theta)\big\}$
and $\psi_{n,\alpha}$ is an eigenfunction of $T_{\theta,\alpha}$
for the $n$th eigenvalue, then for any $\alpha>0$ one has the Agmon-type decay estimate
\begin{equation}
    \label{agmon1}
\int_{\cS_\theta}  e^{b\alpha|x|}\Big( \,\dfrac{1}{\alpha^2}\,\big| \nabla\psi_{n,\alpha}(x)\big|^2+\psi^2_{n,\alpha}(x)\Big)\dd x\le B\|\psi_{n,\alpha}\|^2_{L^2(\cS_\theta)}.
\end{equation}
\end{itemize}
\end{prop}
Remark that the above properties of $T_{\theta,\alpha}$ are also of relevance for Steklov-type eigenvalue problems in domains with corners, see e.g. the papers
by Ivrii \cite{ivrii} and Levitin, Parnovski, Polterovich, Sher \cite{LPPS}. For a subsequent use we give a special name to the eigenvalues of the above operator with $\alpha=1$:
\[
\cE_n(\theta):=E_n(T_{\theta,1})  \text{ for $\theta\in\big(0,\tfrac{\pi}{2}\big)$ and $n\in\{1,\dots,\kappa(\theta)\big\}$}.
\]
Remark that due to Proposition~\ref{prop-sector} one has $\cE_1(\theta)=-1/\sin^2 \theta$ and
\begin{gather*}
\cE_n(\theta)<-1,
\quad
E_n(T_{\theta,\alpha})=\cE_n(\theta)\,\alpha^2<-\alpha^2,\\
\text{for }
\theta\in\big(0,\tfrac{\pi}{2}\big), \  n\in\{1,\dots,\kappa(\theta)\big\},
\ \alpha>0.
\end{gather*}

\section{Analysis in truncated convex sectors}\label{sec-trunc}

\subsection{Robin Laplacians in truncated convex sectors}

Recall that the infinite sectors $\cS_\theta$ are defined above in Definition~\ref{stheta}. Let us introduce their truncated versions.

\begin{defi}[Truncated convex sector $\cS_\theta^r$]\label{def25}
Let $\theta\in\big(0,\frac{\pi}{2}\big)$ and $r>0$. Consider the points
\[
A^\pm_r= r (\cos \theta, \pm \sin\theta)\in \partial \cS_\theta, \quad
B_r=r \big(1/\cos \theta,0\big)\in \cS_\theta,
\]
and denote by $\cS_\theta^r$ the interior of the quadrangle $O A^+_r B_r A^-_r$ 
(remark that the sides $B_r A^\pm_r$ are orthogonal to $\partial \cS_\theta$ at $A^\pm_r$, see Fig.~\ref{figure1}).
We will distinguish between two parts of the boundary of $\cS_\theta^r$, namely, we set
\begin{gather*}
\partial_* \cS_\theta^r:=\partial \cS_\theta^r\mathop{\cap} \partial \cS_{\theta}:= \text{polygonal chain } A^+_rOA^-_r,\\
\partial_\ext \cS_\theta^r:=\partial \cS_\theta^r\setminus \partial_* \cS_\theta^r:=\text{polygonal chain } A^+_rB_rA^-_r.
\end{gather*}
\end{defi}

\begin{figure}[b]

\centering

\includegraphics[width=50mm]{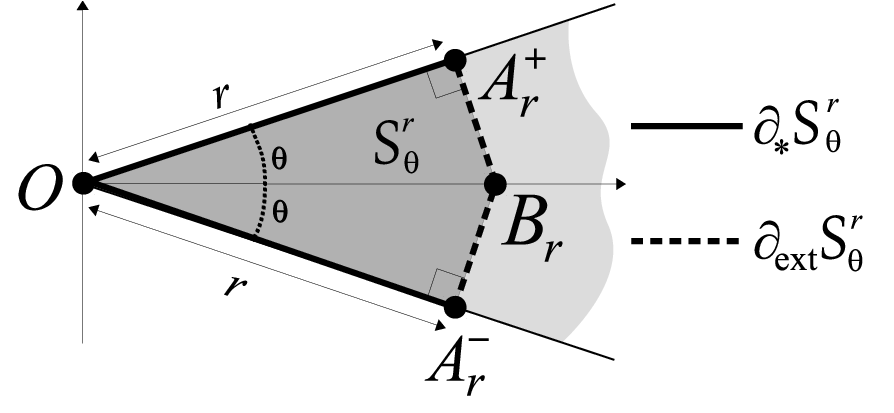}

\caption{The truncated sector $\cS_\theta^r$ is shaded (see Definition~\ref{def25}). The part of the boundary $\partial_*\cS_\theta^r$ is indicated by the thick solid line,
and the part of the boundary $\partial_\ext\cS_\theta^r$ is shown as the thick dashed line.\label{figure1}}

\end{figure}

%\begin{figure}
%
%\centering
%
%\includegraphics[height=35mm]{trunc.eps}
%
%\caption{The truncated sector $\cS_\theta^r$ is shaded (see Definition~\ref{def25}). The part of the boundary $\partial_*\cS_\theta^r$ is indicated by the thick solid line,
%and the part of the boundary $\partial_\ext\cS_\theta^r$ is shown as the thick dashed line.\label{figure1}}
%
%\end{figure}

In what follows we will need some properties of three operators associated with $\cS_\theta^r$, namely, for $\theta\in\big(0,\frac{\pi}{2}\big)$, $\alpha>0$ and $r>0$
we introduce:
\begin{equation}
    \label{eq-tqr}
\begin{aligned}
D_{\theta,\alpha}^r:=\,&  \begin{minipage}[t]{90mm}\raggedright the Laplacian in $\cS_\theta^r$ with $\alpha$-Robin condition
at $\partial_* \cS_\theta^r$ and \emph{Dirichlet}  condition at $\partial_\ext \cS_\theta^r$,\end{minipage}\\
N_{\theta,\alpha}^r:=\,&  \begin{minipage}[t]{90mm}\raggedright the Laplacian in $\cS_\theta^r$ with $\alpha$-Robin  condition
at $\partial_* \cS_\theta^r$ and \emph{Neumann}  condition at $\partial_\ext \cS_\theta^r$,\end{minipage}\\
R_{\theta,\alpha}^r:=\,&  \begin{minipage}[t]{90mm}\raggedright the Laplacian in $\cS_\theta^r$ with $\alpha$-Robin condition
\emph{at the whole boundary}.\end{minipage}
\end{aligned}
\end{equation}
We remark that $N_{\theta,\alpha}^r$ will play a key role in the subsequent considerations (in particular, see Subsection~\ref{ssec-nonres}),
while the other two operators will be used mostly for auxiliary constructions. The following properties of the three operators are easily established by a standard
routine computation:
\begin{lemma}\label{lem-scale}
For $t,r>0$ denote by $\Xi_t$ the unitary operators (dilations)
$\Xi_t:L^2(\cS_{\theta}^{tr})\to L^2(\cS_{\theta}^r)$, $(\Xi_t u)(x)=t \,u(tx)$.
Let $X_{\theta,\alpha}^r$ be any of the three operators $D_{\theta,\alpha}^r$, $N_{\theta,\alpha}^r$, $R_{\theta,\alpha}^r$, then
$X_{\theta,t\alpha}^{r} \Xi_t=t^2 \Xi_t X_{\theta,\alpha}^{tr}$, which then gives the eigenvalue identities
\begin{equation}
     \label{eq-scale1}
E_n(X_{\theta,\alpha}^{r})=\alpha^2 E_n(X_{\theta,1}^{\alpha r}) \text{ for all } n\in\NN.
\end{equation}
\end{lemma}

Let us show that in a suitable asymptotic regime the lowest eigenvalues of the Robin-Dirichlet Laplacians
$D^{r}_{\theta,\alpha}$
are close to the Robin eigenvalues of the associated infinite sectors:

\begin{lemma}\label{dalph}
For some $c>0$ one has
\[
E_n(D_{\theta,\alpha}^r)=E_n(T_{\theta,\alpha})+\cO(\alpha^2 e^{-c\alpha r}) \equiv \alpha^2\big(\cE_n(\theta)+
\cO(e^{-c\alpha r})\big)
\]
for $n\in\big\{1,\dots,\kappa(\theta)\big\}$,
and $E_{\kappa(\theta)+1}(D_{\theta,\alpha}^r)\ge-\alpha^2$ as  $\alpha r\to+\infty$.
\end{lemma}

\begin{proof}
The result is quite standard and is based on the fact that the Robin eigenfunctions of the infinite sectors
satisfy an Agmon-type estimate at infinity, but we provide a proof for the sake of completeness.
In view of the above scaling \eqref{eq-scale1} it is sufficient to study the case $\alpha=1$ and $r\to+\infty$. 
Recall that
\begin{gather*}
D^r_{\theta,1}[u,u]=\int_{\cS^r_\theta} |\nabla u|^2\dd x - \int_{\partial_* \cS^r_\theta} u^2\dd s,\\
\qdom(D^r_{\theta,1})=\big\{H^1(\cS^r_\theta):\, u=0 \text{ on } \partial_\ext \cS^r_\theta\big\},\\
T_{\theta,1}[u,u]=\int_{\cS_\theta} |\nabla u|^2\, dx - \int_{\partial \cS_\theta} u^2\dd s, \quad\qdom(T_{\theta,1})=H^1(\cS_\theta).
\end{gather*}
The min-max principle gives $E_n(D^r_{\theta,1})\ge \Lambda_n(T_{\theta,1})$  for any $r>0$ and $n\in\NN$.
For $n\in \big\{1,\dots, \kappa(\theta)\big\}$ one has $\Lambda_n(T_{\theta,1})=E_n(T_{\theta,1})\equiv \cE_n(\theta)$,
while $\Lambda_{\kappa(\theta)+1}(T_{\theta,1})=\inf\specess T_{\theta,1}=-1$. This proves the required lower bounds.

To prove the upper bound, let us pick $n\in \big\{1,\dots, \kappa(\theta)\big\}$
and let $\psi_j$, $j=1,\dots, n$, be eigenfunctions of the operator $T_{\theta,1}$ in the infinite sector corresponding to the $n$ first eigenvalues
and chosen to form an orthonormal family, i.e.
\[
\langle \psi_j,\psi_k\rangle_{L^2(\cS_\theta)}=\delta_{j,k}, \quad T_{\theta,1}[\psi_j,\psi_k]=\cE_j(\theta)\, \delta_{j,k},
\quad
j,k=1,\dots,n.
\]
Let $\chi_0,\chi_1:\RR\to [0,1]$ be smooth functions such that $\chi_0=1$ in $\big(-\infty, \frac12]$, $\chi_0=0$ in $[1,\infty)$
and $\chi_0^2+\chi_1^2=1$. We define $\chi_j^r:\RR^2\to \RR$ by $\chi^r_j(x)=\chi_j\big(|x|/r\big)$, $j=0,1$,
and $\psi_j^r:\cS_\theta\to \RR$ by $\psi_j^r:=\chi^r_0\psi_j$, $j=1,\dots,n$,
and keep the same symbols for the restrictions of these functions to $\cS^r_\theta$.
Remark that the functions $\psi_j^r$ belong to $H^1(\cS_\theta^r)$ and vanish at $\partial_\ext \cS_\theta^r$,
i.e. they belong to $\qdom(D^r_{\theta,1})$ and can be used to estimate the Rayleigh quotients.
Let us now use the Agmon-type estimate \eqref{agmon1} with suitable $b>0$ and $B>0$ for the eigenfunctions $\psi_j$.
Denote
\[
C^r_{j,k}:=\int_{\cS_\theta} (\chi_1^r)^2 \psi_j \psi_k \dd x,
\]
then $|C^r_{j,k}|\le \frac{1}{2}\,(C^r_{j,j}+C^r_{k,k})$
and
\begin{align*}
C_{j,j}^r=\int_{\cS_\theta}(\chi_1^r)^2 \psi_j^2 \dd x&\le
\int_{\cS_\theta:\, |x|>r/2} \psi_j^2 \dd x\\
&\le
e^{-\frac{br}{2}}\int_{\cS_\theta:\, |x|>r/2} e^{b|x|}\psi_j^2 \dd x\le Be^{-\frac{br}{2}}.
\end{align*}
Therefore, for large $r$ one has $C^r_{j,k}=\cO(e^{-cr})$ with  $c:=\frac{1}{2}\, b$
and
\[
\langle \psi_j^r,\psi_k^r\rangle_{L^2(\cS_\theta^r)}=\langle \psi_j,\psi_k\rangle_{L^2(\cS_\theta)}-C^r_{j,k}=\delta_{j,k}+\cO(e^{-cr}).
\]
In particular, for large $r$ the functions $\psi_j^r$ are linearly independent. Using similar estimates we obtain
\begin{gather*}
\int_{\cS_\theta} \nabla(\chi_1^r \psi_j)^2 \dd x=\cO(e^{-cr}),\\
\int_{\cS_\theta^r} \nabla \psi^r_j \cdot \nabla \psi^r_k \dd x=
\int_{\cS_\theta} \nabla \psi_j \cdot \nabla \psi_k  \dd x+\cO(e^{-cr}).
\end{gather*}
To estimate the quantities
\[
G^r_{j,k}:=\int_{\partial \cS_\theta} (\chi_1^r)^2 \psi_j\, \psi_k\dd s
\]
we remark again that $|G^r_{j,k}|\le \frac{1}{2}\,(G^r_{j,j}+G^r_{k,k})$, and using $\chi_1^r \psi_j$ as a test function in the inequality $T_{\theta,1}\ge -(\sin \theta)^{-2}$
for the Robin Laplacian in the sector we obtain
\[
G^r_{j,j}= \int_{\partial \cS_\theta} (\chi_1^r)^2 \psi_j^2\dd s\le 
\int_{\cS_\theta} \nabla(\chi_1^r \psi_j)^2\dd x
+\dfrac{1}{\sin ^2\theta}\int_{\cS_\theta}\chi_1^r \psi_j^2 \dd x=\cO(e^{-cr}),
\]
which implies $G^r_{j,k}=\cO(e^{-cr})$ and
\[
\int_{\partial_* S_\theta^r} \psi_j^r\psi_k^r\dd s=\int_{\partial \cS_\theta} \psi_j\, \psi_k\dd s
-G^r_{j,k}=\int_{\partial \cS_\theta} \psi_j\, \psi_k\dd s+\cO(e^{-cr}).
\]
Denote $L_r:=\vspan(\psi_1^r,\dots,\psi_n^r)$, which is an $n$-dimensional subspace of $\qdom(D^r_{\theta,1})$ for large~$r$.
For any function $\psi$ of the form
\[
\psi=\xi_1\psi_1^r+\dots+\xi_n\psi^r_n\in L_r, \quad \xi=(\xi_1,\dots,\xi_n)\in\RR^n
\]
one has, 
due to the preceding estimates,  $\|\psi\|^2_{L^2(\cS_\theta^r)}=|\xi|^2 \big(1+\cO(e^{-cr})\big)$
and
\begin{align*}
D^r_{\theta,1}[\psi,\psi]&=\sum_{j,k=1}^n \Big(T_{\theta,1}[\psi_j,\psi_k]\,  + \cO(e^{-cr}) \Big)\xi_j \xi_k\\
&=\sum_{j,k=1}^n \Big( \cE_j(\theta)\, \delta_{j,k}+\cO(e^{-cr})\Big) \xi_j \xi_k\le
\big( \cE_n(\theta) +\cO(e^{-cr}) \big) |\xi|^2,
\end{align*}
and an application of the min-max principle gives
\[
E_n(D^r_{\theta,1})\le \sup_{\psi\in L_r,\, \psi\ne 0} \dfrac{D^r_{\theta,1}[\psi,\psi]}{\|\psi\|^2_{L^2(\cS_\theta^r)}}\le \cE_n(\theta) +\cO(e^{-cr}). \qedhere
\]
\end{proof}

In order to obtain an analogous result on the behavior of the first $\kappa(\theta)$ eigenvalues of $N^r_{\theta,\alpha}$ we need a preliminary estimate.
\begin{defi}\label{pth1}
For $\theta\in\big(0,\frac{\pi}{2}\big)$ and $0<\rho<r$ denote
\[
\cP^{r,\rho}_\theta:=\cS_\theta^r\setminus\overline{\cS_\theta^\rho}\equiv\text{the hexagon $B_\rho A^+_\rho A^+_r B_ r A^-_r A^-_\rho$,}
\]
where one uses the same notation as in the definition of $\cS_\theta^r$, see Figures~\ref{figure1} and~\ref{figp}(a).
We again split the boundary of $\cP^{r,\rho}_\theta$ into two parts by setting
\begin{gather*}
\partial_* \cP^{r,\rho}_\theta:=\partial \cP^{r,\rho}_\theta\mathop{\cap} \partial \cS_{\theta}:= \text{the union of the segments $[A^\pm_\rho,A^\pm_r]$},\\
\partial_\ext \cP^{r,\rho}_\theta:=\partial \cP^{r,\rho}_\theta\setminus \partial_* \cP_\theta^{r,\rho}.
\end{gather*}
\end{defi}

\begin{figure}[b]
\centering
\begin{tabular}{cc}
\begin{minipage}[c]{67mm}
\begin{center}
\includegraphics[width=60mm]{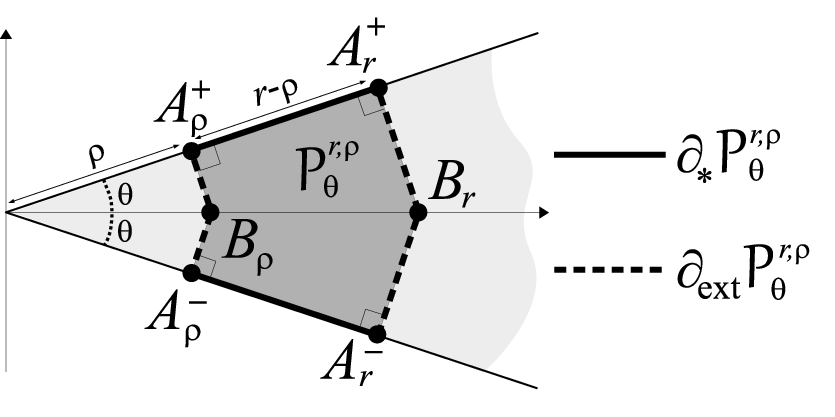}
\end{center}
\end{minipage}
&
\begin{minipage}[c]{37mm}
\begin{center}
\includegraphics[width=30mm]{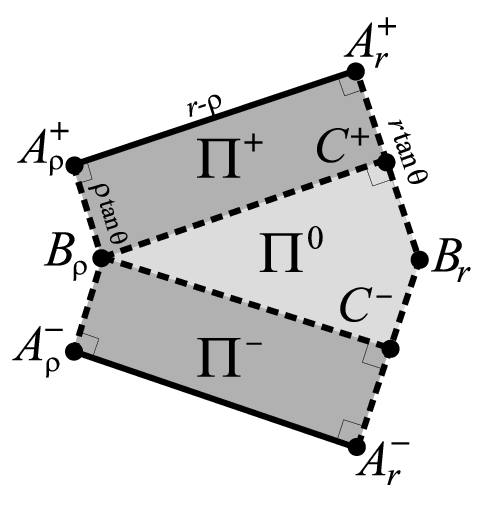}
\end{center}
\end{minipage}\\
(a) & (b) 
\end{tabular}
\caption{(a) Polygon $\cP^{r,\rho}_\theta$, see Definition~\ref{pth1}. 
(b)  Decomposition of $\cP^{r,\rho}_\theta$ for the proof of  Lemma~\ref{lem29}. The solid/dashed lines correspond to Robin/Neumann boundary conditions.\label{figp}}
\end{figure}

\begin{lemma}\label{lem29}
Let $P^{r,\rho}_{\theta,\alpha}$ denote the Laplacian in $\cP^{r,\rho}_\theta$ with $\alpha$-Robin condition
at $\partial_*\cP^{r,\rho}_\theta$ and the Neumann boundary condition at $\partial_\ext\cP^{r,\rho}_\theta$, then
 $E_1(P^{r,\rho}_{\theta,\alpha})\ge -\alpha^2\big( 1+\cO(e^{-\alpha \rho\tan\theta})\big)$ as $\alpha\rho\to+\infty$.
\end{lemma}

\begin{proof}
Denote for shortness $P:=P^{r,\rho}_{\theta,\alpha}$. Let us decompose the polygon $\cP^{r,\rho}_\theta$ as shown in Figure~\ref{figp}(b).
Namely, let $C^\pm$ be orthogonal projection of $B_\rho$
on the segment $[A^\pm_r,B_r]$ and let $U$ be the domain obtained from $\cP^{r,\rho}_\theta$ by taking out the segments $[B_\rho,C^\pm]$, then 
$U$ is the disjoint union of two rectangles $\Pi^\pm:=B_\rho A^\pm_\rho A^\pm_r C^\pm$
and the quadrangle $\Pi^0:=B_\rho C^+B_r C^-$. Let $\Lambda$ be the Laplacian in $U$ with $\alpha$-Robin condition on $\partial_*\cP^{r,\rho}_\theta\subset \partial U$
and Neumann condition at the remaining boundary, then the min-max principle
implies then $E_1(P)\ge E_1(\Lambda)$. Hence, it is sufficient to show the sought lower bound for $E_1(\Lambda)$.

The operator $\Lambda$ is the direct sum $\Lambda_0\boplus\Lambda_+\boplus\Lambda_-$ with
$\Lambda_j$ acting in $L^2(\Pi^j)$. Namely,
$\Lambda_0$ is just the Neumann Laplacian in $\Pi^0$, and, therefore, $E_1(\Lambda_0)=0$.
Furthermore, $\Lambda^\pm$ are Laplacians in the rectangles $\Pi^\pm$ with $\alpha$-Robin condition on the sides $A^\pm_\rho A^\pm_r$ 
and Neumann condition at the remaining boundary. Therefore, they admit a separation of variables
and are both unitary equivalent to $L_N\otimes 1+ 1\otimes T$, where
the operator $L_N$ is the Laplacian on $(0,\rho\tan\theta)$ with $\alpha$-Robin condition at $0$ and Neumann condition
at $\rho\tan\theta$ and $T$ is the Neumann Laplacian on $(0,r-\rho)$. Therefore, $E_1(\Lambda^\pm)=E_1(L_N)+E_1(T)
=E_1(L_N)$.
The operator $L_N$ is covered by Proposition~\ref{prop22} (with $\beta=0$),
and $E_1(L_N)=-\alpha^2\big( 1+\cO(e^{-\alpha \rho\tan\theta})\big)<0$ for $\alpha \rho\to+\infty$.
Therefore, for $\alpha\rho\to+\infty$ one has
\begin{align*}
E_1(P)\ge E_1(\Lambda)=E_1(\Lambda_0\boplus\Lambda_+\boplus\Lambda_-)%=\min_{j\in\{0,+,-\}} E_1(\Lambda^j)\\
&=E_1(\Lambda^+)\\
=E_1(L_N)&=-\alpha^2\big( 1+\cO(e^{-\alpha \rho\tan\theta})\big). \qedhere
\end{align*}
\end{proof}

\begin{lemma}\label{nalph}
For $\alpha r\to+\infty$ there holds
\begin{align*}
E_n(N_{\theta,\alpha}^r)&=\Big[\cE_n(\theta)\,+\cO\Big(\dfrac{1}{(\alpha r)^2}\Big) \Big]\,\alpha^2, \quad n\in\big\{1,\dots,\kappa(\theta)\big\},\\
E_{\kappa(\theta)+1}(N_{\theta,\alpha}^r)&\ge -\alpha^2 +o(\alpha^2).
\end{align*}
\end{lemma}

\begin{proof}
The min-max principle  shows that the eigenvalues of $N_{\theta,\alpha}^r$ are bounded from above by the respective eigenvalues of $D_{\theta,\alpha}^r$.
Hence, the upper bound for $E_n(N_{\theta,\alpha}^r)$ follows from the upper bound for the eigenvalues of $D_{\theta,\alpha}^r$ obtained in Lemma~\ref{dalph} above.

Let us pass to the proof of the lower bound. Let $\chi_0,\chi_1:\RR\to [0,1]$ be smooth functions such that $\chi_0=1$ in $\big(-\infty, \frac12]$, $\chi_0=0$ in $[1,\infty)$
and $\chi_0^2+\chi_1^2=1$. We define $\chi_j^r:\RR^2\to \RR$ by $\chi^r_j(x)=\chi_j\big(|x|/r\big)$, $j=0,1$.
Recall that
\[
N_{\theta,\alpha}^r[u,u]=\int_{\cS^r_\theta} |\nabla u|^2\dd x - \alpha\int_{\partial_*\cS^r_\theta}u^2\dd s,
\quad
\qdom(N_{\theta,\alpha}^r)=H^1(\cS^r_\theta),
\]
and by direct computation  for any $u\in \qdom(N_{\theta,\alpha}^r)$ one has
\begin{multline}
   \label{nn00}
N_{\theta,\alpha}^r[u,u]=N_{\theta,\alpha}^r[\chi^r_0 u,\chi^r_0 u]+N_{\theta,\alpha}^r[\chi^r_1 u,\chi^r_1 u]
-\int_{\cS_\theta^r} \big(|\nabla \chi^r_0|^2+|\nabla \chi^r_1|^2\big)u^2 \dd x\\
\ge N_{\theta,\alpha}^r[\chi^r_0 u,\chi^r_0 u]+N_{\theta,\alpha}^r[\chi^r_1 u,\chi^r_1 u] - \dfrac{a}{r^2}\, \|u\|^2_{L^2(\cS_\theta^r)},
\quad a:= \|\chi_0'\|^2_\infty+\|\chi_1'\|^2_\infty.
\end{multline}
One has $\chi_0^r u\in H^1(\cS_\theta^r)$ and $\chi_0^r u=0$ at $\partial_\ext \cS_\theta^r$.
At the same time, the function $\chi_1^r u$ vanishes inside the disk $|x|\leq\frac{1}{2}\, r$
and, hence, is supported in the quadrangle $\cP_\theta^{r,\rho}$ with $\rho:=\frac{1}{2}\, r \cos\theta$
and belongs to $H^1(\cP_\theta^{r,\rho})$. 
Therefore,   one has $\chi^r_0 u\in \qdom (D_{\theta,\alpha}^r)$ and 
$\chi^r_1 u\in \qdom (P_{\theta,\alpha}^{r,\rho})$, and the inequality \eqref{nn00} rewrites as
$N_{\theta,\alpha}^r[u,u]\ge D_{\theta,\alpha}^r[\chi^r_0 u,\chi^r_0 u]+P_{\theta,\alpha}^{r,\rho}[\chi^r_1 u,\chi^r_1 u]- (a/ r^2) \|u\|^2_{L^2(\cS_\theta^r)}$,
and we recall that $\|u\|^2_{L^2(\cS_\theta^r)}=\|\chi^r_0 u\|^2_{L^2(\cS_\theta^r)}+\|\chi^r_0 u\|^2_{L^2(\cP_\theta^{r,\rho})}$.
By the min-max principle,
\begin{equation}
   \label{enn}
E_n(N_{\theta,\alpha}^r)\ge E_n(D_{\theta,\alpha}^r\boplus P_{\theta,\alpha}^{r,\rho})-a/r^2, \quad
r>0,\ n\in\NN.
\end{equation}
Now let us pick $n\in\big\{1,\dots,\kappa(\theta)\big\}$ and consider the regime $\alpha r\to+\infty$. Then one also has $\alpha \rho\to+\infty$,
and the estimate of Lemma~\ref{lem29} for the first eigenvalue of $P_{\theta,\alpha}^{r,\rho}$ gives $E_1(P_{\theta,\alpha}^{r,\rho})\ge -\alpha^2+o(\alpha^2)$.
On the other hand, the estimate of Lemma~\ref{dalph} for the first eigenvalues of $D_{\theta,\alpha}^r$
shows that $E_n(D_{\theta,\alpha}^R)=\alpha^2 \big(\cE_n(\theta) +\cO(e^{-c\alpha r})\big)$ with some $c>0$, which is below 
$-\alpha^2+o(\alpha^2)$ due to the inequality $\cE_n(\theta)<-1$. Hence,
$E_n(D_{\theta,\alpha}^r\boplus P_{\theta,\alpha}^{r,\rho})=E_n(D_{\theta,\alpha}^r)=\alpha^2 \big(\cE_n(\theta) +\cO(e^{-c\alpha r})\big)$.
Substituting this last estimate into the inequality \eqref{enn} one arrives to
\[
E_n(N_{\theta,\alpha}^r)\ge  \Big[\cE_n(\theta) +\cO(e^{-c\alpha r}) -\dfrac{a}{(\alpha r)^2}\Big]\,\alpha^2
=
\Big[\cE_n(\theta)\,+\cO\Big(\dfrac{1}{(\alpha r)^2}\Big) \Big]\,\alpha^2.
\]
Using \eqref{enn} for $n=\kappa(\theta)+1$ we have
\[
E_{\kappa(\theta)+1}(N_{\theta,\alpha}^r)\ge \min\Big\{E_{\kappa(\theta)+1}(D_{\theta,\alpha}^r),E_1(P_{\theta,\alpha}^{r,\rho})\Big\}-a/r^2.
\]
By assumption one has $1/r=o(\alpha)$. In addition, $E_1(P_{\theta,\alpha}^{r,\rho})\ge -\alpha^2+o(\alpha^2)$, while
$E_{\kappa(\theta)+1} (D_{\theta,\alpha}^r)\ge -\alpha^2$ by Lemma~\ref{dalph}, which concludes the proof.
\end{proof}

Let us give a rough estimate for the first eigenvalue of $R^r_{\theta,\alpha}$ (it will be improved later).

\begin{lemma}\label{rob-trunc}
For some $c>0$ there holds $R^r_{\theta,\alpha}\ge -c\alpha^2$ as $\alpha r\to +\infty$.
\end{lemma}

\begin{proof}
Due to the scaling $\cS^r_\theta=r \cS^1_\theta$ the estimate follows from
Lemma~\ref{rob-scale}.
\end{proof}

\subsection{Eigenfunctions of the Robin-Neumann Laplacians}

We will need an Agmon-type decay estimate for the first $\kappa(\theta)$ eigenfunctions of $N^r_{\theta,\alpha}$, which is established in the following lemma:
\begin{lemma}\label{agmon22a}
There exist  $c>0$ and $C>0$ such that if $n\in\big\{1,\dots,\kappa(\theta)\big\}$ and $\psi^{r,n}_{\theta,\alpha}$
is an eigenfunction of $N^r_{\theta,\alpha}$ for the $n$th eigenvalue, then for $\alpha r\to +\infty$
there holds
\[
\int_{\cS^r_\theta} e^{c\alpha|x|}\Big ( \dfrac{1}{\alpha^2} |\nabla \psi^{r,n}_{\theta,\alpha}|^2+ |\psi^{r,n}_{\theta,\alpha}|^2\Big)\, dx
\le C \,\big\|\psi^{r,n}_{\theta,\alpha}\big\|^2_{L^2(\cS^r_\theta)}.
\]
\end{lemma}

\begin{proof}
Denote for shortness
\[
\text{$N:=N^r_{\theta,\alpha}$, $\psi:=\psi^{r,n}_{\theta,\alpha}$ and $\cE:=\cE_{\kappa(\theta)}(\theta)<-1$.}
\]
For $b>0$ to be chosen later let us consider the function $\phi:\cS^r_\theta\ni x\mapsto b|x|\in \RR$, then $|\nabla\phi|=b$,
and a standard computation gives
\[
N\big[e^{\alpha\phi}\psi,e^{\alpha\phi}\psi]=\int_{\cS^r_\theta} e^{2\alpha\phi} \big(E_n(N) + b^2\alpha^2\big)\psi^2 \dd x.
\]
For $\alpha r\to +\infty$ one has $E_n(N)= \big(\cE_n(\theta)+o(1)\big)\alpha^2$
by Lemma~\ref{nalph}. Therefore, for an arbitrary $\varepsilon>0$ there holds
$E_n(N)\le (\cE+\varepsilon)\alpha^2$, and 
\begin{equation}
  \label{nnn1}
N\big[e^{\alpha\phi}\psi,e^{\alpha\phi}\psi]\le (\cE+b^2+\varepsilon)\alpha^2 \int_{\cS^r_\theta} e^{2\alpha\phi} \psi^2 \dd x.
\end{equation}
On the other hand, let us pick $\eta\in(0,1)$ whose exact value will be chosen later, and set $\rho:=L/\alpha$ with a value $L>0$
to be chosen later as well, then
\begin{align*}
N\big[e^{\alpha\phi}\psi,e^{\alpha\phi}\psi]&\equiv
\int_{\cS^r_\theta} \big|\nabla(e^{\alpha\phi}\psi)\big|^2\dd x-\alpha \int_{\partial_*\cS^r_\theta} e^{2\alpha\phi}\psi^2\dd s\\
&=\eta \int_{\cS^r_\theta} \big|\nabla(e^{\alpha\phi}\psi)\big|^2\dd x\\
&\qquad + (1-\eta)\bigg[ \int_{S^\rho_\theta} \big|\nabla(e^{\alpha\phi}\psi)\big|^2\dd x-\dfrac{\alpha}{1-\eta} \int_{\partial_* \cS^\rho_\theta} e^{2\alpha\phi}\psi^2\dd s\\
&\qquad +\int_{\cS^r_\theta \setminus \overline{\cS^\rho_\theta}} \big|\nabla(e^{\alpha\phi}\psi)\big|^2\dd x-\dfrac{\alpha}{1-\eta} \int_{\partial_* \cS^r_\theta\setminus \partial_* \cS^\rho_\theta} e^{2\alpha\phi}\psi^2\dd s \bigg]\\
&=\eta \int_{\cS^r_\theta} \big|\nabla(e^{\alpha\phi}\psi)\big|^2\dd x\\
&\qquad+ (1-\eta) \Big(N^\rho_{\theta,\frac{\alpha}{1-\eta}}\big[e^{\alpha\phi}\psi,e^{\alpha\phi}\psi]
+ P^{r,\rho}_{\theta,\frac{\alpha}{1-\eta}}\big[e^{\alpha\phi}\psi,e^{\alpha\phi}\psi]\Big)\\
&\ge \eta \int_{\cS^r_\theta} \big|\nabla(e^{\alpha\phi}\psi)\big|^2\dd x + (1-\eta) E_1\big(N^\rho_{\theta,\frac{\alpha}{1-\eta}}\big) \int_{\cS^\rho_\theta} e^{2\alpha\phi}\psi^2\dd x\\
&\qquad+(1-\eta)E_1\big(P^{r,\rho}_{\theta,\frac{\alpha}{1-\eta}}\big) \int_{\cP_{\theta}^{r,\rho}} e^{2\alpha\phi}\psi^2\dd x.
\end{align*}
By applying Lemma~\ref{nalph} for $N^\rho_{\theta,\frac{\alpha}{1-\eta}}$ and Lemma~\ref{lem29} to $P^{r,\rho}_{\theta,\frac{\alpha}{1-\eta}}$
we see that the constant $L$ in the definition of $\rho$  can be chosen sufficiently large to have, for large $\alpha$,
\[
E_1\big(N^{\rho}_{\theta,\frac{\alpha}{1-\eta}}\big)\ge \big(\cE_1(\theta)-\varepsilon\big)\dfrac{\alpha^2}{(1-\eta)^2},
\quad
E_1\big(P^{r, \rho}_{\theta,\frac{\alpha}{1-\eta}}\big)\ge -\dfrac{(1+\varepsilon)\alpha^2}{(1-\eta)^2},
\]
and the substitution into the preceding inequality gives
\begin{align*}
N\big[e^{\alpha\phi}\psi,e^{\alpha\phi}\psi]
&\ge \eta \int_{\cS^r_\theta} \big|\nabla(e^{\alpha\phi}\psi)\big|^2\, \dd x\\
&\quad + \dfrac{\cE_1(\theta)-\varepsilon}{1-\eta}\, \alpha^2 \int_{\cS^\rho_\theta} e^{2\alpha\phi}\psi^2\dd x
-\dfrac{1+\varepsilon}{1-\eta}\, \alpha^2 \int_{\cP_{\theta}^{r,\rho}} e^{2\alpha\phi}\psi^2\dd x.
\end{align*}
Recall that $\cP_{\theta}^{r,\rho}=\cS^r_\theta\setminus \overline{\cS^\rho_\theta}$
and substitute the last inequality into \eqref{nnn1}, then
\begin{multline*}
\eta \int_{\cS^r_\theta} \big|\nabla(e^{\alpha\phi}\psi)\big|^2\dd x + \dfrac{\cE_1(\theta)-\varepsilon}{1-\eta}\, \alpha^2 \int_{\cS^\rho_\theta} e^{2\alpha\phi}\psi^2\dd x\\
{}-\dfrac{1+\varepsilon}{1-\eta}\, \alpha^2 \int_{\cP_{\theta}^{r,\rho}} e^{2\alpha\phi}\psi^2\dd x\\
\le 
(\cE+b^2+\varepsilon)\alpha^2 \int_{\cP_{\theta}^{r,\rho}} e^{2\alpha\phi} \psi^2 \dd x
+
(\cE+b^2+\varepsilon) \alpha^2 \int_{\cS^\rho_\theta} e^{2\alpha\phi} \psi^2 \dd x,
\end{multline*}
which we rewrite as
\begin{gather}
\eta \int_{\cS^r_\theta} \big|\nabla(e^{\alpha\phi}\psi)\big|^2\dd x
+ a_0\alpha^2 \int_{\cS^r_\theta\setminus \overline{\cS^\rho_\theta}} e^{2\alpha\phi} \psi^2 \dd x
\le b_0\alpha^2 \int_{\cS^\rho_\theta} e^{2\alpha\phi} \psi^2 \dd x, \label{preceq00}\\
\begin{aligned}
a_0&:=-\cE-b^2-\varepsilon-\dfrac{1+\varepsilon}{1-\eta}=\dfrac{-\cE - 1 +\big(\eta b^2-b^2 +\eta \cE -2\varepsilon+\varepsilon\eta\big)}{1-\eta},\\
b_0&:=\cE+b^2+\varepsilon -\dfrac{\cE_1(\theta)-\varepsilon}{1-\eta}
=\dfrac{\cE-\cE_1(\theta) -\eta \cE+2\varepsilon -\varepsilon\eta}{1-\eta}+b^2.
\end{aligned} \nonumber
\end{gather}
Due to $\cE_1(\theta)\le\cE<-1$, for any $b>0$ one can choose $\varepsilon>0$ and $\eta>0$ sufficiently small to have $a_0>0$ and $b_0>0$.
For $x\in \cS^\rho_\theta$ one has $|x|< \rho/\cos\theta$, hence,
$\alpha\phi(x)\le \alpha b L/ (\alpha\cos\theta) = bL/\cos\theta$,
and \eqref{preceq00} takes the form
\begin{gather*}
\eta \int_{\cS^r_\theta} \big|\nabla(e^{\alpha\phi}\psi)\big|^2\dd x
+a_0 \alpha^2 \int_{\cS^r_\theta\setminus \overline{\cS^\rho_\theta}} e^{2\alpha\phi} \psi^2 \dd x
\le A\alpha^2 \int_{\cS^\rho_\theta} \psi^2 \dd x, \\
A:=b_0 e^{2bL/\cos\theta},
\end{gather*}
and then
\begin{multline}
   \label{eq-temp4}
\int_{\cS^r_\theta} \big|\nabla(e^{\alpha\phi}\psi)\big|^2\dd x
+2b^2\alpha^2 \int_{\cS^r_\theta} \big|e^{2\alpha\phi}\psi|^2\dd x\\
=\dfrac{1}{\eta}\,\eta \int_{\cS^r_\theta} \big|\nabla(e^{\alpha\phi}\psi)\big|^2\dd x
+ \dfrac{2b^2}{a_0}\, a_0\alpha^2  
\int_{ \cS^r_\theta\setminus \overline{\cS^\rho_\theta} } e^{2\alpha\phi} \psi^2 \dd x
+2b^2\alpha^2\int_{\cS^\rho_\theta} \psi^2 \dd x\\
\le \bigg(\dfrac{1}{\eta}\, A+\dfrac{2b^2}{a_0}\, A+2b^2\bigg)\alpha^2 \int_{\cS^\rho_\theta} \psi^2 dx=:A_0\alpha^2\int_{\cS^\rho_\theta} \psi^2 \dd x\le
A_0\alpha^2\|\psi\|^2_{L^2(\cS^r_\theta)}.
\end{multline}
Using $2xy\le x^2+y^2$ and $ xy\le \frac{1}{4}\, x^2+y^2$ for $x,y\in\RR$, we estimate
\begin{align*}
\big|\nabla(e^{\alpha\phi}\psi)\big|^2&\ge |e^{\alpha\phi}\nabla \psi|^2+b^2\alpha^2|e^{\alpha\phi} \psi|^2 -  2\big|e^{\alpha\phi}\nabla \psi\big| \ \big|b\alpha e^{\alpha\phi}\psi\big|\\
&\ge \dfrac{1}{2}|e^{\alpha\phi}\nabla \psi|^2 - b^2\alpha^2|e^{\alpha\phi} \psi|^2.
\end{align*}
The substitution into \eqref{eq-temp4} gives
\[
\int_{\cS^r_\theta} e^{2b\alpha|x|} \Big( \dfrac{1}{2}\big|\nabla \psi|^2
+b^2\alpha^2 \psi^2\Big)\, dx\le A_0\alpha^2\|\psi\|^2_{L^2(\S^r_\theta)},
\]
and one arrives to the claim by taking $c:=2b$ and $C:=A_0(2+1/b^2)$.
\end{proof}

\subsection{Non-resonant sectors}\label{ssec-nonres}

Recall that the Robin-Neumann Laplacians $N^r_{\theta,\alpha}$ in the truncated convex sectors $\cS^r_\theta$
are defined in \eqref{eq-tqr}, and that due to the asymptotics of Lemma~\ref{nalph}
their first $\kappa(\theta)$ eigenvalues are, in a sense, close to the first $\kappa(\theta)$ eigenvalues
of the Robin Laplacian $T_{\theta,\alpha}$ in the associated infinite sectors $\cS_\theta$
in the regime $\alpha r\to +\infty$. For the subsequent study we will use the notion of a non-resonant angle,
which involves a hypothesis on the behavior of the next eigenvalue of $N^r_{\theta,\alpha}$
in the same asymptotic regime. Namely, we will use the following definition:

\begin{defi}\label{defnonres}
A half-angle $\theta\in\big(0,\frac{\pi}{2}\big)$ is called \emph{non-resonant} if there is $C>0$ such that
\[
		E_{\kappa(\theta)+1}(N^r_{\theta,\alpha})\ge -\alpha^2+ C/r^2 \text{ as $\alpha>0$ is fixed and $r$ is large.}
\]
By the scaling~\eqref{eq-scale1}, the property only depends on $\theta$ and can be equivalently reformulated
as
\[
  	E_{\kappa(\theta)+1}(N^r_{\theta,\alpha})\ge -\alpha^2+  C/r^2 \text{ as $\alpha r$ is large.}
\]
\end{defi}
We show in Proposition~\ref{prop-good} that the non-resonance property is satisfied by an explicit wide range of half-angles, which is a key point for the whole analysis  (we remark that
there exist half-angles which \emph{do not} satisfy the non-resonance property as it will be seen 
 in Subsection~\ref{rem-triangle}). 
%The above condition is strictly adapted to our needs in the proofs of the main results. In particular, the
%precise shape of the cut (two line segments perpendicular to the boundary of the sector) is of a crucial importance. In Subsection~\ref{ssec-nonres2}
%we discuss other approaches appearing in similar problems.

\begin{prop}\label{prop-good}
All half-angles $\theta$ with $\frac{\pi}{4}\le\theta<\frac{\pi}{2}$ are non-resonant.
\end{prop}

\begin{proof}
We prove the result first for $\theta=\pi/4$ (Step 1) by rather direct computations, and then use a kind of monotonicity to extend it to other half-angles in the range indicated (Step 2).
Without loss of generality we set $\alpha=1$ and remove the dependence on $\alpha$ from the notation and write $N^r_{\theta}$ instead of $N^r_{\theta,1}$.

\emph{Step 1: $\theta=\pi/4$.} By Proposition~\ref{prop-sector} we have $\kappa(\pi/4)=1$, so
we need to prove that there exists a constant $C>0$ satisfying
\begin{equation}
\label{nonr1}
E_2(N^r_{\frac{\pi}{4}}) \geq -1 +C/r^2 \text{ as } r \text{ is large}.
\end{equation}
Remark that $\cS^r_{\frac{\pi}{4}}$ is simply a square of side length $r$, and $N^r_{\frac{\pi}{4}}$ is the Laplacian
with $1$-Robin boundary condition on two neighboring sides and Neumann condition on the other two sides.
Hence, on can separate the variables: using the one-dimensional Laplacians $L_N$ on $(0,r)$ with $1$-Robin
condition at $0$ and Neumann  condition at $r$ one has $N^r_{\frac{\pi}{4}}=L_N\otimes 1 + 1\otimes L_N$, and 
then $E_2(N^r_{\frac{\pi}{4}}) = E_1(L_N)+E_2(L_N)$. Using Proposition~\ref{prop22} one has
$E_2(N^r_{\frac{\pi}{4}})\ge -1 + 1/r^2+\cO(e^{-r})$ as $r\to +\infty$,
which gives the sought inequality \eqref{nonr1}. Hence, the claim is proved for $\theta=\pi/4$.

\begin{figure}
   \begin{minipage}[c]{.46\linewidth}
   \centering
      \includegraphics[height=40mm]{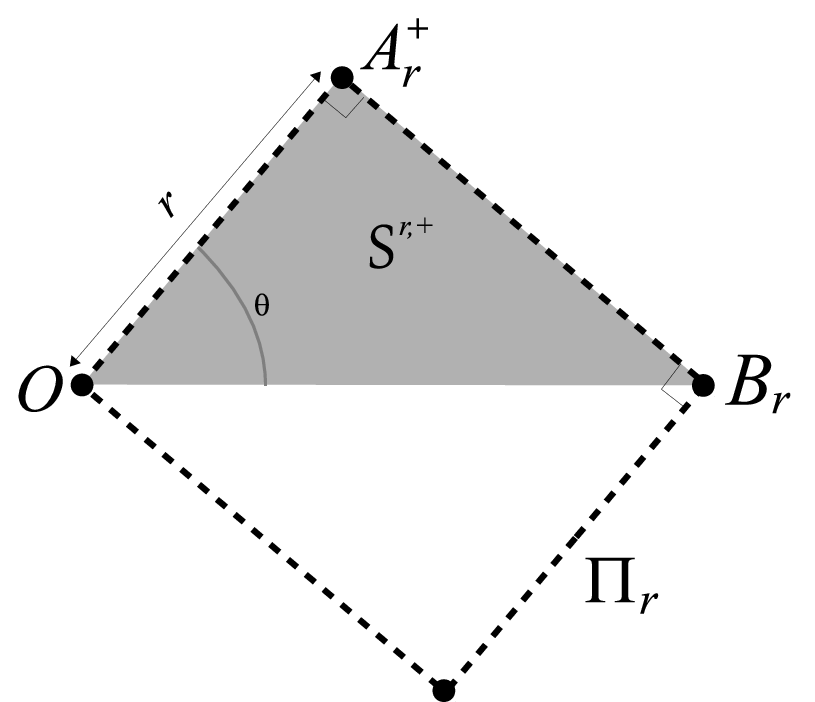}\\(a)
   \end{minipage} \hfill
   \begin{minipage}[c]{.46\linewidth}
   \centering
      \includegraphics[height=40mm]{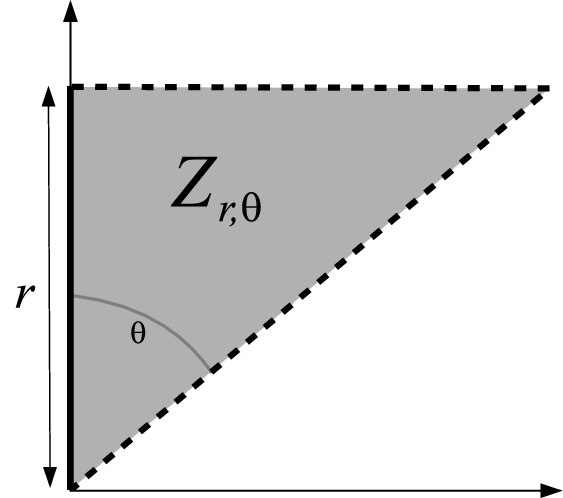}\\(b)
   \end{minipage}

\caption{Constructions for the proof of Proposition~\ref{prop-good}. 
(a) The completion of the triangle $\cS_\theta^{r,+}$ (shaded) to a rectangle $\Pi_r$ (surrounded by the dashed line).
(b) The triangle $Z_{r,\theta}$ is a rotated copy of $\cS_\theta^{r,+}$.
The solid/dashed lines correspond to Robin/Neumann boundary conditions.
\label{fig-zrt}}
\end{figure}

\emph{Step 2: extension to $\theta \in[\frac{\pi}{4},\frac{\pi}{2})$.}
We still have $\kappa(\theta)=1$ by Proposition~\ref{prop-sector}, hence,
we need to show that there exists $C>0$ such that 
\begin{equation}
\label{nonr2}
E_2(N^r_{\theta}) \geq -1 +C/r^2 \text{ as } r \text{ is large}.
\end{equation}
Using the symmetry with respect to the axis $Ox_1$ one easily sees that $N^r_{\theta}$ is unitarily equivalent
to $T^{r,D}_{\theta}\oplus T^{r,N}_{\theta}$, where $T^{r,D/N}_{\theta}$ stand for the Laplacians
in
\[
\mathcal S^{r,+}_{\theta}:=\mathcal S^r_\theta \cap \big\{ (x_1,x_2): \, x_2>0\big\}=\text{triangle } OA^+_rB_r
\]
with $1$-Robin condition at $OA^+_r$, Neumann condition at $A^+_rB_r$
and the Dirichlet/Neumann boundary condition at $OB_r$ (we refer to Figures~\ref{figure1} and~\ref{fig-zrt}(a) for an illustration).
Let us study first the Dirichlet part $T^{r,D}_{\theta}$. Let $\Pi_r$ be the rectangle constructed
on the vectors $OA^+_r$ and $A^+_rB_r$, see Figure~\ref{fig-zrt}(a), then $S^{r,+}_{\theta}\subset \Pi_r$.
Using the standard Dirichlet bracketing we obtain $E_n(T^{r,D}_{\theta})\ge E_n(Q_r)$ for any $n\in\NN$,
where $Q_r$ is the Laplacian in $\Pi_r$ with $1$-Robin condition  at $OA^+_r$, Neumann condition at $A^+_rB_r$
and the Dirichlet boundary condition at the remaining part of the boundary. Remark that $|A^+_r B_r|=r\tan\theta$,
and the operator $Q_r$ admits then a separation of variables and is unitarily equivalent to $L_D\otimes 1 + 1 \otimes D_r$, where $D_r$ is the Laplacian on $(0,r)$
with the Dirichlet boundary condition at $0$ and the Neumann boundary condition at $r$, and $L_D$ is the one-dimensional
Laplacian
on the interval $(0,r\tan\theta)$ with $1$-Robin condition at $0$ and Dirichlet condition on the other end.
Therefore, $E_1(T^{r,D}_{\theta})=E_1(L_D)+E_1(D_r)=E_1(L_D)+\pi^2/(4r^2)$.
Due to Proposition~\ref{prop21} we have $E_1(L_D)=-1+\cO(e^{-r\tan\theta})$, therefore,
$E_1(T^{r,D}_{\theta})\ge -1 + C_D/r^2$ for large $r$ with any fixed $C_D\in(0,\pi^2/4)$.
Therefore, the sought estimate \eqref{nonr2} becomes equivalent to the existence of $C_N>0$ for which there holds
\begin{equation}
   \label{nonr3}
E_2(T^{r,N}_{\theta})\ge -1 + C_N/r^2 \text{ as } r\to+\infty,
\end{equation}
which we already know to hold for $\theta=\frac{\pi}{4}$.
In order to study $T^{r,N}_{\theta}$ we apply a rotation bringing the triangle $\cS^{r,+}_\theta$ onto the triangle
$Z_{r,\theta}:=\big\{ (x_1,x_2): \  0<x_1\cotan\theta<x_2<r \big\}$,
so that $T^{r,N}_{\theta}$ becomes unitary equivalent to the Laplacian $Q_{r,\theta}$ in $L^2(Z_{r,\theta})$ with $1$-Robin condition
along the axis $Ox_2$ and Neumann condition at the remaining boundary, and $E_n(T^{r,N}_{\theta})=E_n(Q_{r,\theta})$ for any $n\in\NN$,
and one easily sees that 
\[
Q_{r,\theta}[u,u]=\int_{Z_{r,\theta}} |\nabla u|^2  \dd x - \int_0^r u(0,x_2)^2 \dd x_2, \quad \qdom(Q_{r,\theta})=H^1(Z_{r,\theta}),
\]
see Figure~\ref{fig-zrt}(b) for an illustration.
Using the unitary transform
\[
V: L^2(Z_{r\tan\theta,\frac{\pi}{4}})\to L^2(Z_{r,\theta}),
\quad
(V u)(x_1,x_2)=\sqrt{\tan\theta}\,u(x_1,x_2\tan\theta),
\]
which satisfies $V \big(H^1(Z_{r\tan\theta,\frac{\pi}{4}})\big)=H^1(Z_{r,\theta})$,
we obtain, with $u_j:=\partial u/\partial x_j$,
\begin{align*}
Q_{R,\theta}[Vu,Vu] &=\tan \theta \int_{Z_{R,\theta}} \Big( u_1 (x_1,x_2\tan\theta)^2 + \tan^2\theta \,u_2 (x_1,x_2\tan\theta)^2 \Big)\, \dd x\\
&\qquad -\tan\theta \int_0^r u(0,x_2\tan\theta)^2 \dd x_2\\
&=\int_{Z_{r\tan\theta,\frac{\pi}{4}}} \Big( u_1 (x_1,x_2)^2 + \tan^2\theta \,u_2 (x_1,x_2)^2 \Big)\dd x\\
&\qquad - \alpha \int_0^{r\tan\theta} u(0,x_2)^2 \dd x_2\\
&=Q_{r\tan\theta,\frac{\pi}{4}}[u,u]+ (\tan^2\theta -1) \int_{Z_{r\tan\theta,\frac{\pi}{4}}} u_2^2 \dd x.
\end{align*}
For $\theta\in \big[\frac{\pi}{4},\frac{\pi}{2}\big)$ we have $\tan\theta\ge 1$, hence, $Q_{r,\theta}[Vu,Vu]\ge Q_{r\tan\theta,\frac{\pi}{4}}[u,u]$ for all  $u \in H^1(Z_{r,\theta})$,
and by the min-max principle we have
\[
E_n(T^{r,N}_\theta)=E_n(Q_{r,\theta})\ge E_n(Q_{r\tan\theta,\frac{\pi}{4}})=E_n(T^{r\tan\theta,N}_\frac{\pi}{4}).
\]
It was already shown in Step~1 that for some $C>0$ we have $E_2(T^{r\tan\theta,N}_\frac{\pi}{4})\ge -1 + C/(r\tan \theta)^2$
for large $r$, so the substitution into the preceding inequality gives the sought estimate \eqref{nonr3} with $C_N=C \cotan^2\theta$.
\end{proof}

Now we state some consequences of the non-resonance condition, which will provide important components for the subsequent
asymptotic analysis:

\begin{lemma}\label{lem-trace0}
Assume that $\theta$ is non-resonant and denote by $\cL$
the subspace spanned by the eigenfunctions corresponding to the first $\kappa(\theta)$ eigenvalues of $N^r_{\theta,\alpha}$.
Then there exists $b>0$ such that for $\alpha r\to +\infty$ there holds
\begin{align}
   %\label{nnn000}
	\label{trr1a}
\|v\|^2_{L^2(\cS^r_\theta)}&\le b\,r^2\big( N^r_{\theta,\alpha}[v,v] + \alpha^2\|v\|^2_{L^2(\cS^r_\theta)} \big), &
 v\in H^1(\cS^r_\theta)\mathop{\cap} \cL^\perp,\\
	\label{trr2a}
\int_{\partial_\ext \cS^r_\theta} v^2\dd s
&\le b\alpha r^2\big( N^r_{\theta,\alpha}[v,v] + \alpha^2\|v\|^2_{L^2(\cS^r_\theta)}\big),
&  v\in H^1(\cS^r_\theta)\mathop{\cap} \cL^\perp.
\end{align}
\end{lemma}

\begin{proof}
The norm estimate \eqref{trr1a} directly follows from the definition of a non-resonant half-angle (Definition~\ref{defnonres})
with the help of the spectral theorem. For \eqref{trr2a}, recall that by Lemma~\ref{rob-scale}
one can find $c_0>0$ such that $E_1(R^r_{\theta,\alpha})\ge-c_0\alpha^2$ for large $\alpha r$.
Due to
\[
\cQ(N^r_{\theta,\alpha})=\cQ(R^r_{\theta,\alpha})=H^1(\cS^r_\theta),
\quad
R^r_{\theta,\alpha}[v,v]=N^r_{\theta,\alpha}[v,v]-\alpha\int_{\partial_\ext \cS^r_\theta}v^2\dd s,
\]
the preceding inequality for $E_1(R^r_{\theta,\alpha})$ takes the form
\[
 \int_{\partial_\ext \cS^r_\theta}v^2\dd s\le \dfrac{1}{\alpha} N^r_{\theta,\alpha}[v,v]+ c_0 \alpha\|v\|^2_{L^2(\cS^r_\theta)} \text{ for all } v\in H^1(\cS^r_\theta).
\]
Assume in addition that $v\perp \cL$, then one bounds from above the second term on the right-hand side using~\eqref{trr1a}, which gives
\begin{align*}
 \int_{\partial_\ext \cS^r_\theta}v^2\dd s&\le \dfrac{1}{\alpha} N^r_{\theta,\alpha}[v,v]+ c_0 b\,\alpha r^2\big( N^r_{\theta,\alpha}[v,v] + \alpha^2\|v\|^2_{L^2(\cS^r_\theta)} \big)\\
&=\Big(\dfrac{1}{\alpha}+c_0 b\,\alpha r^2\Big)N^r_{\theta,\alpha}[v,v] + c_0 b\,\alpha r^2 \alpha^2\|v\|^2_{L^2(\cS^r_\theta)}\\
&\le\Big(\dfrac{1}{\alpha}+c_0 b\,\alpha r^2\Big)\Big(N^r_{\theta,\alpha}[v,v]+\alpha^2\|v\|^2_{L^2(\cS^r_\theta)}\Big).
\end{align*}
It remains to estimate, for $\alpha r\to +\infty$,
\[
\dfrac{1}{\alpha}+c_0 b\,\alpha r^2=\dfrac{1+c_0 b\,(\alpha r)^2}{\alpha}\le \dfrac{2c_0 b\,(\alpha r)^2}{\alpha}=2c_0 b \alpha r^2.\qedhere
\]
\end{proof}

\section{Robin eigenvalues in polygons: Proof of Theorem~\ref{thm-poly}}\label{sec-thm-poly}

\subsection{Decomposition of a polygon}

In this section we assume that $\Omega$ is a polygon with straight edges. We assume that $\Omega$ has $M$ vertices
$A_1,\dots, A_M$, and for the notation convenience we identify  $A_0\equiv A_M$ and $A_{M+1}\equiv A_1$,
and the same cyclic numbering convention will be applied to other related objects.
We denote by $\ell_j$ the length of the side $\Gamma_j:=[A_j,A_{j+1}]$, $j=1,\dots, M$, and
introduce the maps 
\[
\gamma_j:[0,\ell_j]\ni t \mapsto A_j + \dfrac{A_{j+1}-A_j}{\ell_j}\,t \in \RR^2
\]
providing an arc-length parametrization of $\Gamma_j$ with $\gamma_j(0)=A_j$ and $\gamma_j(\ell_j)=A_{j+1}$.
In addition, for $t\in(0,\ell_j)$ by $\nu_j(t)$ we denote the \emph{outer} unit normal to $\partial\Omega$
at the point $\gamma_j(t)$ of $\Gamma_j$.

\begin{figure}

\centering

\includegraphics[width=50mm]{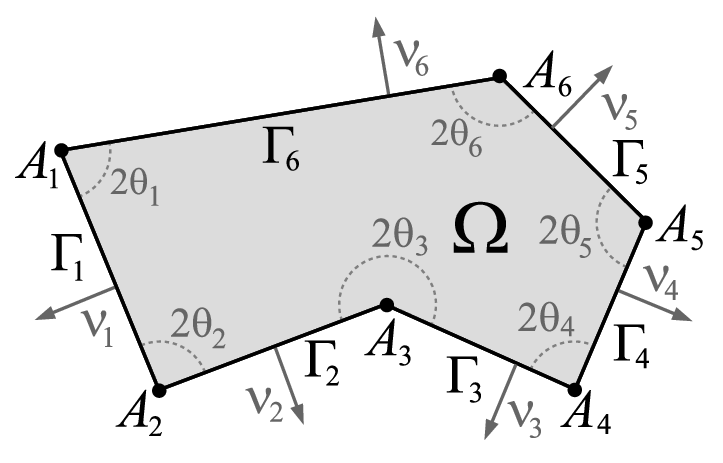}

\caption{An example of a polygon $\Omega$ with six vertices. The vertex $A_3$ is concave, the other vertices are convex, and
$\cJ_\cvx=\{1,2,4,5,6\}$.\label{fig-polygon}}

\end{figure}

By $\theta_j\in[0,\pi]$ we denote the half-angle of $\Omega$ at the vertex $A_j$, i.e. $\theta_j$
is the half of the angle between $\Gamma_{j-1}$ and $\Gamma_j$ when measured inside $\Omega$.
Our assumption is that there are neither zero angles nor artificial vertices, i.e. that
$\theta_j\notin \big\{0, \tfrac{\pi}{2},\pi\big\}$ for all $j=1,\dots,M$.
One says that a vertex $A_j$ is \emph{convex} if $\theta_j<\frac{\pi}{2}$,
otherwise it will be called \emph{concave}. We denote
\[
\cJ_\cvx:=\{j:\, A_j\text{ is convex}\}.
\]
We refer to Figure~\ref{fig-polygon} for an illustration.

%
%\begin{figure}
%\centering
%\begin{tabular}{cc}
%\begin{minipage}[c]{75mm}
%\begin{center}
%\includegraphics[height=40mm]{conv-vois.eps}
%\end{center}
%\end{minipage}
%&
%\begin{minipage}[c]{75mm}
%\begin{center}
%\includegraphics[height=40mm]{conc-vois.eps}
%\end{center}
%\end{minipage}\\
%(a)  & (b) 
%\end{tabular}
%\caption{The construction of the neighborhoods $V_{j,\delta}$: (a) convex vertex, (b) concave vertex.
%The partial boundary $\partial_* V_{j,\delta}$
%is shown with the thick solid line, the part $\partial_\ext V_{j,\delta}$ is indicated with the thick dashed line,
%and the part $\partial_\out V_{j,\delta}$ with the gray dotted line.\label{conv-vois}}
%\end{figure}

\begin{figure}[b]
\centering
\begin{tabular}{cc}
\begin{minipage}[c]{52mm}
\begin{center}
\includegraphics[height=30mm]{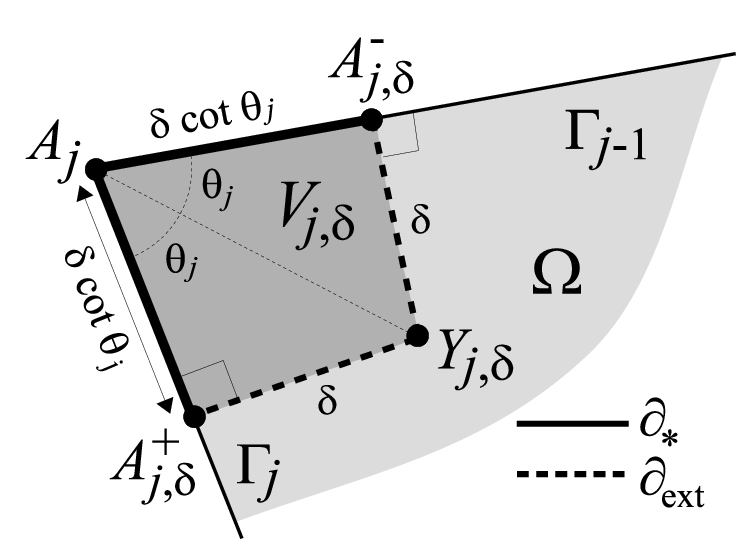}
\end{center}
\end{minipage}
&
\begin{minipage}[c]{52mm}
\begin{center}
\includegraphics[height=30mm]{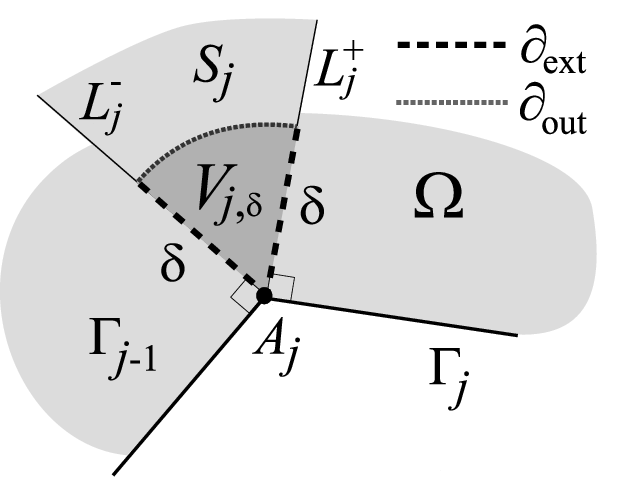}
\end{center}
\end{minipage}\\
(a)  & (b) 
\end{tabular}
\caption{The construction of the neighborhoods $V_{j,\delta}$: (a) convex vertex, (b) concave vertex.
The partial boundary $\partial_* V_{j,\delta}$
is shown with the thick solid line, the part $\partial_\ext V_{j,\delta}$ is indicated with the thick dashed line,
and the part $\partial_\out V_{j,\delta}$ with the gray dotted line.\label{fig8}}
\end{figure}

For small $\delta>0$ denote
\[
\Omega_\delta=\big\{x\in \Omega:\, \dist(x,\partial\Omega)<\delta\big\},
\quad \Omega_\delta^c:=\Omega\setminus \overline{\Omega_\delta},
\]
and $\Omega_\delta$ will be further decomposed near each vertex. The construction is different for convex and concave vertices.
\begin{itemize}
\item Let $A_j$ be a convex vertex, then there exists a unique point $Y_{j,\delta}\in \Omega$
such that $\dist(Y_{j,\delta},\Gamma_{j-1})=\dist(Y_{j,\delta},\Gamma_j)=\delta$.
Denote $\lambda_j:=\cot \theta_j$, then the points
\[
A_{j,\delta}^-:=\gamma_{j-1}(\ell_j-\lambda_j \delta), \quad
A_{j,\delta}^+:=\gamma_j(\lambda_j \delta)
\]
are exactly the orthogonal projections of $Y_{j,\delta}$ on $\Gamma_{j-1}$ and $\Gamma_j$, respectively.
We denote the interior of the quadrangle $A_j A_{j,\delta}^- Y_{j,\delta} A_{j,\delta}^+$ by $V_{j,\delta}$,
and, in turn, we decompose the boundary of $V_{j,\delta}$  into the following
parts:
\[
\partial_* V_{j,\delta}:=\partial V_{j,\delta}\cap\partial\Omega, \quad
\partial_\ext V_{j,\delta}:=\partial V_{j,\delta}\setminus \partial_* V_{j,\delta},
\quad
\partial_\out V_{j,\delta}:=\emptyset.
\]
\item Let $A_j$ be a concave vertex. Let $L^-_j$ be the half-line emanating from $A_j$ which is orthogonal to $\Gamma_{j-1}$ and directed inside $\Omega$.
By $L^+_j$ we denote the half-line emanating from $A_j$, orthogonal to $\Gamma_j$ at $A_j$ and directed inside $\Omega$.
Denote by $S_j$ the infinite sector bounded by $L^-_j$ and $L^+_j$ which lies inside $\Omega$ near $A_j$. Then we set
\[
V_{j,\delta}:=S_j\cap B(A_j,\delta), \quad \lambda_j:=0,
\]
where $B(a,r)$ is the disk of radius $r$ centered at $a$. We decompose the boundary of $V_{j,\delta}$ as follows:
\[
\partial_* V_{j,\delta}:=\emptyset,
\quad
\partial_\out V_{j,\delta}:=\partial V_{j,\delta}\cap \partial \Omega^c_\delta,
\quad
\partial_\ext V_{j,\delta}:=\partial V_{j,\delta}\setminus \partial_\out V_{j,\delta}.
\]
\end{itemize}
Remark that the numbers $\lambda_j$ represent a kind of ``length deficiency'':
the length of $\partial_*V_{j,\delta}$ is equal to $2\lambda_j \delta$ for both convex and concave vertices.

The set $\cW_\delta:=\Omega_\delta\setminus\overline{\bigcup_{j=1}^M V_{j,\delta}}$ 
is then the union of $M$ disjoint thin rectangles. Namely, denote
\begin{equation}
\label{rectangles}
\begin{aligned}
I_{j,\delta}:=(\lambda_j \delta,\ell_j-\lambda_{j+1}\delta), \quad \Pi_{j,\delta}:=I_{j,\delta}\times(0,\delta),\\
W_{j,\delta}:=\Phi_j(\Pi_{j,\delta}), \quad \Phi_j(s,t):=\gamma_j(s)-t\nu_j(s),
\end{aligned}
\end{equation}
then $\cW_\delta=\bigcup_{j=1}^M W_{j,\delta}$, and $W_{j,\delta}\cap W_{k,\delta}=\emptyset$ for $j\ne k$.
We decompose the boundary of each rectangle $W_{j,\delta}$ as follows:
\begin{gather*}
\partial_* W_{j,\delta}:=\partial W_{j,\delta} \cap \partial\Omega,
\quad
\partial_\out W_{j,\delta}:=\partial W_{j,\delta}\cap \partial \Omega^c_\delta,\\
\partial_\ext W_{j,\delta}:=\partial W_{j,\delta}\setminus\Big(
\partial_* W_{j,\delta}\cup \partial_\out W_{j,\delta}
\Big).
\end{gather*}
By construction we have the equality
\begin{equation}
    \label{eq-vjwj}
\bigcup\nolimits_{j=1}^M \partial_\ext V_{j,\delta} = \bigcup\nolimits_{j=1}^M \partial_\ext W_{j,\delta},
\end{equation}
which will be of importance later. We refer to Fig.~\ref{fig9-decomp} for an illustration of the above decomposition.

\begin{figure}

\centering

\includegraphics[width=100mm]{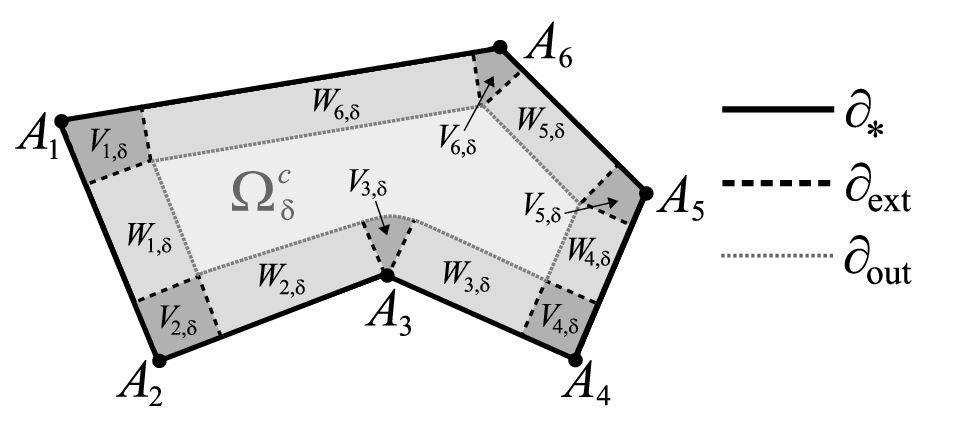}

\caption{Decomposition of a polygon.\label{fig9-decomp}}

\end{figure}

\subsection{First estimates for side-induced eigenvalues}

With each $j\in\{1,\dots,M\}$ we associate the corresponding number $\kappa(\theta_j)$
of discrete eigenvalues of the Robin Laplacians in the infinite sector of aperture $2\theta_j$ (see Section~\ref{sec-sectors})
and set
\begin{align*}
K&:=\kappa(\theta_1)+\dots +\kappa(\theta_M)\equiv \sum\nolimits_{j\in\cJ_\cvx} \kappa(\theta_j),\\
\cE&:=\text{ the disjoint union of } \big\{ \cE_n(\theta_j), \ n=1,\dots,\kappa(\theta_j)\big\} \text{  for } j\in\cJ_\cvx,\\
\cE_n&:=\text{ the $n$th element of $\cE$ when numbered in the non-decreasing order,}
\end{align*}
(see Subsection~\ref{sec-sectors} for a detailed notation).
For what follows we introduce several operators:
\begin{align*}
N^V_j&:=\begin{minipage}[t]{90mm}\raggedright 
the Laplacian in $V_{j,\delta}$ with the $\alpha$-Robin boundary condition at $\partial_* V_{j,\delta}$
and the Neumann boundary condition at the rest of the boundary,\end{minipage}\\
D^V_j&:=\begin{minipage}[t]{90mm}\raggedright 
the Laplacian in $V_{j,\delta}$ with the $\alpha$-Robin boundary condition at $\partial_* V_{j,\delta}$
and the Dirichlet boundary condition at the rest of the boundary.\end{minipage}
\end{align*}
We remark that for concave vertices $A_j$, the respective operators $(N/D)^V_j$
are just the Neumann/Dirichlet Laplacians in $V_{j,\delta}$ due to $\partial_* V_{j,\delta}=\emptyset$.
Furthermore, denote
\begin{align*}
N^W_j&:=\begin{minipage}[t]{90mm}\raggedright 
the Laplacian in $W_{j,\delta}$ with the $\alpha$-Robin boundary condition at $\partial_* W_{j,\delta}$
and the Neumann boundary condition at the rest of the boundary,\end{minipage}\\
D^W_j&:=\begin{minipage}[t]{90mm}\raggedright 
the Laplacian in $W_{j,\delta}$ with the $\alpha$-Robin boundary condition at $\partial_* W_{j,\delta}$
and the Dirichlet boundary condition at the rest of the boundary.\end{minipage}
\end{align*}
Finally, introduce
\[
N^c:=
\begin{minipage}[t]{90mm}\raggedright 
the Neumann Laplacian in $\Omega^c_\delta$.\end{minipage}
\]
One easily sees that  $D^V_j$  and $N^V_j$ with $j\in\cJ_\cvx$ are covered by the analysis of Section~\ref{sec-trunc},
and the behavior of the first $\kappa(\theta_j)$ eigenvalues for $\alpha\delta\to+\infty$ is given in Lemmas~\ref{dalph} and~\ref{nalph}, respectively.
On the other hand, for $j\notin\cJ_\cvx$ one has $D^V_j\ge0$, $N^V_j\ge 0$, and $N^c\ge 0$.
For the rest of the section we assume that 
\begin{equation}
    \label{rescond0}
\text{all convex vertices are non-resonant,}
\end{equation}
then, in addition, we have a lower bound for the $\big(\kappa(\theta_j)+1)$-th eigenvalue
of each $N^V_j$ with $j\in\cJ_\cvx$ due to Definition~\ref{defnonres}.

In the subsequent constructions we choose $\delta>0$ depending on $\alpha$ in such a way that
\begin{equation}
  \label{eq-daa}
\delta\to 0^+, \quad \alpha\delta\to+\infty \text{ for } \alpha\to+\infty,
\end{equation}
and summarize the preceding observations as follows:
\begin{lemma}\label{lem-prec0}
With some some $c>0$ and $c_0>0$ there holds, for $\alpha\to +\infty$,
\begin{align*}
E_n\big(\boplus\nolimits_{j=1}^M N^V_j \big)&=\cE_n\,\alpha^2+\cO(1/\delta^2), & \text{for }n &=1,\dots, K,\\
E_n\big(\boplus\nolimits_{j=1}^M D^V_j \big)&=\cE_n\,\alpha^2+\cO(\alpha^2 e^{-c\alpha\delta}), & \text{for } n &=1,\dots, K,\\
E_{K+1}\big(\boplus\nolimits_{j=1}^M D^V_j \big)&\ge E_{K+1}\big(\boplus\nolimits_{j=1}^M N^V_j \big)\ge-\alpha^2+c_0/\delta^2.
\end{align*}
\end{lemma}

In what follows we are going to use several one-dimensional operators. We denote
\begin{align*}
D_j&:=\text{ the Dirichlet Laplacian on $(0,\ell_j)$}, \\
D_{j,\delta}&:=\text{ the Dirichlet Laplacian on $I_{j,\delta}$},\\
N_j&:=\text{ the Neumann Laplacian on $(0,\ell_j)$}, \\
N_{j,\delta}&:=\text{ the Neumann Laplacian on $I_{j,\delta}$}, 
\end{align*}
and remark that for each fixed $n\in\NN$ one has $E_n(D_{j,\delta})=E_n(D_j)+\cO(\delta)$ and $E_n(N_{j,\delta})=E_n(N_j)+\cO(\delta)$.
We start with a simple estimate for the eigenvalues of $R^\Omega_\alpha$:

\begin{prop}\label{DtN}
%For any $n\in\{1,\dots,K\}$ there holds, with some $c>0$,
%\[
%E_n(R^\Omega_\alpha)=\cE_n\,\alpha^2+\cO(e^{-c\alpha\delta}+1/\delta^2).
%\]
There holds, with some $c>0$,
\begin{equation}
  \label{eq-ap01}
E_n(R^\Omega_\alpha)=\cE_n\,\alpha^2+\cO(1/\delta^2+\alpha^2e^{-c\alpha\delta}), \quad n\in\{1,\dots,K\}.
\end{equation}
In addition, for any $n\in\NN$ there holds
\begin{multline}
   \label{eq-ap02}
-\alpha^2 +E_n\big(\boplus\nolimits_{j=1}^M N_j\big) + \cO(\delta+\alpha^2 e^{-\alpha\delta}) \le E_{K+n}(R^\Omega_\alpha)\\
\le -\alpha^2+E_n\big(\boplus\nolimits_{j=1}^M D_j\big) + \cO(\delta+\alpha^2e^{-\alpha\delta}).
\end{multline}
\end{prop}

\begin{proof}
Due to the standard Dirichlet-Neumann bracketing, for any $n\in \NN$ one has
\begin{equation}
  \label{braket1a}
\begin{aligned}
E_n \big( N^c \oplus (\boplus\nolimits_{j=1}^M N^V_j ) \oplus (\boplus\nolimits_{j=1}^M N^W_j)\big)
&\le
E_n(R^\Omega_\alpha)\\
&\le E_n\big(\boplus\nolimits_{j=1}^M D^V_j \oplus (\boplus\nolimits_{j=1}^M D^W_j) \big).
\end{aligned}
\end{equation}
The operators $N^W_j$ and $D^W_j$ admits a separation of variables: if one denotes $L_{N/D}$ the Laplacian on $(0,\delta)$
with $\alpha$-Robin condition at $0$ and Neumann/Dirichlet condition at $\delta$, then
one has the unitary equivalences $N^W_j\simeq N_{j,\delta}\otimes 1 + 1\otimes L_N$ and $D^W_j\simeq D_{j,\delta}\otimes 1 + 1\otimes L_D$.
For each fixed $n\in\NN$ one has $E_n(N_{j,\delta})=\cO(1)$ and $E_n(D_{j,\delta})=\cO(1)$. On the other hand, by
Propositions~\ref{prop21} and~\ref{prop22} we have $E_1(L_{N/D})=-\alpha^2+ \cO(\alpha^2e^{-\alpha\delta})$ and $E_2(L_{N/D})\ge 0$.
Therefore, $E_n(N^W_j)=E_1(L_N)+E_n(N_{j,\delta})$ and $E_n(D^W_j)=E_1(L_D)+E_n(D_{j,\delta})$,
and then
\begin{align*}
E_n\big(\boplus\nolimits_{j=1}^M N^W_j \big)&= E_1(L_N) + E_n \big(\boplus\nolimits_{j=1}^M N_{j,\delta} \big)\\
&=-\alpha^2+ E_n \big(\boplus\nolimits_{j=1}^M N_j \big) + \cO(\delta+\alpha^2e^{-\alpha\delta}),\\
E_n\big(\boplus\nolimits_{j=1}^M D^W_j \big)&= E_1(L_D) + E_n \big(\boplus\nolimits_{j=1}^M D_{j,\delta} \big)\\
&=-\alpha^2+ E_n \big(\boplus\nolimits_{j=1}^M D_j \big) + \cO(\delta+\alpha^2e^{-\alpha\delta}).
\end{align*}
In view of the estimates of Lemma~\ref{lem-prec0} one has then
\begin{align*}
E_K\Big(N^c \oplus \big(\boplus\nolimits_{j=1}^M N^V_j \big)\Big)
&\le E_n\big(\boplus\nolimits_{j=1}^M N^W_j \big)
\le 
E_{K+1}\Big(N^c \oplus \big(\boplus\nolimits_{j=1}^M N^V_j \big)\Big), \\
E_K\Big( \big(\boplus\nolimits_{j=1}^M D^V_j \big)\Big)
&\le E_n\big(\boplus\nolimits_{j=1}^M D^W_j \big)
\le 
E_{K+1}\Big(\big(\boplus\nolimits_{j=1}^M D^V_j \big)\Big).
\end{align*}
Therefore, for each $n\in\{1,\dots,K\}$ one has
\begin{align*}
E_n \big( N^c \oplus (\boplus\nolimits_{j=1}^M N^V_j) \oplus (\boplus\nolimits_{j=1}^M N^W_j)\big)&=E_n \big(\boplus\nolimits_{j=1}^M N^V_j\big)\\
&=\cE_n\,\alpha^2+\cO(1/\delta^2),\\
E_n\big((\boplus\nolimits_{j=1}^M D^V_j) \oplus (\boplus\nolimits_{j=1}^M D^W_j) \big) &=E_n\big(\boplus\nolimits_{j=1}^M D^V_j\big)\\
&=\cE_n\,\alpha^2+\cO(\alpha^2e^{-c\alpha\delta}),
\end{align*}
and \eqref{braket1a} reads as $\cE_n\,\alpha^2+\cO(1/\delta^2)\le E_n(R^\Omega_\alpha)\le \cE_n\,\alpha^2+\cO(e^{-c\alpha\delta})$ and gives \eqref{eq-ap01}.
In order to obtain \eqref{eq-ap02} we remark that for each $n\in\NN$ one has
\begin{align*}
E_{K+n} \Big( N^c \oplus \big(\boplus_{j=1}^M N^V_j \big) \oplus &\big(\boplus_{j=1}^M N^W_j \big)\Big)=
E_n\big(\boplus_{j=1}^M N^W_j \big)\\
&=-\alpha^2+ E_n \big(\boplus_{j=1}^M N_j \big) + \cO(\delta+\alpha^2e^{-\alpha\delta}),\\
E_{K+n}\Big(\boplus_{j=1}^M D^V_j \oplus &\big(\boplus_{j=1}^M D^W_j \big) \Big)=E_n\big(\boplus_{j=1}^M D^W_j \big)\\
&=-\alpha^2+ E_n \big(\boplus_{j=1}^M D_j \big) + \cO(\delta+\alpha^2e^{-\alpha\delta}),
\end{align*}
and it remains to use these bounds on the both sides of \eqref{braket1a}.
\end{proof}

By taking $\delta:=b\log\alpha /\alpha$ with a sufficiently large $b$ one then obtains:
\begin{coro}\label{corol-11}
There holds
\begin{equation}
   \label{eq-ap11}
E_n(R^\Omega_\alpha)=\cE_n\,\alpha^2+o(\alpha^2) \text{ for $n\in\{1,\dots,K\}$.}
\end{equation}
In addition, for any $n\in\NN$ there holds
\begin{equation}
   \label{eq-ap12}
\begin{aligned}
E_{K+n}(R^\Omega_\alpha)&= -\alpha^2+\cO(1), \\
E_{K+n}(R^\Omega_\alpha)&\le -\alpha^2+E_n\big(\boplus\nolimits_{j=1}^M D_j\big) + \cO\Big(\tfrac{\log\alpha}{\alpha}\Big).
\end{aligned}
\end{equation}
\end{coro}

Remark that \eqref{eq-ap11} is only given for completeness (and as a preparation for the analysis of the curvilinear case): the remainder is not optimal and can be improved to $\cO(e^{-c\alpha})$ with a suitable $c>0$ by using more advanced methods as shown by Khalile~\cite{khalile2}.

\subsection{Lower bound for side-induced eigenvalues}

It remains to obtain a more precise lower bound for $E_{K+n}(R^\Omega_\alpha)$.
This  is the most involved part of the whole analysis, and it will be done in the present subsection with the help of the Proposition~\ref{prop6}
by constructing a suitable identification map.
All estimates of this subsection are for $\alpha$ and $\delta$ in the asymptotic regime~\eqref{eq-daa}.
Introduce some additional objects:
\begin{align*}
L&:=\begin{minipage}[t]{90mm}
the subspace of $L^2(\Omega)$ spanned by the first $K$ eigenfunctions  of $R^\Omega_\alpha$,
\end{minipage}\\
L_j&:=\begin{minipage}[t]{90mm}
the subspace of $L^2(V_{j,\delta})$ spanned by the first $\kappa(\theta_j)$ eigenfunctions of
$N^V_j$, with $j\in\cJ_\cvx$,
\end{minipage}\\
\sigma_j&:L^2(\Omega)\to L^2(V_{j,\delta}) \text{ the operator of restriction,}\\
&  \quad \text{$(\sigma_j u)(x)=u(x)$ for $x\in V_{j,\delta}$,}
\end{align*}
then the adjoint operators  $\sigma^*_j:L^2(V_{j,\delta})\to L^2(\Omega)$ are the operators of extension by zero.
Recall that the distance  $d(E,F)$ between subspaces $E$ and $F$ was discussed in Subsection~\ref{sub-dist}.
\begin{lemma}\label{dist-ll0}
For  $j\in\cJ_\cvx$ one has  $d(\sigma^*_j L_j, L)=\cO(e^{-c\alpha\delta})$ with some fixed $c>0$.
\end{lemma}

\begin{proof}
During the proof we denote $\Lambda_j:=\sigma^*_j L_j\subset L^2(\Omega)$, and for $v\in L^2(V_{j,\delta})$
we denote $v_*:=\sigma_j^*v\in L^2(\Omega)$.

Let $0<a<b<1$. Consider a $C^\infty$ function $\varphi:\RR\to [0,1]$ with $\varphi(t)=1$ for $t\le a$ and $\varphi(t)=0$ for $t\ge b$.
Introduce $\varphi_\delta:\Omega\to \RR$ by $\varphi_\delta(x)= \varphi\big( |x-A_j|/(\delta\cot \theta_j)\big)$,
which clearly satisfies (as $\alpha$ is large, hence, $\delta$ is small):
\begin{itemize}
\item $0\le \varphi_\delta\le 1$, and for all $\beta\in\NN^2$ with $1\le |\beta|\le 2$
there holds $\|\partial^\beta \varphi_\delta\|_\infty\le C \delta^{-|\beta|}$,
\item $\varphi_\delta=1$ in $V_{j,a\delta}$, and $\varphi_\delta=0$ in $\Omega\setminus \overline{V_{j,b\delta}}$,
\item the normal derivative of $\varphi_\delta$ at $\partial\Omega$ is zero,
\end{itemize}
where $C>0$ is some fixed constant.
Denote
\[
\varphi_\delta \Lambda_j:=\big\{ \varphi_\delta v_*:\, v_*\in \Lambda_j\big\}\subset L^2(\Omega),
\]
then 
\begin{equation}
  \label{dist01}
d(\Lambda_j,L)\le d(\Lambda_j,\varphi_\delta \Lambda_j)+d(\varphi_\delta \Lambda_j,L).
\end{equation}
The first term on the right-hand side can be easily estimated by applying directly the definition of
the distance. Namely, due to the Agmon-type estimate for the first $\kappa(\theta_j)$
eigenfunctions of $N^V_j$ (Lemma~\ref{agmon22a}), with some $b_0>0$ and $B>0$ there holds
\[
\int_{V_{j,\delta}} e^{b_0\alpha|x-A_j|}
\Big(\dfrac{1}{\alpha^2} |\nabla v|^2 + v^2\Big) dx\le B\|v\|^2_{L^2(V_{j,\delta})},
\quad
v\in L_j.
\]
Writing
\begin{multline*}
\int_{V_{j,\delta}\setminus \overline{V_{j,a\delta}}} \Big( \dfrac{1}{\alpha^2} |\nabla v|^2 + v^2\Big)\dd x\\
=\int_{V_{j,\delta}\setminus \overline{V_{j,a\delta}}}e^{b_0\alpha|x-A_j|}\cdot e^{ b_0\alpha|x-A_j|}\Big( \dfrac{1}{\alpha^2} |\nabla v|^2 + v^2_*\Big)\dd x
\end{multline*}
we obtain the following upper bound
\begin{multline*}
 \int_{V_{j,\delta}\setminus \overline{V_{j,a\delta}}} \Big( \dfrac{1}{\alpha^2} |\nabla v|^2 + v^2\Big)\dd x\\
\le e^{-b_0\alpha a  \delta} \int_{V_{j,\delta}\setminus \overline{V_{j,a\delta}}} e^{ b_0\alpha|x-A_j|}\Big( \dfrac{1}{\alpha^2} |\nabla v|^2 + v^2_*\Big)\dd x\\
\le e^{-b_0 \alpha a \delta} \int_{V_{j,\delta}} e^{ {b_0}\alpha|x-A_j|}\Big( \dfrac{1}{\alpha^2} |\nabla v|^2 + v^2_*\Big)\dd x.
\end{multline*}
This finally gives
\begin{align}
 \label{eq-est123}
  \int_{V_{j,\delta}\setminus \overline{V_{j,a\delta}}} \Big( \dfrac{1}{\alpha^2} |\nabla v|^2 + v^2\Big)\dd x
  \leq
  Be^{-2c\alpha\delta} \|v\|^2_{L^2(V_{j,\delta})}, \quad c:={b_0}a/2.
 \end{align}
Therefore, for any $v_*\in \Lambda_j$ we have
\begin{multline*}
\big\| v_*-\varphi_\delta v_*\big\|^2_{L^2(\Omega)}= \int_{\Omega} (1-\varphi_\delta)^2 v_*^2\, dx
\le \int_{\Omega\setminus \overline{V_{j,a\delta}}} v_*^2\, dx\\
\equiv \int_{V_{j,\delta}\setminus \overline{V_{j,a\delta}}} v^2\, dx\le 
Be^{-2c\alpha\delta} \|v\|^2_{L^2(V_{j,\delta})}\equiv Be^{-2c\alpha\delta} \|v_*\|^2_{L^2(\Omega)}.
\end{multline*}
Denote by $P_j$ the orthogonal projector on  $\varphi_\delta \Lambda_j$ in $L^2(\Omega)$,
then for any $u\in L^2(\Omega)$ we have by definition $\|u-P_j u\|=\inf_{\phi\in \varphi_\delta \Lambda_j} \|u -\phi\|$.
Therefore, for any non-zero $v_*\in\Lambda_j$ we have
\begin{gather}
\dfrac{\|v_*-P_jv_*\|}{\|v_*\|}\le \dfrac{\|v_*-\varphi_\delta v_*\|}{\|v_*\|}
\equiv \dfrac{\big\|(1-\varphi_\delta)v_*\big\|}{\|v_*\|}\le \sqrt{B}\, e^{-c\alpha \delta}, \nonumber\\
d(\Lambda_j,\varphi_\delta \Lambda_j)=\sup_{v_*\in \Lambda_j,\, v_*\ne 0} \dfrac{\|v_*-P_j v_*\|}{\|v_*\|}
\le \sqrt{B}\, e^{-c\alpha \delta}. \label{dist02}
\end{gather}

Now we need an estimate for the second term on the right-hand side of \eqref{dist01}, which will be obtained
with the help of Proposition~\ref{propdist2}. Namely, let $v^n$ with $n\in\big\{1,\dots,\kappa(\theta_j)\big\}$
be the eigenfunctions of $N^V_j$ for the eigenvalues $E_n:=E_n(N^V_j)$ forming an orthonormal basis of $L_j$,
then, in particular,
\[
-\Delta v^n=E_n v^n \text{ in } V_{j,\delta}, \quad \frac{\partial v^n}{\partial\nu}=\alpha v^n \text{ at } \partial_*V_{j,\delta}\subset \partial\Omega,
\]
where $\nu$ the outer unit normal.
Consider the functions $\psi_n:=\varphi_\delta v^n_*$, then using the above properties
of $\varphi_\delta$ we have
\begin{gather*}
\Delta \psi_n
=\big((\Delta \varphi_\delta) v^n +2\nabla \varphi_\delta\cdot \nabla v^n +\varphi_\delta \Delta v^n\big)_* \in L^2(\Omega),\\
\dfrac{\partial\psi_n}{\partial \nu}= \dfrac{\partial \varphi_\delta}{\partial \nu}\, v^n
+ \varphi_\delta\dfrac{\partial v^n}{\partial \nu}
=\varphi_\delta\dfrac{\partial v^n}{\partial \nu}=\alpha \varphi_\delta v^n=\alpha \psi_n \text{ on } \partial\Omega,
 \end{gather*}
which shows that $\psi_n$ belong to the domain of $R^\Omega_\alpha$.
We represent now
\[
(R^\Omega_\alpha-E_n)\psi_n =(-\Delta -E_n)\psi_n= \big(-(\Delta \varphi_\delta) v^n -2\nabla \varphi_\delta\cdot \nabla v^n\big)_*
\]
and note that the supports of $\nabla\varphi_\delta$ and $\Delta\varphi_\delta$ are contained
in $V_{j,b\delta}\setminus \overline{V_{j,a\delta}}$. Therefore, with the help of \eqref{eq-est123}
we can estimate
\begin{align*}
\int_\Omega \big| (\Delta \varphi_\delta) v_*^n\big|^2\dd x
& \le \dfrac{C^2}{\delta^4}
\int_{V_{j,b\delta}\setminus \overline{V_{j,a\delta}}} (v^n)^2\dd x\\
&\le \dfrac{BC^2}{\delta^4}e^{-2c\alpha\delta} \|v^n\|^2_{L^2(V_{j,\delta})}\equiv \dfrac{BC^2}{\delta^4}e^{-2c\alpha\delta},\\
\int_\Omega \big| \nabla \varphi_\delta\cdot \nabla v_*^n\big|^2\dd x
& \le \int_\Omega |\nabla \varphi_\delta|^2 |\nabla v_*^n|^2\dd x
\le \dfrac{C^2}{\delta^2}
\int_{V_{j,b\delta}\setminus \overline{V_{j,a\delta}}} (\nabla v^n)^2\dd x\\
& \le\dfrac{BC^2\alpha^2}{\delta^2}e^{-2c\alpha\delta} \|v^n\|^2_{L^2(V_{j,\delta})}\equiv \dfrac{BC^2\alpha^2}{\delta^2}e^{-2c\alpha\delta},
\end{align*} 
and by noting that $1/\delta^2=o(\alpha/\delta)$ we have
\[
\big\| (R^\Omega_\alpha-E_n)\psi_n\big\|_{L^2(\Omega)}
=\cO \big( (\alpha/\delta)\, e^{-c\alpha\delta}\big).
\]

Let us estimate the Gram matrix $G$ of $(\psi_n)$.
We have, using Cauchy-Schwarz and \eqref{eq-est123},
\begin{multline*}
\Big| \langle \psi_k,\psi_n\rangle_{L^2(\Omega)}-\langle v^k_*,v^n_*\rangle_{L^2(\Omega)}\Big|\\
=
\Big|\int_\Omega (\varphi_\delta^2-1) v^k_* v^n_*\dd x\Big|
\le \int_{V_{j,\delta}\setminus \overline{V_{j,a\delta}}} |v^k\, v^n|\,\dd x\\
\le \dfrac{1}{2}\Big(
\int_{V_{j,\delta}\setminus \overline{V_{j,a\delta}}} (v^k)^2\,\dd x
+
\int_{V_{j,\delta}\setminus \overline{V_{j,a\delta}}} (v^n)^2\,\dd x
\Big)\le Be^{-2c\alpha\delta}.
\end{multline*}
Therefore, we have $\langle \psi_k,\psi_n\rangle_{L^2(\Omega)}=\delta_{k,n}+\cO(e^{-2c\alpha\delta})$,
and the lowest eigenvalue $\lambda$ of $G$ is estimated as $\lambda=1+\cO(e^{-2c\alpha\delta})$.

Finally let $h:=(-\cE_K -1)/2$, then the interval $I:=\big((\cE_1-h)\alpha^2, (\cE_K+h)\alpha^2\big)$
contains all the above  eigenvalues $E_n$ due to Lemma~\ref{lem-prec0},
and it also contains the first $K$ eigenvalues of $R^\Omega_\alpha$
and satisfies $\dist \big( I, \spec (R^\Omega_\alpha)\setminus I\big)\ge \frac{1}{4}\,h \alpha^2$ by \eqref{eq-ap11}.
Therefore, we are exactly in the situation of Proposition~\ref{propdist2} with the parameters
\[
E=\varphi_\delta \Lambda_j, \quad
F=L, \quad
\varepsilon=\cO \big( \tfrac{\alpha}{\delta}\, e^{-c\alpha\delta}\big),
\quad \eta\ge \tfrac{1}{8}h \alpha^2, \quad
\lambda=1+\cO(e^{-2c\alpha\delta}),
\]
which gives $d(\varphi_\delta \Lambda_j,L)= \cO \big(e^{-c\alpha\delta}/(\alpha\delta)\big)$.
By combining this last inequality with \eqref{dist02} in the initial triangular inequality \eqref{dist01}
one arrives at the conclusion.
\end{proof}

\begin{lemma}\label{vertices1}
There exist $b>0$ and $c>0$ such that for each $j=1,\dots, M$ there holds
\begin{align}
\|\sigma_j u\|^2_{L^2(V_{j,\delta})}
&\le b\delta^2 \Big(N^V_j[\sigma_j u,\sigma_j u]+ \alpha^2 \|\sigma_j u\|^2_{L^2(V_{j,\delta})}\Big)\nonumber\\
&\qquad+ b\alpha^2\delta^2e^{-c\alpha\delta} \|u\|^2_{L^2(\Omega)},\label{ineq2}\\
\int_{\partial_\ext V_{j,\delta}} (\sigma_j u)^2\dd s
&\le b \alpha\delta^2 \Big(N^V_j[\sigma_j u,\sigma_j u]+ \alpha^2 \|\sigma_j u\|^2_{L^2(V_{j,\delta})}\Big)\nonumber\\
&\qquad + b\alpha^3\delta^2 e^{-c\alpha\delta} \|u\|^2_{L^2(\Omega)} \label{ineq1}
\end{align}
as $u\in H^1(\Omega)$ with $u\perp L$.
\end{lemma}

\begin{proof}
Remark that the sought inequalities  look quite similar to those in Lemma~\ref{lem-trace0}. The novelty
is that we do not assume $\sigma_j u\perp L_j$ (in this case the result would follow directly)
but just $u\perp L$. The main technical ingredient of the proof below is to show that the orthogonal
projection of $\sigma_j u$ onto $L_j$ is sufficiently small and absorbed by the last summands in the above inequalities~\eqref{ineq2} and~\eqref{ineq1}.
This will be achieved using the distance estimate of Lemma~\ref{dist-ll0}.

Assume first that $j\in\cJ_\cvx$. Let $P$ be the orthogonal projector on $L$ in $L^2(\Omega)$
and $P_j$ be the orthogonal projector on $L_j$ in $L^2(V_{j,\delta})$. Consider the following functions of $L^2(V_{j,\delta})$:
\[
u^V:=\sigma_j u,
\quad
v_0:=P_j u^V, \quad v:=(1-P_j) u^V.
\]
Due to $u\perp L$ we have $u=(1-P)u$, hence,
\begin{align*}
\|v_0\|_{L^2(V_{j,\delta})}=\|\sigma_j^*v_0\|_{L^2(\Omega)}&=\big\| \sigma_j^*P_j\sigma_j(1-P) u\big\|_{L^2(\Omega)}\\
&\le \big\| \sigma_j^*P_j\sigma_j(1-P)\big\| \, \|u\|_{L^2(\Omega)}.
\end{align*}
The operator $\Pi_j:=\sigma_j^*P_j\sigma_j$ is exactly the orthogonal projector
on $\sigma_j^*L_j$ in $L^2(\Omega)$, and by Lemma~\ref{dist-ll0} one 
has $\big\| \sigma_j^*P_j\sigma_j(1-P)\big\|=\|\Pi_j- \Pi_j P\|=d(\sigma_j^*L_j,L)=\cO(e^{-c\alpha\delta})$
with some $c>0$. Then for some $b>0$ one has 
\begin{equation}
    \label{vnorm}
\|v_0\|_{L^2(V_{j,\delta})}\le be^{-c\alpha\delta}\|u\|_{L^2(\Omega)}.
\end{equation}
As $P_j$ is a spectral projector for  $N^V_j$, one has
$N^V_j[u^V,u^V]=N^V_j[v_0,v_0]+N^V_j[v,v]$, and due to the spectral theorem we have the inequalities
\[
E_{1}( N^V_j)\|v_0\|^2_{L^2(V_{j,\delta})}\le N^V_j[v_0,v_0]
\le
 E_{\kappa(\theta_j)}( N^V_j)\|v_0\|^2_{L^2(V_{j,\delta})}.
\]
By Lemma~\ref{nalph} we have $E_n( N^V_j)=\cO(\alpha^2)$ for $n=1,\dots,\kappa(\theta_j)$ and using~\eqref{vnorm} one arrives at
\begin{equation}
 \label{nseq2}
\begin{aligned}
\Big| N^V_j[v_0,v_0] \Big| &\le a_0 \alpha^2 e^{-2c\alpha\delta}\|u\|^2_{L^2(\Omega)},\\
N^V_j[v,v] &\le N^V_j[u^V,u^V] + a_0 \alpha^2 e^{-2c\alpha\delta}\|u\|^2_{L^2(\Omega)}.
\end{aligned}
\end{equation}
As $v\perp L_j$, one can apply the trace and norm estimate for non-resonant truncated sectors (Lemma~\ref{lem-trace0}).
Using first the norm estimate one has, with some $c_1>0$,
\[
\|v\|^2_{L^2(V_{j,\delta})}\le c_1\delta^2\Big(
N^V_j[v,v]+\alpha^2\|v\|^2_{L^2(V_{j,\delta})}
\Big),
\]
and by using \eqref{vnorm}, \eqref{nseq2} and the inequality $\lVert v\rVert^2_{L^2(V_{j,\delta})} \leq \lVert u^V\rVert^2_{L^2(V_{j,\delta})}$ we have
\begin{align*}
\| u^V\|^2_{L^2(V_{j,\delta})}&=\|v\|^2_{L^2(V_{j,\delta})}+\|v_0\|^2_{L^2(V_{j,\delta})}\\
&\le  c_1\delta^2\Big(
N^V_j[v,v]+\alpha^2\|v\|^2_{L^2(V_{j,\delta})}\Big) + b^2e^{-2c\alpha\delta}\|u\|^2_{L^2(\Omega)}\\
&\le c_1\delta^2\Big(
N^V_j[u^V,u^V]+\alpha^2\|u^V\|^2_{L^2(V_{j,\delta})}\Big)\\
&\qquad + (a_0c_1\alpha^2\delta^2e^{-2c\alpha\delta}+b^2e^{-2c\alpha\delta})\|u\|^2_{L^2(\Omega)}\\
&\le c_1\delta^2\Big(
N^V_j[u^V,u^V]+\alpha^2\|u^V\|^2_{L^2(V_{j,\delta})}\Big)\\
&\qquad + b_0\alpha^2\delta^2e^{-2c\alpha\delta}\|u\|^2_{L^2(\Omega)}
\end{align*}
with a sufficiently large $b_0>0$, which proves \eqref{ineq2}.
Furthermore, using first the trace estimate of Lemma~\ref{lem-trace0}
and then \eqref{nseq2} we have, with some $c_2>0$,
\begin{equation}
    \label{eq-trace2-1}
\begin{aligned}
\int_{\partial_{\ext} V_{j,\delta}} v^2\dd s&\le c_1 \alpha\delta^2 \Big(
N^V_j[v,v]+\alpha^2\|v\|^2_{L^2(V_{j,\delta})}
\Big)\\
&\le 
c_1 \alpha\delta^2 \Big(
N^V_j[u^V,u^V]+\alpha^2\|u^V\|^2_{L^2(V_{j,\delta})}
\Big)\\
&\quad + c_2 \alpha^3\delta^2 e^{-2c\alpha\delta}\|u\|^2_{L^2(\Omega)},
\end{aligned}
\end{equation}
with $c_2 := c_1 a_0$. Let $R^V_j$ be the Laplacian in $V_{j,\delta}$
with $\alpha$-Robin condition at the whole boundary, then $E_1(R^V_j)\ge-c_0\alpha^2$ with some $c_0>0$ (see Lemma~\ref{rob-scale}),
i.e.
\begin{gather*}
R^V_j[f,f]\equiv N^V_j[f,f]-\alpha\int_{\partial_\ext V_{j,\delta}} f^2\dd s
\ge -c_0\alpha^2 \|f\|^2_{L^2(V_{j,\delta})},\\
\int_{\partial_\ext V_{j,\delta}}f^2 \dd s\le \dfrac{1}{\alpha}\Big(N^V_j[f,f]+ c_0 \alpha^2\|f\|^2_{L^2(V_{j,\delta})}\Big)
 \text{ for all } f\in H^1(V_{j,\delta}).
\end{gather*}
Using this inequality for $f:=v_0$ and then applying~\eqref{vnorm} and \eqref{nseq2} on both terms on the right-hand side
we arrive at
\begin{equation}
  \label{eq-trace2-2}
\int_{\partial_\ext V_{j,\delta}}v_0^2 \dd s\le c_3\alpha e^{-2c\alpha\delta}\|u\|^2_{L^2(\Omega)}
\end{equation}
with some $c_3>0$. Finally,
\[
\int_{\partial_\ext V_{j,\delta}} (u^V)^2\dd s\equiv 
\int_{\partial_\ext V_{j,\delta}} (v+v_0)^2\dd s\le 2\int_{\partial_\ext V_{j,\delta}} v^2\dd s
+2\int_{\partial_\ext V_{j,\delta}} v^2_0\dd s,
\]
and by estimating the two terms on the right-hand side by \eqref{eq-trace2-1} and \eqref{eq-trace2-2}
one arrives at \eqref{ineq1}.

Now assume that $j\notin\cJ_\cvx$, then $N^V_j\ge 0$ is just the Neumann Laplacian in $V_{j,\delta}$. In particular, for large $\alpha$ one has the obvious estimate
\[
\|\sigma_j u \|^2_{L^2(V_{j,\delta})}\le \dfrac{1}{\alpha^2} \Big( N^V_j[\sigma_j u ,\sigma_j u] + \alpha^2 \|\sigma_j u\|^2_{L^2(V_{j,\delta})}\Big),
\]
implying \eqref{ineq2} due to $1/\alpha^2=o(\delta^2)$. To obtain \eqref{ineq1} consider the Laplacian
$R^V_j$ in $V_{j,\delta}$ with the $\alpha$-Robin boundary condition at the whole boundary, then
$R^V_j\ge -c_4\alpha^2$ with some $c_4>0$ by Lemma~\ref{rob-scale}, and
for all $f \in H^1(V_{j,\delta})$ we have
\begin{align*}
\int_{\partial V_{j,\delta}} f^2\dd s&\le \dfrac{1}{\alpha}\Big( \int_{V_{j,\delta}} |\nabla f|^2\dd x
+c_4\alpha^2  \int_{V_{j,\delta}} f^2\dd x\Big )\\
&\equiv
\dfrac{1}{\alpha} \Big(N^V_j[f,f]+c_4\alpha^2\|f\|^2_{L^2(V_{j,\delta})}\Big).
\end{align*}
The terms on the right-hand side are non-negative, so for $c_5:=\max\{1,c_4\}$ one has
\[
\int_{\partial V_{j,\delta}} f^2\dd s\le
\dfrac{c_5}{\alpha} \Big(N^V_j[f,f]+\alpha^2\|f\|^2_{L^2(V_{j,\delta})}\Big) \text{ for all } f \in H^1(V_{j,\delta}).
\]
Using this estimate for $f:=\sigma_j u$ and noting that $1/\alpha=o(\alpha\delta^2)$ we arrive at \eqref{ineq1}.
Remark that the last summands in \eqref{ineq2} and \eqref{ineq1} appear for convex vertices only.
\end{proof}

Now we have collected all technical ingredients for the proof of the main estimate:
\begin{prop}\label{prop-low-rob1}
For any fixed $n\in\NN$ one has $E_{K+n}(R^\Omega_\alpha)\ge -\alpha^2 + E_n\big(\boplus\nolimits_{j=1}^M D_j\big) + \cO \big( \tfrac{\log\alpha}{\sqrt{\alpha}}\big)$
as $\alpha\to+\infty$.
\end{prop}

\begin{proof}
%During the proof we choose $\delta$ satisfying \eqref{eq-daa}.
%The estimate will be obtained with the help of an identification operator (see Proposition~\ref{prop6}).
Consider the Hilbert spaces
\[
\cH:=\text{ the orthogonal complement of $L$ in $L^2(\Omega)$},
\quad
\cH':=\boplus_{j=1}^M L^2(I_{j,\delta}).
\]
During the proof for $u\in \cH$ we denote $\lVert u \rVert\coloneqq \lVert u \rVert_{L^2(\Omega)}$ and 
\begin{align*}
v_j&:=\text{ the restriction of $u$ to $V_{j,\delta}$},& \|v_j\|&:=\|v_j\|_{L^2(V_{j,\delta})},\\
w_j&:=\text{ the restriction of $u$ to $W_{j,\delta}$},& \|w_j\|&:=\|w_j\|_{L^2(W_{j,\delta})},\\
u_c&:=\text{ the restriction of $u$ to $\Omega^c_\delta$},& \|u_c\|&:=\|u_c\|_{L^2(\Omega^c_\delta)},
\end{align*}
and remark that due to the preceding constructions and the equality~\eqref{eq-vjwj} we have
\begin{equation}
  \label{eq-uvw}
\sum_{j=1}^M \int_{\partial_\ext V_{j,\delta}} v_j^2\dd s=\sum_{j=1}^M \int_{\partial_\ext W_{j,\delta}} w_j^2\dd s.
\end{equation}
Applying Lemma~\ref{vertices1} we obtain, with some $b>0$ and $c>0$,
the inequalities
\begin{align*}
\|v_j\|^2
&\le b\delta^2 \Big(N^V_j[v_j,v_j]+ \alpha^2 \|v_j\|^2\Big) + b\alpha^2\delta^2e^{-c\alpha\delta} \|u\|^2,\\
\int_{\partial_\ext V_{j,\delta}} v_j^2 \dd s
&\le b \alpha\delta^2 \Big(N^V_j[v_j,v_j]+ \alpha^2 \|v_j\|^2\Big)
+ b\alpha^3\delta^2 e^{-c\alpha\delta} \|u\|^2.
\end{align*}

Now recall that each $N^W_j$ admit a separation of variables, $N^W_j\simeq L_N\otimes 1 + 1 \otimes N_{j,\delta}$, where
$L_N$ is the Laplacian on $(0,\delta)$ with $\alpha$-Robin condition at $0$ and Neumann condition at $\delta$.
Denote by $\psi$ a normalized eigenfunction for the first eigenvalue of $L_N$, consider the maps
\begin{equation}
   \label{eq-pj}
P_j:\cH\to L^2(I_{j,\delta}), \quad
(P_j u)(s):=\int_0^\delta \psi(t) \,w_j\big(\Phi_j(s,t)\big)\dd t,
\end{equation}
with $\Phi_j$ defined in \eqref{rectangles}, and denote $\Tilde w_j:=w_j\circ \Phi_j$ and $z_j:= \Tilde w_j - (P_j u) \otimes \psi$, then one has the pointwise orthogonality
\[
\int_0^\delta \psi (t) z_j(\cdot, t)\dd t=0.
\]
Using the standard change of variables and then the spectral theorem for $L_N$ one obtains
\begin{multline*}
N^W[w_j,w_j]= \int_{W_{j,\delta}} |\nabla w_j|^2\dd x -\alpha \int_{\partial_* W_{j,\delta}} w^2\dd s\\
\begin{aligned}
&=\int_{I_{j,\delta}} \int_0^\delta \Big( \big(\partial_s \Tilde w_j \big)^2+\big(\partial_t \Tilde w_j \big)^2\Big)\dd t \dd s -\alpha\int_{I_{j,\delta}} {\Tilde w_j(s,0)^2}\dd s\\
&=\int_{I_{j,\delta}} \big((P_j u)'(s)\big)^2\dd s + \int_{I_{j,\delta}} \int_0^\delta \big(\partial_s z_j \big)^2\dd t \dd s\\
&\qquad+ \int_{I_{j,\delta}} \bigg( \int_0^\delta \big(\partial_t \Tilde w_j \big)^2\dd t  -\alpha {\Tilde w_j(s,0)^2}\bigg)\dd s\\
&\ge\int_{I_{j,\delta}} \big((P_j u)'(s)\big)^2\dd s + \int_{I_{j,\delta}} \int_0^\delta \big(\partial_s z_j \big)^2\dd t \dd s\\
&\qquad+E_1(L_N) \|P_j u\|^2 +E_2(L_N) \|z_j\|^2.
\end{aligned}
\end{multline*}
Using Proposition~\ref{prop22} we estimate $E_1(L)\ge -\alpha^2-b_1\alpha^2e^{-c\alpha\delta}$ and $E_2(L_N)\ge 0$, which leads to
\[
N^W_j[w_j,w_j]\ge 
-\alpha^2\|P_j u\|^2+\big\|(P_j u)'\big\|^2-b_1 \alpha^2e^{-c\alpha\delta}\|P_j u\|^2.
\]
Now let us set $\delta:=(c'\log\alpha)/\alpha$ with $c'\ge 3/c$, then the conditions \eqref{eq-daa} for the choice
of $\delta$ are satisfied, and $\alpha^2e^{-c\alpha\delta}=o(\delta)$, which implies  $\alpha^2\delta^2e^{-c\alpha\delta}=o(\delta^3)$ and  $\alpha^3\delta^2 e^{-c\alpha\delta}=o(\alpha\delta^3)$.
This simplifies the remainders in the above inequalities, and
one can pick a sufficiently large $a>0$ ,
\begin{align}
   \label{nv1}
\|v_j\|^2
&\le \dfrac{a\log^2\alpha}{\alpha^2} \Big(N^V_j[v_j,v_j]+ \alpha^2 \|v_j\|^2\Big) + \dfrac{a\log^3\alpha}{\alpha^3} \|u\|^2,\\
   \label{nv2}
\int_{\partial_\ext V_{j,\delta}} v_j^2 \dd s
&\le \dfrac{a\log^2\alpha}{\alpha} \Big(N^V_j[v_j,v_j]+ \alpha^2 \|v_j\|^2\Big)
+ \dfrac{a\log^3\alpha}{\alpha^2} \|u\|^2,\\
N^W_j[w_j,w_j]&\ge -\alpha^2\|P_j u\|^2+ \big\|(P_j u)'\big\|^2-\dfrac{a\log\alpha}{\alpha}\,\|P_j u\|^2.
   \label{nw0}
	\end{align}
Consider the self-adjoint operators
\begin{align*}
B&:= R^\Omega_\alpha +\alpha^2 +\dfrac{(M+a)\log\alpha}{\alpha}
\text{ viewed as an operator in $\cH$,}\\
B'&:=\boplus_{j=1}^M D_{j,\delta} \text{ in } \cH',
\end{align*}
with
$\qdom(B)=H^1(\Omega)\mathop{\cap} \cH$ and $\qdom(B')=\bigoplus\nolimits_{j=1}^M H^1_0(I_{j,\delta})$.
Recall that 
\begin{equation}
   \label{eq-buu1}
\begin{aligned}
B[u,u]&=\int_{\Omega} |\nabla u|^2\dd x-\alpha\int_{\partial\Omega} u^2\dd s+ \alpha^2\int_\Omega u^2\dd x\\
&\quad+\dfrac{(M+a)\log\alpha}{\alpha}\int_\Omega u^2\dd x\\
&=\sum\nolimits_{j=1}^M \bigg(
\int_{V_{j,\delta}} |\nabla u|^2\dd x-\alpha\int_{\partial_* V_{j,\delta}} u^2\dd s +
\alpha^2\int_{V_{j,\delta}} u^2\dd x
\bigg)\\
&\quad +\sum\nolimits_{j=1}^M \bigg(
\int_{W_{j,\delta}} |\nabla u|^2\dd x-\alpha\int_{\partial_* W_{j,\delta}} u^2\dd s +
\alpha^2\int_{W_{j,\delta}} u^2\dd x
\bigg)\\
&\quad +\int_{\Omega^c_\delta} |\nabla u|^2\dd x + \alpha^2\int_{\Omega^c_\delta} u^2\dd x
+\dfrac{(M+a)\log\alpha}{\alpha}\int_\Omega u^2\dd x\\
&\ge\sum\nolimits_{j=1}^M \Big(
N^V_j[v_j,v_j] + \alpha^2\|v_j\|^2\Big)\\
&\quad+\sum\nolimits_{j=1}^M \Big(N^W_j[w_j,w_j] + \alpha^2\|w_j\|^2\Big)\\
&\quad + \alpha^2 \|u_c\|^2
+\dfrac{(M+a)\log\alpha}{\alpha}\|u\|^2.
\end{aligned}
\end{equation}
Using \eqref{nv1} and \eqref{nv2} we obtain
\begin{multline}
 \label{bound1}
    \sum_{j=1}^M \Big(
N^V_j[v_j,v_j] + \alpha^2\|v_j\|^2\Big)\\
\ge \dfrac{\alpha^2}{a\log^2\alpha}\,\dfrac{1}{2}\Big(\sum_{j=1}^M \|v_j\|^2 + {\frac1\alpha}\sum_{j=1}^M \int_{\partial_\ext V_{j,\delta}} v_j^2\dd s\Big)
-\dfrac{M\log\alpha}{\alpha}\, \|u\|^2.
\end{multline}
On the other hand, with the help of \eqref{nw0} we estimate
\begin{multline}
    \label{bound2}
\sum_{j=1}^M \Big(N^W_j[w_j,w_j] + \alpha^2\|w_j\|^2\Big)\\
\ge
\sum_{j=1}^M \big\|(P_j u)'\big\|^2-\dfrac{a\log\alpha}{\alpha} \|u\|^2 +\alpha^2\sum_{j=1}^M\Big( \|w_j\|^2-\|P_j u\|^2\Big),
\end{multline}
where we use $\sum_j \|P_j u\|^2\le \sum_j \lVert w_j\rVert^2\leq \lVert u \rVert^2$. 
Using \eqref{bound1} and \eqref{bound2} in \eqref{eq-buu1} we arrive at
\begin{equation}
    \label{eq-buu2}
\begin{aligned}
B[u,u]&\ge \dfrac{\alpha^2}{2a\log^2\alpha}\sum\nolimits_{j=1}^M \|v_j\|^2\\
&\quad + \dfrac{\alpha}{2a\log^2\alpha} \sum\nolimits_{j=1}^M \int_{\partial_\ext V_{j,\delta}} v_j^2\dd s + \sum_{j=1}^M\big\|(P_j u)'\big\|^2\\
&\quad + \alpha^2 \sum\nolimits_{j=1}^M\Big( \|w_j\|^2-\|P_j u\|^2\Big) + \alpha^2 \|u_c\|^2.
\end{aligned}
\end{equation}
Each term on the right-hand side is non-negative and, hence, the left-hand side is an upper bound for each term on the right-hand side.
It also implies that $B$ is positive and then
\begin{equation*}
    %\label{e1b}
E_1(B)\equiv E_{K+1}(R^\Omega_\alpha)+ \alpha^2 +\dfrac{(M+a)\log\alpha}{\alpha}\ge 0.
\end{equation*}
By \eqref{eq-ap12}, for any fixed $n\in\mathbb N$ there is $\mu_n>0$ independent of $\alpha$ such that
\begin{equation}
\label{hyp1}
0\le E_n(B)\le \mu_n, \quad \big( 1+ E_n(B)\big)^{-1} \ge (1+\mu_n)^{-1}.
\end{equation}

In order to construct a suitable identification map $J:\qdom(B)\to \qdom(B')$ we pick functions
$\rho^\pm_j\in C^1\big([0,\ell_j]\big)$ such that
\begin{align*}
\rho^+_j&=\begin{cases}
1 & \text{ in a neighborhood of $0$},\\
0 & \text{ in a neighborhood of $\ell_j$},
\end{cases}
\quad j\in \cJ_*,\\
\rho^-_j&=\begin{cases}
0 & \text{ in a neighborhood of $0$},\\
1 & \text{ in a neighborhood of $\ell_j$},
\end{cases}
\quad j\in \cJ_*,
\end{align*}
and then choose a constant $\rho_0>0$ such that
\begin{equation}
   \label{eq-rho0}
\|\rho^\pm_j\|_{L^\infty(0,\ell_j)}+\|(\rho^\pm_j)'\|_{L^\infty(0,\ell_j)}\le \rho_0 \text{ for all }j\in \cJ_*.
\end{equation}
Recall that 
$I_{j,\delta}:=(\lambda_j\delta,\ell_j-\lambda_{j+1}\delta)=:(\iota_j,\tau_j)$, hence,
\[
{\rho_j^+(\iota_j)=1, \quad \rho_j^+(\tau_j)=0, \quad
\rho_j^-(\iota_j)=0, \quad \rho_j^-(\tau_j)=1},
\]
as $\alpha$ is sufficiently large. Therefore, the map
\begin{gather*}
J:\qdom(B)\to \qdom(B')\equiv \boplus\nolimits_{j=1}^M H^1_0(I_{j,\delta}),
\quad J u= (J_j u),\\
(J_j u)(s):=(P_j u)(s)-(P_j u)(\iota_j)\rho^+_j(s) - (P_j u)(\tau_j)\rho^-_j(s)
\end{gather*}
is well-defined. We estimate, using the Cauchy-Schwarz inequality,
\begin{multline}
   \label{pj002}
\big|P_j u(\iota_j)\big|^2+\big|P_j u(\tau_j)\big|^2\\
=\Big( \int_0^\delta \psi(t)\,w_j\big(\Phi_j(\iota_j,t)\big)\dd t\Big)^2
+\Big( \int_0^\delta \psi(t) \,w_j\big(\Phi_j(\tau_j,t)\big)\dd t\Big)^2\\
\le \int_0^\delta w_j\big(\Phi_j(\iota_j,t)\big)^2\, dt + \int_0^\delta w_j\big(\Phi_j(\tau_j,t)\big)^2\dd t
\equiv \int_{\partial_\ext W_{j,\delta}} w_j^2\dd s.
\end{multline}
Recall the inequality $(x+y)^2\ge(1-\varepsilon)x^2-y^2/\varepsilon$ valid for any $x,y\in\RR$ and $\varepsilon>0$. Then
\begin{align*}
\|J_j u\|^2&=\int_{I_{j,\delta}} \Big|(P_j u)(s)-(P_j u)(\iota_j)\rho^+_j(s) - (P_j u)(\tau_j)\rho^-_j(s)\Big|^2\dd s\\
&\ge (1-\varepsilon)\int_{I_{j,\delta}} \Big|(P_j u)(s)\Big|^2\dd s\\
&\quad-\dfrac{1}{\varepsilon}\int_{I_{j,\delta}}\Big|(P_j u)(\iota_j)\rho^+_j(s) + (P_j u)(\tau_j)\rho^-_j(s)\Big|^2\dd s,
\end{align*}
and, using \eqref{pj002} and the constant $\rho_0$ from \eqref{eq-rho0} we have
\begin{align*}
&\begin{aligned}
\int_{I_{j,\delta}}&\Big|(P_j u)(\iota_j)\rho^+_j(s)  + (P_j u)(\tau_j)\rho^-_j(s)\Big|^2\dd s\\
%\le 2 \int_{I_{j,\delta}}\bigg(\Big|(P_j u)(\iota_j)\rho^+_j(s)\Big|^2 + \Big|(P_j u)(\tau_j)\rho^-_j(s)\Big|^2\bigg)\dd s\\
&\le 2\ell_j\rho_0^2 \Big( \big|P_j u(\iota_j)\big|^2+\big|P_j u(\tau_j)\big|^2 \Big)\le 2 \ell_j \rho_0^2 \int_{\partial_\ext W_{j,\delta}} w_j^2\dd s,
\end{aligned}\\
&\|J_j u\|^2\ge (1-\varepsilon)\|P_j u\|^2-\dfrac{2 \ell \rho_0^2}{\varepsilon}\int_{\partial_\ext W_{j,\delta}} w_j^2\dd s,
\quad \ell:=\max \ell_j.
\end{align*}
Therefore, using \eqref{eq-uvw}  we arrive at
\begin{align*}
\|u\|^2-\|Ju\|^2&=\sum_{j=1}^M \|v_j\|^2+\sum_{j=1}^M \|w_j\|^2+\|u_c\|^2 -\sum_{j=1}^M \|J_j u\|^2\\
&\le \sum_{j=1}^M \|v_j\|^2+\sum_{j=1}^M \|w_j\|^2+\|u_c\|^2\\
&\qquad-(1-\varepsilon)\sum_{j=1}^M\|P_j u\|^2 
+\dfrac{2 \ell \rho_0^2}{\varepsilon}\sum_{j=1}^M\int_{\partial_\ext W_{j,\delta}} w_j^2\dd s\\
&=\sum_{j=1}^M \|v_j\|^2+\sum_{j=1}^M\big( \|w_j\|^2- \|P_j u\|^2 \big)\\
&\qquad+\varepsilon\sum_{j=1}^M\|P_j u\|^2
+ \dfrac{2 \ell \rho_0^2}{\varepsilon}\sum_{j=1}^M\int_{\partial_\ext V_{j,\delta}} v_j^2\dd s+\|u_c\|^2\\
&\le \sum_{j=1}^M \|v_j\|^2+\sum_{j=1}^M\big( \|w_j\|^2- \|P_j u\|^2 \big)\\
&\qquad+ \varepsilon \|u\|^2 + \dfrac{2 \ell \rho_0^2}{\varepsilon}\sum_{j=1}^M\int_{\partial_\ext V_{j,\delta}} v_j^2\dd s+\|u_c\|^2.
\end{align*}
Using \eqref{eq-buu2} we obtain an upper bound for all terms on the right-hand side except $\varepsilon \|u\|^2$:
\[
\|u\|^2-\|Ju\|^2\le \Big(\dfrac{2a\log^2\alpha}{\alpha^2}+{\dfrac{2}{\alpha^2}}+\dfrac{4 \ell \rho_0^2 a\log^2\alpha}{\varepsilon \alpha}\Big) B[u,u]+\varepsilon \|u\|^2.
\]
Taking $\varepsilon:=\log\alpha/\sqrt{\alpha}$ and choosing $c_1>0$ sufficiently large we obtain
\begin{equation}
    \label{eps1}
\|u\|^2-\|Ju\|^2\le \dfrac{c_1 \log\alpha}{\sqrt{\alpha}}\Big( B[u,u]+\|u\|^2\Big).
\end{equation}

To study the difference $B'[Ju,Ju]-B[u,u]$ recall that $B'[Ju,Ju]=\sum_{j=1}^M \|(J_ju)'\|^2$.
Using the inequality $(x+y)^2\le(1+\varepsilon)x^2+2y^2/\varepsilon$ valid for all $x,y\in\RR$ and $\varepsilon\in(0,1)$
we estimate
\begin{align*}
\big\|(J_j u)'\big\|^2&=\int_{I_{j,\delta}} \Big|(P_j u)'(s)-(P_j u)(\iota_j)(\rho^+_j)'(s) - (P_j u)(\tau_j)(\rho^-_j)'(s)\Big|^2\dd s\\
&\le (1+\varepsilon)\int_{I_{j,\delta}} \Big|(P_j u)'(s)\Big|^2\dd s\\
&\quad + \dfrac{2}{\varepsilon}\int_{I_{j,\delta}}\Big|(P_j u)(\iota_j)(\rho^+_j)'(s) + (P_j u)(\tau_j)(\rho^-_j)'(s)\Big|^2\dd s.
\end{align*}
Using \eqref{pj002} for the last term and the constant $\rho_0$ from \eqref{eq-rho0} we have
\begin{align*}
&\int_{I_{j,\delta}}\Big|(P_j u)(\iota_j)(\rho^+_j)'(s) + (P_j u)(\tau_j)(\rho^-_j)'(s)\Big|^2\dd s\le 2 \ell \rho_0^2 \int_{\partial_\ext W_{j,\delta}} w_j^2\dd s,\\
&\begin{aligned}
B'[Ju,Ju]&\le (1+\varepsilon)\sum\nolimits_{j=1}^M \big\|(P_ju)'\big\|^2 + \dfrac{4\ell\rho_0^2}{\varepsilon} \sum\nolimits_{j=1}^M \int_{\partial_\ext W_{j,\delta}} w_j^2\dd s
\end{aligned}
\end{align*}
Recall that due to \eqref{eq-buu2} we have $B[u,u]\ge \sum_{j=1}^M \big\|(P_ju)'\big\|^2$ and
\[
B[u,u]\ge \dfrac{\alpha}{2a\log^2\alpha} \sum_{j=1}^M \int_{\partial_\ext V_{j,\delta}} v_j^2\dd s
\equiv \dfrac{\alpha}{2a\log^2\alpha} \sum_{j=1}^M \int_{\partial_\ext W_{j,\delta}} w_j^2\dd s,
\]
where we have used \eqref{eq-uvw}.
Therefore,
\begin{align*}
B'[Ju,Ju]-B[u,u]&\le \varepsilon \sum\nolimits_{j=1}^M \big\|(P_ju)'\big\|^2+\dfrac{4\ell\rho_0^2}{\varepsilon} \sum\nolimits_{j=1}^M \int_{\partial_\ext W_{j,\delta}} w_j^2\dd s\\
&\le \left(\varepsilon+\dfrac{4\ell\rho_0^2}{\varepsilon}
\cdot \dfrac{2a\log^2\alpha}{\alpha}\right)\, B[u,u].
\end{align*}
Therefore, by setting $\varepsilon=\log\alpha/\sqrt\alpha$ and by choosing $c_2>0$ sufficiently large
we arrive at
\begin{equation}
    \label{eps2}
B'[Ju,Ju]-B[u,u]\le \dfrac{c_2 \log\alpha}{\sqrt{\alpha}}\, B[u,u]\le \dfrac{c_2 \log\alpha}{\sqrt{\alpha}}\,\Big(B[u,u]+\|u\|^2\Big).
\end{equation}
With the inequalities \eqref{eps1} and \eqref{eps2} at hand, we are in the situation of Proposition~\ref{prop6} with
 $\varepsilon_j:=c_j\log\alpha/\sqrt{\alpha}$, $j\in\{1,2\}$. Remark that for each fixed $n$ the assumption $\varepsilon_1<1/\big(1+E_n(B)\big)$ is satisfied due to~\eqref{hyp1}. Hence, for each fixed
$n$ we have
\begin{equation}
    \label{eq-bbb}
E_n\Big(\boplus_{j=1}^M D_{j,\delta}\Big)\equiv E_n(B')\le E_n(B) + \dfrac{\log\alpha}{\sqrt{\alpha}}\cdot\dfrac{\big(c_1 E_n(B)+c_2\big)\big(1+E_n(B)\big)}{1-c_1\big(1+E_n(B)\big)\log\alpha/\sqrt{\alpha}}.
\end{equation}
By \eqref{hyp1} we have $E_n(B)=\cO(1)$ for each fixed $n$,
and the substitution into \eqref{eq-bbb} gives
\[
E_{K+n}(R^\Omega_\alpha)\ge -\alpha^2+E_n\Big(\boplus_{j=1}^M D_{j,\delta}\Big)+\cO\Big( \dfrac{\log\alpha}{\sqrt{\alpha}}\Big),
\]
and it remains to note that for fixed $n$ and $j$ one has
\[
E_n(D_{j,\delta})=E_n(D_j)+\cO(\delta)=E_n(D_j)+\cO\Big( \dfrac{\log\alpha}{\alpha}\Big)=E_n(D_j)+o\Big( \dfrac{\log\alpha}{\sqrt{\alpha}}\Big),
\]
which implies $E_n(\boplus\nolimits_{j=1}^M D_{j,\delta})=E_n(\boplus\nolimits_{j=1}^M D_j)
+\cO(\frac{\log \alpha }{\sqrt{\alpha}})$.  This concludes the proof of Proposition~\ref{prop-low-rob1}.
\end{proof}

\section{Robin eigenvalues in curvilinear polygons}\label{sec-thm23}

If one tries to adapt the preceding proof scheme to curvilinear polygons, a number of points require more attention:

\begin{itemize}

\item[(a)] The initial construction of the vertex neighborhood $V_{j,\delta}$ become more technical: the shape
of these neighborhoods cannot be chosen at random, as the subsequent analysis need the presence of two straight sides to which the side neighborhoods $W_{j,\delta}$ are glued.

\item[(b)] A suitable control of eigenvalues of Robin-Dirichlet/Neumann Laplacians in $V_{j,\delta}$ is needed. Remark that
these neighborhoods are not truncated sectors anymore, but curvilinear polygons. Hence, a suitable analog of the non-resonance condition
is needed. This can be achieved using a suitable diffeomorphism between $V_{j,\delta}$ and truncated convex sectors.

\item[(c)] We need an analog of the radial cut-off functions $\varphi_\delta$ to prove an analog of Lemma~\ref{dist-ll0} for the curvilinear case.
The cut-off functions are needed to satisfy the Neumann boundary condition at $\partial\Omega$, in order to ensure
that the truncated eigenfunction are still in the domain of the Robin laplacian. This is important for the constructions,
as it allows one to apply Proposition~\ref{propdist2} to estimate the distance between two subspaces.

\item[(d)] The analysis of the Robin laplacians in side neighborhoods $W_{j,\delta}$ become more involved, as a non-trivial curvature contribution appears.

\item[(e)] The lower bound for the eigenvalues in Proposition~\ref{prop-low-rob1} is in part based on the fact  that the individual eigenvalues of the operators $B$ and $B'$
are bounded for large $\alpha$. This is an important point when using Proposition~\ref{prop6}: if the eigenvalues of $B$
become large, it becomes difficult to satisfy the initial assumption on $\varepsilon_1 \big(1+E_1(B)\big)<1$, see the discussion in subsection~\ref{ssec-var}.
\end{itemize}

We remark that the points (a)--(c) are purely geometric, and can be of importance for the analysis of other problems in curvilinear polygons.
We are not aware of any suitable construction in the literature (due to the very specific shape of the vertex neighborhoods), and we have decided to
give a self-contained discussion in Appendix~\ref{appa}, to which we refer in the main text. In the present text we were not able
to overcome completely the difficulties mentioned under (d) and (e), and we concentrate ourselves on two special but important cases.

\subsection{Decomposition of curvilinear polygons}\label{ssec-poly}

\begin{figure}

\centering

\begin{tabular}{cc}
\begin{minipage}[c]{52mm}
\begin{center}
\includegraphics[height=26mm]{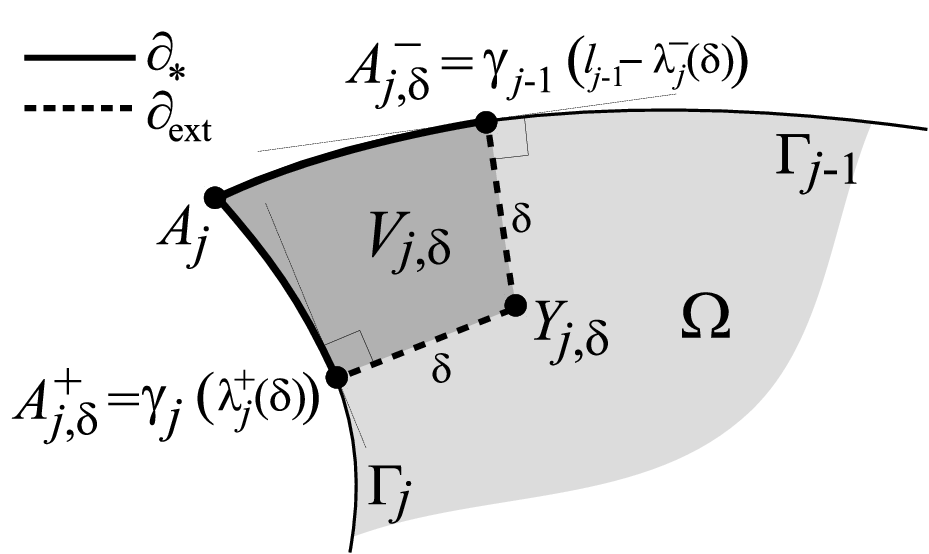}
\end{center}
\end{minipage}
&
\begin{minipage}[c]{52mm}
\begin{center}
\includegraphics[height=26mm]{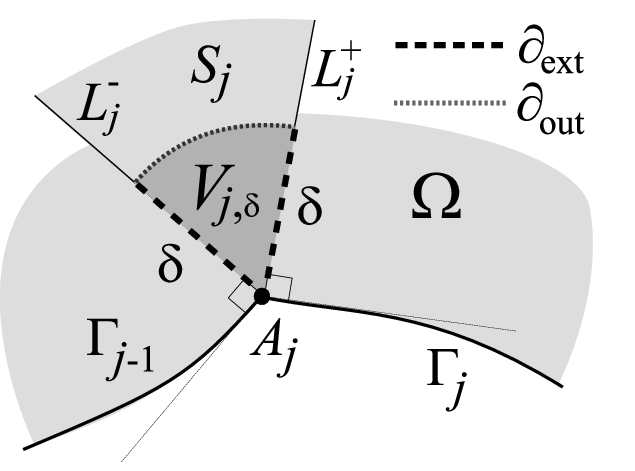}
\end{center}
\end{minipage}\\
(a)  & (b) 
\end{tabular}
\caption{The construction of the neighborhoods $V_{j,\delta}$: (a) convex vertex, (b) concave vertex.
The partial boundary $\partial_* V_{j,\delta}$
is shown with the thick solid line, the part $\partial_\ext V_{j,\delta}$ is indicated with the thick dashed line,
and the part $\partial_\out V_{j,\delta}$ with the gray dotted line.\label{conv-vois}}
\end{figure}

Let us describe more precisely the class of domains $\Omega$ we are going to deal with
as well as its decomposition into pieces of special shape. Once the geometric justifications has been made, see Appendix~\ref{appa}, the latter differs only in minor details from the case of straight polygons. 
A bounded domain $\Omega\subset \mathbb{R}^2$ will be called a \emph{curvilinear polygon}
with $M\ge 1$ vertices if there exist $A_1,\dots,A_M\in\mathbb{R}^2$ and $\ell_1,\dots,\ell_M >0$ such that:
\begin{itemize}
\item there are injective $C^3$ maps $\gamma_j:\RR\to \RR^2$ with $|\gamma'_j|=1$ such that
\[
	\gamma_j(0) = A_j,\quad \gamma_j(\ell_j) = A_{j+1}, \quad j\in \{1,\dots,M\},
\]
where we identify $A_0\equiv A_M$ and $A_{M+1}\equiv A_1$, and the same numbering convention applies to the finite arcs
$\Gamma_j:=\gamma_j\big((0,\ell_j)\big)$
which we assume mutually disjoint and such that
$\Gamma:=\partial\Omega = \bigcup_{j=1}^M \overline{\Gamma_j}$.
\item 
The orientation of each $\gamma_j$ is assumed to be chosen in such a way that if $\nu_j(t)$ is the \emph{outer} unit normal to $\partial\Omega$
at a point $\gamma_j(t)$, then $\nu_j(s)\wedge \gamma'_j(s)=1$, i.e. $\nu_j(s)$ is obtained by rotating the tangent vector $\gamma'_j(t)$ by $\frac{\pi}{2}$ in the clockwise direction,
and the  curvature $H_j(t)$ of $\Gamma_j$ at the point $\gamma_j(s)$ is defined by
$\nu'_j(s)=H_j(t) \,\gamma'_j(s)$.
\item  By $\theta_j\in[0,\pi]$ we denote the half-angle of the boundary at a vertex $A_j$, i.e. the number $\theta_j\in[0,\pi]$ is 
characterized by the conditions
\[
	\cos(2 \theta_j) = \gamma_{j-1}'(\ell_{j-1})\cdot \big(-\gamma'_j(0)\big),
	\quad \sin(2\theta_j) = \gamma_j'(0) \wedge \big(-\gamma_{j-1}'(\ell_{j-1})\big).
\]
Our assumption is that there are neither zero angles nor artificial vertices, i.e. $\theta_j\notin \big\{0, \frac{\pi}{2},\pi\big\}$ for $j=1,\dots,M$.
\end{itemize}
The above points $A_j\in\partial\Omega$ will be called the \emph{vertices} of $\Omega$. Furthermore, one says that $A_j$ is a \emph{convex} vertex if $\theta_j<\frac{\pi}{2}$
and is a \emph{concave} one otherwise, and we denote
\[
\cJ_\cvx:=\{j:\, A_j\text{ is convex}\}.
\]
We refer to Figure~\ref{fig-domain} in the introduction for an illustration, and in that case one has $\cJ_\cvx=\{1,2\}$.
Let us now proceed with a special decomposition of $\Omega$. For small $\delta>0$, denote
\[
\Omega_\delta=\big\{x\in \Omega:\, \dist(x,\partial\Omega)<\delta\big\},
\quad \Omega_\delta^c:=\Omega\setminus \overline{\Omega_\delta}.
\]
We further decompose $\Omega_\delta$ near each vertex as follows:
\begin{itemize}
\item Let $A_j$ be a convex vertex. The following constructions are consequences of Lemma~\ref{prox} and are  illustrated in Figure~\ref{conv-vois}(a).
For sufficiently small $\delta$ there exists a unique point $Y_{j,\delta}\in \Omega$
such that $\dist(Y_{j,\delta},\Gamma_{j-1})=\dist(Y_{j,\delta},\Gamma_j)=\delta$, and there are uniquely defined numbers $\lambda_j^\pm(\delta)>0$ such that
the points
\[
A_{j,\delta}^-:=\gamma_{j-1}\big(\ell_j-\lambda_j^-(\delta)\big), \quad
A_{j,\delta}^+:=\gamma_j\big(\lambda_j^+(\delta) \big)
\]
satisfy $|Y_{j,\delta}-A_{j,\delta}^-|=|Y_{j,\delta}-A_{j,\delta}^+|=\delta$. The quantities $\lambda_j^\pm(\delta)>0$ satisfy
\begin{equation}
   \label{ljt0}
\lambda_j^\pm(\delta) =  \delta \cotan \theta_j + \cO(\delta^2) \quad \text{ for } \delta\to 0^+.
\end{equation}
We denote by $V_{j,\delta}$ the curvilinear quadrangle whose boundary
consists of the arcs $\gamma_{j-1}\big(\big[\ell_j-\lambda_j^-(\delta),\ell_j\big]\big)$, $\gamma_{j}\big(\big[0,\lambda_j^+(\delta)\big]\big)$
and the segments $A^\pm_{j,\delta}Y_{j,\delta}$, and we decompose its boundary into the following
parts:
\[
\partial_* V_{j,\delta}:=\partial V_{j,\delta}\cap\partial\Omega, \quad
\partial_\ext V_{j,\delta}:=\partial V_{j,\delta}\setminus \partial_* V_{j,\delta},
\quad
\partial_\out V_{j,\delta}:=\emptyset.
\]

\item Let $A_j$ be a concave vertex.  The constructions are illustrated in Figure~\ref{conv-vois}(b).
Let $L^-_j$ be the half-line emanating from $A_j$, orthogonal to $\Gamma_{j-1}$ at $A_j$ and directed inside $\Omega$.
By $L^+_j$ we denote the half-line emanating from $A_j$, orthogonal to $\Gamma_j$ at $A_j$ and directed inside $\Omega$.
Denote by $S_j$ the infinite sector bounded by $L^-_j$ and $L^+_j$ which lies inside $\Omega$ near $A_j$. Then we set
\[
V_{j,\delta}:=S_j\cap B(A_j,\delta), \quad \lambda^\pm_j(\delta):=0,
\]
and decompose its boundary as follows:
\[
\partial_* V_{j,\delta}:=\emptyset,
\quad
\partial_\out V_{j,\delta}:=\partial V_{j,\delta}\cap \partial \Omega^c_\delta,
\quad
\partial_\ext V_{j,\delta}:=\partial V_{j,\delta}\setminus \partial_\out V_{j,\delta}.
\]
\end{itemize}
The ``length deficiency'' $\lambda^+_j(\delta)+\lambda_j^-(\delta)$
is exactly the length of $\partial_*V_{j,\delta}$ for both convex and concave vertices.

The set
${W_\delta}:=\Omega_\delta\setminus\overline{\bigcup_{j=1}^M V_{j,\delta}}$ 
is the union of $M$ disjoint curvilinear rectangles: if one denotes
\begin{gather*}
I_{j,\delta}:=\big(\lambda^+_j(\delta),\ell_j-\lambda^-_{j+1}(\delta)\big), \quad \Pi_{j,\delta}:=I_{j,\delta}\times(0,\delta),\\
W_{j,\delta}:=\Phi_j(\Pi_{j,\delta}), \quad \Phi_j(s,t):=\gamma_j(s)-t\nu_j(s),
\end{gather*}
then ${W_\delta}=\bigcup_{j=1}^M W_{j,\delta}$. We decompose the boundary of each $W_{j,\delta}$ as follows:
\begin{gather*}
\partial_* W_{j,\delta}:=\partial W_{j,\delta} \cap \partial\Omega,
\quad
\partial_\out W_{j,\delta}:=\partial W_{j,\delta}\cap \partial \Omega^c_\delta,\\
\partial_\ext W_{j,\delta}:=\partial W_{j,\delta}\setminus\Big(
\partial_* W_{j,\delta}\cup \partial_\out W_{j,\delta}
\Big).
\end{gather*}
The resulting decomposition of $\Omega$ is illustrated in Figure~\ref{fig-domain2}, and we always have
\begin{equation}
    \label{eq-vjwj2}
\bigcup\nolimits_{j=1}^M \partial_\ext V_{j,\delta} = \bigcup\nolimits_{j=1}^M \partial_\ext W_{j,\delta}.
\end{equation}

\begin{figure}

\centering

\includegraphics[width=60mm]{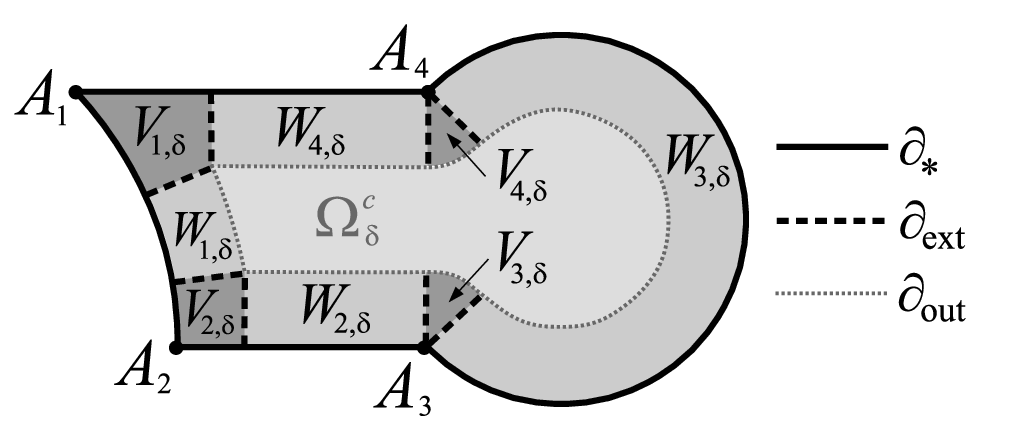}

\caption{Decomposition of a curvilinear polygon.\label{fig-domain2}}

\end{figure}

\subsection{Some estimates for curvilinear neighborhoods}

With each $j\in\{1,\dots,M\}$ we associate the corresponding number $\kappa(\theta_j)$
of discrete eigenvalues of the Robin Laplacians in the infinite sector of  aperture $2\theta_j$ (see Section~\ref{sec-sectors})
and set
\begin{align*}
K&:=\kappa(\theta_1)+\dots +\kappa(\theta_M)\equiv \sum\nolimits_{j\in\cJ_\cvx} \kappa(\theta_j),\\
\cE&:=\text{ the disjoint union of } \big\{ \cE_n(\theta_j), \ n=1,\dots,\kappa(\theta_j)\big\} \text{  for } j\in\cJ_\cvx,\\
\cE_n&:=\text{ the $n$th element of $\cE$ when numbered in the non-decreasing order,}
\end{align*}
(see Subsection~\ref{sec-sectors} for a detailed notation).
For what follows we introduce several operators:
\begin{align*}
N^V_j&:=\begin{minipage}[t]{90mm}\raggedright 
the Laplacian in $V_{j,\delta}$ with $\alpha$-Robin condition at $\partial_* V_{j,\delta}$
and Neumann condition at the rest of the boundary,\end{minipage}\\
D^V_j&:=\begin{minipage}[t]{90mm}\raggedright 
the Laplacian in $V_{j,\delta}$ with $\alpha$-Robin condition at $\partial_* V_{j,\delta}$
and Dirichlet condition at the rest of the boundary.\end{minipage}\\
R^V_j&:=\begin{minipage}[t]{90mm}\raggedright 
the Laplacian in $V_{j,\delta}$ with $\alpha$-Robin condition at the whole boundary.\end{minipage}
\end{align*}
We remark that for concave vertices $A_j$, the respective operators $(N/D)^V_j$
are just the Neumann/Dirichlet Laplacians in $V_{j,\delta}$ due to $\partial_* V_{j,\delta}=\emptyset$.
Furthermore, denote
\begin{align*}
N^W_j&:=\begin{minipage}[t]{90mm}\raggedright 
the Laplacian in $W_{j,\delta}$ with $\alpha$-Robin condition at $\partial_* W_{j,\delta}$
and Neumann condition at the rest of the boundary,\end{minipage}\\
D^W_j&:=\begin{minipage}[t]{90mm}\raggedright 
the Laplacian in $W_{j,\delta}$ with $\alpha$-Robin condition at $\partial_* W_{j,\delta}$
and Dirichlet condition at the rest of the boundary,\end{minipage}
\end{align*}
Finally, introduce
\[
N_c:=
\begin{minipage}[t]{90mm}\raggedright 
the Neumann Laplacian in $\Omega^c_\delta$.\end{minipage}
\]
Due to Lemma~\ref{mapfi} (Appendix~\ref{appa}), for each $j\in\cJ_\cvx$ one can find a bi-Lipschitz map $\Psi_j$
between a neighborhood of the origin and a neighborhood of $A_j$,
a rotation $Z_j$ with $\Psi'_j(x)=Z_j+\cO\big(|x|\big)$ for $x\to 0$
and a $C^2$ smooth function $r_j$ defined near $0$ with $r(0)=0$ and $r'(0)=\cotan\theta_j$
such that for all sufficiently small $\delta>0$ one has
\[
\Psi_j(\cS_{\theta_j}^{r(\delta)})=V_{j,t},
\quad
\Psi_j(\partial_*\cS_{\theta_j}^{r(\delta)})=\partial_* V_{j,\delta},
\quad
\Phi_j(\partial_\ext\cS_{\theta_j}^{r(\delta)})=\partial_\ext V_{j,\delta}.
\]
Hence, for $u\in H^1(V_{j,\delta})$ one can use $u\circ \Psi_j\in H^1(\cS_{\theta_j}^{r(\delta)})$ as test functions in truncated sectors,
which implies in the standard way, see e.g. \cite[Lemma~3.3]{bp}, the following estimates for the eigenvalues of $N^V_j$, $D^V_j$ and $R^V_j$:

\begin{lemma}\label{lem-curv}
There exist $a>0$, $a_0>0$, $\delta_0>0$ such that for all $\delta\in(0,\delta_0)$, $\alpha>0$, $n\in\NN$, $j\in\cJ_\cvx$ there holds
\begin{gather*}
(1-a_0\delta) E_n(N_{\theta_j,\alpha(1+a\delta)}^{r(\delta)})\le E_n(N^V_j)\le (1+a_0\delta) E_n(N_{\theta,\alpha(1-a\delta)}^{r(\delta)}),\\
(1-a_0\delta) E_n(D_{\theta_j,\alpha(1+a\delta)}^{r(\delta)})\le E_n(D^V_j)\le (1+a_0\delta) E_n(D_{\theta,\alpha(1-a\delta)}^{r(\delta)}),\\
(1-a_0\delta) E_n(R_{\theta_j,\alpha(1+a\delta)}^{r(\delta)})\le E_n(R^V_j)\le (1+a_0\delta) E_n(R_{\theta,\alpha(1-a\delta)}^{r(\delta)}).
\end{gather*}
\end{lemma}
Here we recall that the operators $(D/N/R)_{\theta,\alpha}^r$ in truncated sectors $\cS_\theta^r$
were defined in~\eqref{eq-tqr}. Using the estimates of Lemmas~\ref{dalph} and~\ref{nalph}
for the eigenvalues in truncated sectors and by literally repeating the proof to obtain an Agmon estimate
in~Lemma~\ref{agmon22a}, we arrive then at first estimates for $(D/N/R)^V_j$:

\begin{coro}\label{lem-trunc5}
There is $b>0$ such that for $\delta\to 0^+$ and $\alpha\delta\to+\infty$ one has
\begin{align*}
E_n(D^V_j)&= \cE_n(\theta_j)\,\alpha^2+ \cO(\alpha^2\delta+\alpha^2 e^{-b\alpha\delta}), & n\in\big\{1,\dots,\kappa(\theta)\big\},\\
E_n(N^V_j)&= \cE_n(\theta_j)\,\alpha^2+ \cO(\alpha^2\delta+ 1/\delta^2), & n\in\big\{1,\dots,\kappa(\theta)\big\},
\end{align*}
and $E_{\kappa(\theta_j)+1}(D^V_j)\ge E_{\kappa(\theta_j)+1}(N^V_j)\ge -\alpha^2+o(\alpha^2)$.
In addition, there exist  $c>0$ and $C>0$ such that if $n\in\big\{1,\dots,\kappa(\theta_j)\big\}$ and $\psi_n$
is an eigenfunction of $N^V_j$ for the $n$th eigenvalue, then for $\delta\to 0^+$ and $\alpha\delta\to +\infty$
there holds
\[
\int_{V_{j,\delta}} e^{c\alpha|x|}\Big ( \dfrac{1}{\alpha^2} \big|\nabla \psi_n(x)\big|^2+ \big|\psi_n(x)\big|^2\Big)\, dx
\le C \,\|\psi_n\|^2_{L^2(V_{j,\delta})}.
\]
There exists $a>0$ such that $E_1(R^V_j)\ge -a \alpha^2$ for  $\delta\to 0^+$ and $\alpha\delta\to+\infty$.
\end{coro}

Like in the case of straight polygons, we introduce the following subspaces:
\begin{align*}
L&:=\begin{minipage}[t]{90mm}\raggedright
the subspace of $L^2(\Omega)$ spanned by the first $K$ eigenfunctions  of $R^\Omega_\alpha$,
\end{minipage}\\
L_j&:=\begin{minipage}[t]{90mm}\raggedright
the subspace of $L^2(V_{j,\delta})$ spanned by the first $\kappa(\theta_j)$ eigenfunctions of
$N^V_j$, with $j\in\cJ_\cvx$,
\end{minipage}\\
\sigma_j&:L^2(\Omega)\to L^2(V_{j,\delta}) \text{ the operator of restriction,}\\
&\quad\text{$(\sigma_j u)(x)=u(x)$
for $x\in V_{j,\delta}$,}
\end{align*}
then the adjoint operators  $\sigma^*_j:L^2(V_{j,\delta})\to L^2(\Omega)$ are the operators of extension by zero.
The following distance estimate will again be of importance:

\begin{lemma}\label{dist-ll}
Let $j\in\cJ_\cvx$, then in the limit $\delta\to 0^+$
and $\alpha\delta\to +\infty$ there holds $d(\sigma^*_j L_j, L)=\cO(e^{-c\alpha\delta})$ with some fixed $c>0$.
\end{lemma}

\begin{proof}
Pick $0<a<b<1$, then due to Lemma~\ref{cutoff} one can find smooth cut-off functions $\varphi_\delta\in C^2(\Bar\Omega)$
with the following properties:
\begin{itemize}
\item $0\le \varphi_\delta\le 1$, and for all $\beta\in\NN^2$ with $1\le |\beta|\le 2$
there holds $\|\partial^\beta \varphi_\delta\|_\infty\le C \delta^{-|\beta|}$,
\item $\varphi_\delta=1$ in $V_{j,a\delta}$, and $\varphi_\delta=0$ in $\Omega\setminus \overline{V_{j,b\delta}}$,
\item the normal derivative of $\varphi_\delta$ at $\partial\Omega$ is zero,
\end{itemize}
where $C>0$ is some fixed constant. By Corollary~\ref{lem-length}
one can find $a_0>0$ such that $|x-A_j|> a_0 \delta $ for $x\in V_{j,\delta}\setminus \overline{V_{j,a\delta}}$.
As the normal derivative of $\varphi_\delta$ at $\partial\Omega$ is zero, it follows that
for any $v\in \dom(N^V_j)$ we have then $\varphi_\delta v\in\dom(R^\Omega_\alpha)$,
and the proof works literally as for the straight case (Lemma~\ref{dist-ll0}),
as all other necessary components are contained in Corollary~\ref{lem-trunc5}.
\end{proof}

Let us now give some first estimates for the eigenvalues of $N^W_j$. Recall that  $H_j$ stands for the curvature of the $j$th side of $\Omega$. We denote
\begin{equation}
   \label{hhj0}
H_{j,*}:=\max_{s\in[0,\ell_j]} H_j(s), \qquad H_*:=\max_{j\in\{1,\dots,M\}} H_{j,*},
\qquad
\cJ_*:=\{j: H_{j,*}=H_*\}.
\end{equation}
As in the case of straight polygons, let us consider the one-dimensional operators
\begin{align*}
D_j&:=\text{ the Dirichlet Laplacian on $(0,\ell_j)$}, \\
D_{j,\delta}&:=\text{ the Dirichlet Laplacian on $I_{j,\delta}$},\\
N_j&:=\text{ the Neumann Laplacian on $(0,\ell_j)$}, \\
N_{j,\delta}&:=\text{ the Neumann Laplacian on $I_{j,\delta}$}.
\end{align*}
A literal repetition of the constructions of~\cite[Sec.~6]{pp15b} gives the following result:
\begin{prop}\label{prop-pp}
For any $j\in\{1,\dots,M\}$, $n\in\NN$ and $\delta:=\alpha^{-\kappa}$ with $\kappa\in[\frac{2}{3},1]$, for $\alpha\to+\infty$ there holds
\begin{equation}
  \label{eq-eend1}
\begin{aligned}
E_n(N^W_j)&=-\alpha^2+E_n(N_{j,\delta}-\alpha H_j)+\cO(1), \\
E_n(D^W_j)&=-\alpha^2+E_n(D_{j,\delta}-\alpha H_j)+\cO(1).
\end{aligned}
\end{equation}
In particular,
\begin{gather}
  \label{eq-eend2}
\begin{aligned}
E_n(N_{j,\delta}-\alpha H_j)&=-H_{j,*}\alpha+ o(\alpha),\\
E_n(D_{j,\delta}-\alpha H_j)&=-H_{j,*}\alpha+ o(\alpha),
\end{aligned}\\
   \label{eq-eend3}
\begin{aligned}
E_n(N^W_j)&=-\alpha^2-H_{j,*} \alpha +o(\alpha),\\
E_n(D^W_j)&=-\alpha^2-H_{j,*}\alpha+ o(\alpha).
\end{aligned}
\end{gather}
\end{prop}

The preceding estimates work for all half-angles angle $\theta_j$. For the rest of the section we assume that
\begin{equation}
   \label{cond-nonres}
\theta_j \text{ is non-resonant for all $j\in\cJ_\cvx$.}
\end{equation}

\subsection{Estimates for non-resonant convex sectors}

Let us pick any $j\in\cJ_\cvx$, which is then non-resonant by assumption.
The estimates for $N^V_j$ given in this subsection will be of crucial importance for the subsequent analysis.
They slightly differ from the respective estimates for the straight case (Subsection~\ref{ssec-nonres}),
as some more parameters will be needed later.
For the rest of the section we assume that $\delta$ is chosen depending
on $\alpha$ such that
\begin{equation}
   \label{eq-daa2}
\alpha\delta\to +\infty, \quad \delta\to 0^+, \quad \alpha^2\delta^3\to0^+
\qquad \text{ as } \alpha\to+\infty.
\end{equation}
An exact choice of $\delta$ will be made at a later stage.

\begin{coro}\label{lem-nonres2}
For any $A\in\RR$ there exists $c>0$ such that under the assumption~\eqref{eq-daa2}
there holds $E_{\kappa(\theta_j)+1}(D^V_j)\ge E_{\kappa(\theta_j)+1}(N^V_j)\ge -\alpha^2 +A\alpha+ c/\delta^2$.
\end{coro}

\begin{proof}
By Lemma~\ref{lem-curv} one has $E_n(N^V_j)\ge (1-a_0\delta) E_n(N_{\theta_j,\alpha(1+a\delta)}^{r(\delta)})$.
As $\theta_j$ is non-resonant,  with some $C>0$ we have
$E_{\kappa(\theta_j)+1}(N_{\theta_j,\alpha}^{r})\ge -\alpha^2+C/r^2$ as $\alpha r$ is large.
In the asymptotic regime under consideration we have  $\alpha(1+a\delta)r_j(\delta)\sim \alpha \delta\cotan\theta_j\to+\infty$, hence,
\begin{align*}
E_n(N_{\theta_j,\alpha(1+a\delta)}^{r(\delta)})&\ge -\alpha^2(1+a\delta)^2+\dfrac{C}{r(\delta)^2}\\
&\ge -\alpha^2-3a\alpha^2\delta + \dfrac{C_0}{\delta^2}, \quad C_0:=\dfrac{C \tan^2\theta_j}{2},\\
(1-a_0\delta) E_n(N_{\theta_j,\alpha(1+a\delta)}^{r(\delta)})&\ge (1-a_0 \delta)\Big(
{}-\alpha^2-3a\alpha^2\delta + \dfrac{C_0}{\delta^2}\Big)\\
&\ge-\alpha^2 -3a\alpha^2\delta + \dfrac{C_0}{\delta^2}-\dfrac{a_0C_0}{\delta}\\
&\ge -\alpha^2 -3a\alpha^2\delta + \dfrac{C_0}{2\delta^2}\\
&=-\alpha^2+A\alpha +\dfrac{1}{\delta^2}\Big( \dfrac{1}{2}\, C_0 - 3a\alpha^2\delta^3-A\alpha\delta^2\Big).
\end{align*}
For $\alpha^2\delta^3\to 0^+$ one has $\alpha \delta^2=\alpha^2\delta^3/(\alpha\delta)\to 0^+$,
and for any fixed $c\in(0,C_0/2)$ there holds
\[
E_{\kappa(\theta_j)+1}(N^V_j)\ge-\alpha^2+A\alpha+ c/\delta^2.
\]
The inequality $E_{\kappa(\theta_j)+1}(D^V_j)\ge E_{\kappa(\theta_j)+1}(N^V_j)$ follows from the min-max principle.
\end{proof}

Proceeding almost literally as in the straight case (Lemma~\ref{lem-trace0})
one puts the preceding assertion into the following special form:

\begin{coro}\label{lem-trace}
For any $A\in\RR$ there exists $b>0$ such that under the assumption~\eqref{eq-daa2} 
there following inequalities holds for any $u\in H^1(V_{j,\delta})\mathop{\cap} L_j^\perp$:
\begin{align}
   %\label{nnn000}
	\label{trr1}
\|v\|^2_{L^2(V_{j,\delta})}&\le b\,\delta^2\Big( N^V_j [v,v] + (\alpha^2-A\alpha)\|v\|^2_{L^2(V_{j,\delta})} \Big),\\
	\label{trr2}
\int_{\partial_\ext V_\delta} v^2\dd s
&\le b\alpha\delta^2\Big( N^V_j [v,v] + (\alpha^2-A\alpha)\|v\|^2_{L^2(V_{j,\delta})}\Big).
\end{align}
\end{coro}

By combining the eigenvalues and eigenfunction estimates of Corollary~\ref{lem-trace} with the distance estimate of Lemma~\ref{dist-ll} like in the proof of Lemma~\ref{vertices1}
we arrive then to the following estimate:
\begin{lemma}\label{vertices2}
For any $A\in\RR$ one can find $b>0$ and $c>0$ such that under the assumption~\eqref{eq-daa2}
there holds, for any $j=1,\dots, M$, and any $u\in H^1(\Omega)$ with $u\perp L$,
\begin{align*}
\|\sigma_j u\|^2_{L^2(V_{j,\delta})}
&\le b\delta^2 \Big(N^V_j[\sigma_j u,\sigma_j u]+ (\alpha^2-A\alpha)\|\sigma_j u\|^2_{L^2(V_{j,\delta})}\Big)\\
&\quad
+ b\alpha^2\delta^2e^{-c\alpha\delta} \|u\|^2_{L^2(\Omega)},\\
\int_{\partial_\ext V_{j,\delta}} (\sigma_j u)^2\dd s
&\le b \alpha\delta^2 \Big(N^V_j[\sigma_j u,\sigma_j u]+ (\alpha^2-A\alpha) \|\sigma_j u\|^2_{L^2(V_{j,\delta})}
\Big)\\
&\quad + b\alpha^3\delta^2 e^{-c\alpha\delta} \|u\|^2_{L^2(\Omega)}.
\end{align*}
\end{lemma}

We are now able to obtain an analog of Proposition~\ref{DtN} for curvilinear polygons. 
\begin{coro}\label{corol1}
For any fixed $n\in\NN$ and $\alpha\to+\infty$ there holds
\begin{equation}
   \label{bra002}
E_n(\boplus\nolimits_{j=1}^M N^W_j)	\le E_{K+n}(R^\Omega_\alpha)\le E_n(\boplus\nolimits_{j=1}^M D^W_j),
\end{equation}
in particular,
\begin{multline*}
-\alpha^2+E_n\Big( \boplus_{j\in\cJ_*}(N_{j,\delta}-\alpha H_j)\Big)+\cO(1)\le E_{K+n}(R^\Omega_\alpha)\\
\le-\alpha^2+E_n\Big( \boplus_{j\in\cJ_*}(D_{j,\delta}-\alpha H_j)\Big)+\cO(1).
\end{multline*}
\end{coro}

\begin{proof}
The standard Dirichlet-Neumannn bracketing gives
\begin{multline}
   \label{bra001}
E_n\big( N^c \oplus (\boplus\nolimits_{j=1}^M N^V_j)\oplus (\boplus\nolimits_{j=1}^M N^W_j) \big)\le
E_n(R^\Omega_\alpha)\\
\le E_n\big( (\boplus\nolimits_{j=1}^M D^V_j)\oplus (\boplus\nolimits_{j=1}^M D^W_j) \big).
\end{multline}
In the asymptotic regime \eqref{eq-daa2}, thanks to Corollary \ref{lem-trunc5}, we have
\begin{align*}
E_{n}(\boplus\nolimits_{j=1}^M D^V_j)&=\cE_n\alpha^2+o(\alpha^2),\\
 E_{n}(\boplus\nolimits_{j=1}^M N^V_j)&=\cE_n\alpha^2+o(\alpha^2) \text { for $n\in\{1,\dots,K\}$},
\end{align*}
while for any $A>0$ and some $c>0$ one has, by Corollary~\ref{lem-nonres2},
\[
E_{K+1}(\boplus\nolimits_{j=1}^M D^V_j)\ge E_{K+1}(\boplus\nolimits_{j=1}^M N^V_j)\ge -\alpha^2+A\alpha + c/\delta^2.
\]
By Proposition~\ref{prop-pp}, for each $n\in \NN$ we have $E_n(D^W_j)=-\alpha^2+\cO(\alpha)$ and $E_n(N^W_j)=-\alpha^2+\cO(\alpha)$, 
hence 
\begin{align*}
E_K(\boplus\nolimits_{j=1}^M N^V_j)&\le E_n(N^W_j)=-\alpha^2+\cO(\alpha)\le E_{K+1}(\boplus\nolimits_{j=1}^M N^V_j), \\
E_K(\boplus\nolimits_{j=1}^M D^V_j)&\le E_n(D^W_j)=-\alpha^2+\cO(\alpha)\le E_{K+1}(\boplus\nolimits_{j=1}^M D^V_j).
\end{align*}
It follows that for any $n\in\NN$ one has
\begin{align*}
E_{K+n}\big( (\boplus\nolimits_{j=1}^M D^V_j)\oplus (\boplus\nolimits_{j=1}^M D^W_j) \big)&=E_{n}(\boplus\nolimits_{j=1}^M D^W_j),\\
E_{K+n}\big(  N^c\oplus (\boplus\nolimits_{j=1}^M N^V_j)\oplus (\boplus\nolimits_{j=1}^M N^W_j) \big)&={E_{n}(\boplus\nolimits_{j=1}^M N^W_j)},
\end{align*}
and \eqref{bra001} implies \eqref{bra002}. Now it is sufficient to apply Proposition~\ref{prop-pp} to each of the operators in the direct sums. In particular, due to \eqref{eq-eend2}
only $j\in\cJ_*$ contribute to the asymptotics of the individual eigenvalues.
\end{proof}

\subsection{Curvatures taking their maxima away from corners: Proof of Theorem~\ref{thm2}}\label{ssec-thm2}

Remark that the estimate of Corollary~\ref{corol1} only gives a rough asymptotics in general, as there is a discrepancy between the lower and upper bounds due to the  different boundary conditions (Neumann/Dirichlet). In the particular case of
constant curvatures one obtains the same asymptotics for each individual eigenvalue.
We remark nevertheless that the discrepancy can be very small under suitable geometric assumptions. The considerations of the present subsection
rely on some general and well-known ideas of the semiclassical analysis (see e.g. Helffer's monograph \cite{bhbook}), so we only give a sketch of the proofs.

Namely,
in the present subsection we consider the case when the maximum curvature $H_*$ is not attained at any corner:
\begin{equation}
   \label{eq-hhh0}
\text{for all $j\in\{1,\dots,M\}$ there holds $H_j(0)\ne H_*$ and $H_j(\ell_j)\ne H_j$.}
\end{equation}
The analysis of both $E_n(N_{j,\delta}-\alpha H_j)$ and $E_n(D_{j,\delta}-\alpha H_j)$ is covered by the standard framework of the semiclassical analysis:
the eigenfunctions for the lowest eigenvalues are exponentially localized near the set at which
$H_j$ takes its maximum value, and the boundary conditions only influence the eigenvalue asymptotics in exponentially small terms,
see \cite[\S 3]{bhbook}. One obtains then the following assertion:
\begin{prop}\label{prop-agmon}
Under the assumption \eqref{eq-hhh0}, there exists $c>0$ such that for each $j\in\cJ_*$ and each $n\in\NN$ the following estimates hold in the asymptotic regime \eqref{eq-daa2}:
\begin{align*}
E_n(D_{j,\delta}-\alpha H_j)&=E_n(D_j-\alpha H_j)+\cO(e^{-c\sqrt{\alpha}}),\\
E_n(N_{j,\delta}-\alpha H_j)&=E_n(N_j-\alpha H_j)+\cO(e^{-c\sqrt{\alpha}}),\\
E_n(D_j-\alpha H_j)-E_n(N_j-\alpha H_j)&=\cO(e^{-c\sqrt{\alpha}}).
\end{align*}
\end{prop}

By combining Corollary~\ref{corol1} with Proposition~\ref{prop-agmon}  one obtains
then the following main result:

\begin{prop}\label{maxcurv}
Let $\Omega$ be a curvilinear polygon whose half-angles satisfy \eqref{cond-nonres}
and the side curvatures satisfy \eqref{eq-hhh0}. Then for any fixed $n\in\NN$ one has
\begin{align*}
E_{K+n}(R^\Omega_\alpha)&=-\alpha^2+E_n\big( \boplus\nolimits_{j\in\cJ_*}(D_j-\alpha H_j)\big)+\cO(1)\\
& = -\alpha^2+E_n\big( \boplus\nolimits_{j\in\cJ_*}(N_j-\alpha H_j)\big)+\cO(1).
\end{align*}
\end{prop}

We finally note that the analysis can be made more precise under additional geometric assumptions.
In particular, the construction of Helffer-Kachmar \cite{HK} can be easily adapted
in order to obtain the following result:

\begin{prop}\label{thaa1}
Let $\Omega$ be a curvilinear polygon whose half-angles satisfy \eqref{cond-nonres}.
Assume that there exists a unique $k\in\{1,\dots,M\}$
and a unique $s_*\in (0,\ell_k)$ such that $H_k(s_*)=H_*$ and $h_*:=-H''_j(s_*)>0$ and that $\gamma_k$ is $C^\infty$ in a neighborhood of $s_*$, then,  for any $n\in\NN$ there exists a sequence $(\beta_{i,n})_{i\ge0}$ such that
for any $m\in\NN$ there holds
\[
E_{K+n}(R^\Omega_\alpha)=-\alpha^2-H_*\alpha + (2n-1)\sqrt{\dfrac{h_*}{2}}\sqrt{\alpha}
+\sum_{i=0}^m \beta_{i,n} \alpha^{-\frac{i}{2}} + o(\alpha^{-\frac{m}{2}}).
\]
\end{prop}

For the proof one remarks first that due to \eqref{eq-eend3} one has
\begin{gather*}
E_n(\boplus\nolimits_{j=1}^M N^W_j)=E_n(N^W_k),
\quad
E_n(\boplus\nolimits_{j=1}^M D^W_j)=E_n(D^W_k),
\end{gather*}
and then $E_n(N^W_k)\le E_{K+n}(R^\Omega_\alpha)\le E_n(D^W_k)$ by \eqref{bra002}.
The eigenvalues of $N^W_k$ and $D^W_k$ are then analyzed literally
as in \cite[Thm.~1.1]{HK}, since all the analysis is done
in a small neighborhood of $\gamma_k(s_*)$.
Remark that the first three terms can also be deduced directly from
Proposition~\ref{maxcurv} by applying the standard WKB analysis to the 1D operators
$N_k-\alpha H_k$ and $D_k-\alpha H_k$, \cite[Sec.~3]{DS}.

\subsection{Constant curvatures: Proof of Theorem~\ref{thm3}}\label{ssec-thm3}

The main assumption in the present subsection is as follows:
\begin{equation}
   \label{hconst}
\text{ each function $H_j$ is constant, and } H_*:=\max\nolimits_{j=1}^M H_j,
\end{equation}
and we recall that the we still assume the non-resonance condition~\eqref{cond-nonres},
and that in the beginning of the section we introduced the diffeomorphisms $\Phi_j$
by $\Phi_j(s,t):=\gamma_j(s)-t\nu_j(s)$ and the open sets
\[
W_{j,\delta}:=\Phi_j(\Pi_{j,\delta}),
\quad \Pi_{j,\delta}=I_{j,\delta}\times(0,\delta).
\]

%\begin{figure}
%\centering
%
%%\includegraphics[width=90mm]{tube1.eps}
%
%\caption{The curvilinear rectangle $W_\delta$ is shaded. The boundary part $\partial_*W_\delta\equiv\Gamma_\delta$ is shown by the thick solid line,
%the part $\partial_\ext W_\delta$ is shown by the thick dashed line, and $\partial_\out W_\delta$ is marked by the thin dashed line.\label{fig-tube1}}
%
%\end{figure}

We will need some constructions in curvilinear coordinates in $W_{j,\delta}$. 
In order to avoid the use of special functions we prefer to use tubular coordinates instead of polar coordinates. The following lemma is obtained by direct computations using the standard change of variables.

\begin{lemma}\label{change3}
Consider the unitary transform
\begin{equation}
   \label{eq-ggj}
G_j:L^2(W_{j,\delta})\to L^2(\Pi_{j,\delta}), \quad (G_j u)(s,t)=(1-tH_j)^{\frac{1}{2}}u\big(\Phi_j(s,t)\big),
\end{equation}
then $u\in H^1(W_{j,\delta})$ if and only if $g:=G_j u\in H^1(\Pi_{j,\delta})$,
and there exists $b>0$  such that for sufficiently small $\delta$, all  $u$ and $g$ as above and all $\alpha>0$ one has the two-sided estimate
\begin{align*}
B_-[g,g]&\le\int_{W_{j,\delta}} |\nabla u|^2dx-\alpha\int_{\partial_* W_{j,\delta}} u^2\dd s\le B_+[g,g],\\
B_\pm[g,g]&:=\int_{I_{j,\delta}}\int_0^\delta \Big[ \big(1 \pm b\delta\big) \Big( \dfrac{\partial g}{\partial s}\Big)^2 +\Big( \dfrac{\partial g}{\partial t}\Big)^2-\Big( \dfrac{H^2_j}{4} \mp b\delta\Big) g^2 \Big]\dd t\dd s\\
&\qquad {}-\int_{I_{j,\delta}} \Big(\alpha + \dfrac{H_j}{2}\Big) g(s,0)^2\dd s
\pm b \int_{I_{j,\delta}} g(s,\delta)^2\dd s
\end{align*}
\end{lemma}
The above change of variables will be now used for some constructions involving $(D/N)^W_j$.
We start with an eigenvalue estimate for $D^W_j$:

\begin{lemma}\label{wd-dirdir}
One can find $b>0$ such that for $\delta\to 0^+$ and $\alpha\delta\to +\infty$ there holds
\[
E_n(D^W_j)\le -\alpha^2-\alpha H_j -\dfrac{H^2_j}{2} +(1+b\delta)E_n(D_j)
+ b (\delta+\alpha^2 e^{-\alpha\delta}),\quad n\in\NN.
\]
\end{lemma}

\begin{proof} Due to Lemma~\ref{change3}, for some $b_0>0$ one has $E_n(D^W)\le E_n(B_+)$,
where $B_+$ is the self-adjoint operator in $H^1(\Pi_{j,\delta})$ with
\begin{align*}
\cQ(B_+)&=\big\{ g\in H^1\big(I_{j,\delta}\times(0,\delta)\big):\  g(\cdot,\delta)=0,\\
&\qquad  g(\iota,\cdot)=0 \text{ for each } \iota\in\partial I_{j,\delta}\big\},\\
B_+[g,g]&=\int_{I_{j,\delta}}\int_0^\delta \Big[ \big(1 + b_0\delta\big)
( g_s)^2 +(g_t)^2%\\&\quad {}
-\Big( \dfrac{H_j^2}{4} - b_0\delta\Big) g^2 \Big]\dd t\dd s\\
&\quad -\int_{I_{j,\delta}} \Big(\alpha + \dfrac{H_j}{2}\Big) g(s,0)^2\dd s,
\end{align*}
where we have set $g_s := \partial g/\partial s$ and $g_t := \partial g/\partial t$.
As $H_j$ is constant, the operator $B_+$ admits a separation of variables and is unitarily equivalent to
\[
C_+:=(1+b_0\delta)D_{j,\delta} \otimes 1 + 1\otimes L_D-( H^2_j/4- b_0\delta),
\] where $L_D$ is the Laplacian on $(0,\delta)$ with $(\alpha + H_j/2)$-Robin condition
at $0$ and Dirichlet condition at $\delta$, so using Proposition~\ref{prop21}
with some $b_1>0$ we have
$E_1(L_D)=-(\alpha + H_j/2)^2+ b_1 \alpha^2 e^{-\alpha\delta}$ and $E_2(L_D)\ge 0$.
Therefore, for each fixed $n\in\NN$ due to $E_n(D_{j,\delta})=\cO(1)$ we have
\begin{align*}
E_n(D^W_j)&\le E_n(B_+)=E_n(C_+)\\
&=E_1(L_D)+ (1+b_0\delta)E_n(D_{j,\delta})
-\Big( \dfrac{H^2_j}{4} - b_0\delta\Big),\\
&\le -\alpha^2-\alpha H_j-\dfrac{H^2_j}{2} +(1+b_0\delta)E_n(D_{j,\delta}) + b_0\delta + b_1 \alpha^2 e^{\alpha\delta}.
\end{align*}
One arrives at the sought result by taking $b:=\max\{b_0,b_1\}$ and by noting that
$E_n(D_{j,\delta})=E_n(D_j)+\cO(\delta)$ for any fixed $n\in\NN$.
\end{proof}

By applying the estimate of Lemma~\ref{wd-dirdir} to each operator
in the right-hand side of \eqref{bra002} with $\delta\coloneqq(3\log \alpha)/\alpha$ so that \eqref{eq-daa2} is satisfied, we arrive to an improved upper bound
for $E_{K+n}(R^\Omega_\alpha)$:

\begin{prop}\label{prop-up-rob}
Under the assumptions \eqref{cond-nonres} and \eqref{hconst},
for any fixed $n\in\NN$ there holds
\[
E_{K+n}(R^\Omega_\alpha)\leq-\alpha^2-H_*\alpha-\dfrac{H_*^2}{2}+E_n\Big( \boplus_{j\in \cJ_*} D_j\Big)+\cO\Big(\dfrac{\log \alpha}{\alpha}\Big)
\quad \text{as }\alpha\to+\infty.
\]
\end{prop}

Now we will bring the Robin-Neumann Laplacian $N^W_j$ to a special form
to use during subsequent proofs:

\begin{lemma}\label{wneum}
Let $G_j$ be defined by \eqref{eq-ggj}. There are
functions $\psi_j\in L^2(0,\delta)$ with $\lVert\psi \rVert^2_{L^2(0,\delta)}=1$ such that if one defines
the map
\begin{align}
\label{projection}
P_j: L^2(W_{j,\delta})\to L^2(I_{j,\delta}), \quad (P_j u)(s):=\int_0^\delta \psi_j(t)(G_j u)(s,t)\dd t,
\end{align}
then one has for $\delta\to 0^+$ and $\alpha\delta \to +\infty$, with some  $b>0$, 
\begin{equation}
  \label{lala1}
\begin{aligned}
N^W_j[u,u]&\ge 
-\Big(\alpha^2+\alpha H_j+\dfrac{H^2_j}{2}\Big) \|P_ju\|^2_{L^2(I_{j,\delta})}
+
\big(1 - b\delta\big)\big\|(P_ju)'\big\|^2_{L^2(I_{j,\delta})}\\
&\quad -b\big(\delta+\alpha^2e^{-\alpha\delta}\big)\|P_ju\|^2_{L^2(I_{j,\delta})}
\text{ for all } u\in H^1(W_{j,\delta}).
\end{aligned}
\end{equation}
In particular, 
\begin{equation}
   \label{ineq0}
E_1(N^W_j)\ge -\Big(\alpha^2+\alpha H_j+\dfrac{H^2_j}{2}\Big) + \cO(\delta+\alpha^2e^{-c\alpha\delta}).
\end{equation}
\end{lemma}

\begin{proof}
We  drop the index $j$ in the notation.
Denote $g:=G u\in L^2(\Pi_\delta)$, then due to the standard change of variables (Lemma~\ref{change3})
one can find $b_0>0$ and $\beta>0$
to have, for all $u\in H^1(W_\delta)$,
\begin{multline*}
N^W[u,u]\ge B_-[g,g]:=\int_{I_\delta}\int_0^\delta \Big[ \big(1 - b_0\delta\big) g_s^2 +g_t^2\\
\quad {}-\Big( \dfrac{H^2}{4} + b_0\delta\Big) g^2 \Big]\dd t\dd s -\int_{I_{\delta}} \Big(\alpha + \dfrac{H}{2}\Big) g(s,0)^2\dd s
- \beta \int_{I_{\delta}} g(s,\delta)^2\dd s,
\end{multline*}
Denote by $L_N$ the one-dimensional Laplacian in $(0,\delta)$ with the $(\alpha+H/2)$-Robin boundary condition at $0$ and the $\beta$-Robin boundary condition at $\delta$, and let $\psi$ be its eigenfunction for the first eigenvalue, normalized by $\lVert \psi \rVert_{L^2(0,\delta)}=1$. With this choice of $\psi$, define the map $P$ as in \eqref{projection}. 
For shortness we denote $f:=Pu$ and define $z\in L^2(\Pi_\delta)$ by $z(s,t)\coloneqq g(s,t)-f(s)\psi(t)$,
then, with $z_s := \partial z/\partial s$, we have the identities
\begin{gather}
  \label{orto1}
\int_0^\delta \psi(t) z(\cdot,t)\dd t=0,
\quad
\int_0^\delta \psi(t) z_s(\cdot,t)\dd t=0,\\
\|u\|^2_{L^2(W_\delta)}=\|g\|^2_{L^2(\Pi_\delta)}=\|f\|^2_{L^2(I_\delta)}+\|z\|_{L^2(\Pi_\delta)}^2, \nonumber
\end{gather}
and due to the spectral theorem for the operator $L_N$
there holds
\begin{multline*}
\int_{I_\delta} \int_0^\delta  g_t^2\dd t \dd s -\int_{I_{\delta}} \Big(\alpha + \dfrac{H}{2}\Big) g(s,0)^2\dd s -\beta\int_{I_\delta}g(s,\delta)^2\dd s\\
\ge \int_{I_\delta} \int_0^\delta \Big( E_1(L_N) f(s)^2\psi(t)^2 + E_2(L_N) z(s,t)^2\Big)\dd t \dd s\\
= E_1(L_N) \|f\|^2_{L^2(I_\delta)}+E_2(L_N)\|z\|^2_{L^2(\Pi_\delta)},
\end{multline*}
an using the second equality in \eqref{orto1} we also have
\[
\int_{I_\delta} \int_0^\delta g_s^2\dd t \dd s
= \|f'\|^2_{L^2(I_\delta)} + \|z_s\|^2_{L^2(\Pi_\delta)}\ge \|f'\|^2_{L^2(I_\delta)}.
\]
Therefore, 
\begin{multline*}
B_-[g,g]\ge (1-b_0\delta)\|f'\|^2_{L^2(I_\delta)}+\Big(E_1(L_N) -\dfrac{H^2}{4} - b_0\delta\Big)\|f\|^2_{L^2(I_\delta)}\\
+\Big(E_2(L_N)-\dfrac{H^2}{4} - b_0\delta\Big)\|z\|^2_{L^2(\Pi_\delta)}.
\end{multline*}
Using Proposition~\ref{prop22} in order to estimate the eigenvalues of $L_N$ one has then, with a suitable $a_0>0$,
\begin{align*}
E_1(L_N) -\dfrac{H^2}{4} - b_0\delta&=-\Big(\alpha+\dfrac{H}{2}\Big)^2-a_0 \alpha^2e^{-\alpha\delta}-\dfrac{H^2}{4}-b_0\delta\\
&\ge -\alpha^2-\alpha H-\dfrac{H^2}{2} -a_1\big(\delta+ \alpha^2e^{-\alpha\delta}),\\
&\qquad a_1:=\max\{a_0,b_0\},\\
E_2(L_N)-\dfrac{H^2}{4} - b_0\delta&\ge \dfrac{1}{\delta^2}-\dfrac{H^2}{4} - b_0\delta\ge 0,
\end{align*}
and then
\begin{multline*}
B_-[g,g]\ge (1-b_0\delta)\|f'\|^2_{L^2(I_\delta)}-{\Big(\alpha^2+\alpha H+\dfrac{H^2}{2}\Big)}\|f\|^2_{L^2(I_\delta)}\\
{}-a_1\big(\delta+ \alpha^2e^{-\alpha\delta})\|f\|^2_{L^2(I_\delta)}.
\end{multline*}
Hence, one arrives at the sought inequality \eqref{lala1} by taking $b:=\max\{b_0,a_1\}$.
To prove the lower bound \eqref{ineq0} it remains to use the inequality $\|f\|_{L^2(I_\delta)}\le\|u\|_{L^2(W_\delta)}$.
\end{proof}

As in the straight case we will obtain the sought lower bound for the eigenvalues
$E_{K+n}(R^\Omega_\alpha)$  with the help of the Proposition~\ref{prop6} by constructing a suitable identification map.
Nevertheless, the construction involves a number of new components, so we prefer to 
give a sketch. 

\begin{prop}\label{prop-low-rob}
Under the assumptions \eqref{cond-nonres} and  \eqref{hconst} and 
for any fixed $n\in\NN$ one has, as $\alpha\to+\infty$,
\[
E_{K+n}(R^\Omega_\alpha)\ge -\alpha^2-H_*\alpha-\dfrac{H_*^2}{2} + E_n\Big(\bigoplus_{j\in\cJ_*} D_j\Big) + \cO \Big( \dfrac{\log\alpha}{\sqrt{\alpha}}\Big).
\]
\end{prop}

\begin{proof}
With the above preparations and with a suitable redefinition of the main objects,
the proof becomes almost identical to the one of Proposition~\ref{prop-low-rob1}. We are not giving all details, but just introducing the main objects and identifying the main steps.

Assume first that $\delta$ satisfies \eqref{eq-daa2}
and consider the Hilbert spaces
\[
\cH:=\text{ the orthogonal complement of $L$ in $L^2(\Omega)$},
\quad
\cH':=\bigoplus\nolimits_{j\in\cJ_*}L^2(I_{j,\delta}).
\]
During the proof for $u\in \cH$ we denote
\begin{align*}
v_j&:=\text{ the restriction of $u$ to $V_{j,\delta}$},& \|v_j\|&:=\|v_j\|_{L^2(V_{j,\delta})},\\
w_j&:=\text{ the restriction of $u$ to $W_{j,\delta}$},& \|w_j\|&:=\|w_j\|_{L^2(W_{j,\delta})},\\
u_c&:=\text{ the restriction of $u$ to $\Omega^c_\delta$},& \|u_c\|&:=\|u_c\|_{L^2(\Omega^c_\delta)},
\end{align*}
and remark that due to the constructions and the equality~\eqref{eq-vjwj} we have
\begin{equation}
  \label{eq-uvw2}
\sum_{j=1}^M \int_{\partial_\ext V_{j,\delta}} v_j^2\dd s=\sum_{j=1}^M \int_{\partial_\ext W_{j,\delta}} w_j^2\dd s.
\end{equation}
Applying Lemma~\ref{vertices2} with $A:=-H_j$ we obtain, with some $b>0$ and $c>0$,
the inequalities
\begin{align*}
\|v_j\|^2
&\le b\delta^2 \big(N^V_j[v_j,v_j]+ (\alpha^2+H_j\alpha) \|v_j\|^2\big) + b\alpha^2\delta^2e^{-c\alpha\delta} \|u\|^2,\\
\int_{\partial_\ext V_{j,\delta}} v_j^2 \dd s
&\le b \alpha\delta^2 \big(N^V_j[v_j,v_j]+ (\alpha^2+H_j\alpha) \|v_j\|^2\big)
+ b\alpha^3\delta^2 e^{-c\alpha\delta} \|u\|^2.
\end{align*}
Furthermore, by applying Lemma~\ref{wneum} to each $W_{j,\delta}$
we conclude that there are functions $\psi_j\in L^2(0,\delta)$ with $\|\psi_j\|^2_{L^2(0,\delta)}=1$
such that if one defines
\[
 P_j:\cH\to L^2(I_{j,\delta}), \quad
(P_j u)(s):=\int_0^\delta \psi_j(t) \sqrt{1-H_j t} \,w_j\big(\Phi_j(s,t)\big)\dd t,
\]
then, with some $b_1>0$,
\begin{align*}
N^W_j[w_j,w_j]&\ge 
-\Big(\alpha^2+\alpha H_j+\dfrac{H^2_j}{2}\Big) \|P_j u\|^2
+
\big(1 - b_1\delta\big)\big\|(P_j u)'\big\|^2\\
&\quad -b_1\big(\delta+\alpha^2e^{-c\alpha\delta}\big)\|P_j u\|^2,
\end{align*}
and we recall that, using the Cauchy-Schwarz inequality and $\|\psi_j\|_{L^2(0,\delta)}=1$,
\begin{multline}
   \label{pj001}
\|P_j u\|^2=\int_{I_{j,\delta}}\Big( \int_0^\delta \psi(t) \sqrt{1-H_j t} \,w_j\big(\Phi_j(s,t)\big)\dd t\Big)^2\dd s\\
\le \int_{I_{j,\delta}} \int_0^\delta (1-H_j t)w_j\big(\Phi_j(s,t)\big)^2\dd t \dd s=\int_{W_{j,\delta}} w_j^2\dd x=\|w_j\|^2.
\end{multline}

Now let us set $\delta:=(c'\log\alpha)/\alpha$ with $c'\ge 3/c$, then the conditions \eqref{eq-daa2} for the choice
of $\delta$ are satisfied, and $\alpha^2e^{-c\alpha\delta}=o(\delta)$, which implies  $\alpha^2\delta^2e^{-c\alpha\delta}=o(\delta^3)$ and  $\alpha^3\delta^2 e^{-c\alpha\delta}=o(\alpha\delta^3)$.
This simplifies the remainders in the above inequalities, and
one can pick a sufficiently large $a>0$ such that, for the same choice of $\psi_j$,
\begin{align*}
   \|v_j\|^2
&\le \dfrac{a\log^2\alpha}{\alpha^2} \Big(N^V_j[v_j,v_j]+ (\alpha^2+H_j\alpha) \|v_j\|^2\Big) + \dfrac{a\log^3\alpha}{\alpha^3} \|u\|^2,\\
   \label{nv2}
\int_{\partial_\ext V_{j,\delta}} v_j^2 \dd s
&\le \dfrac{a\log^2\alpha}{\alpha} \Big(N^V_j[v_j,v_j]+ (\alpha^2+H_j\alpha) \|v_j\|^2\Big)
+ \dfrac{a\log^3\alpha}{\alpha^2} \|u\|^2,\\
N^W_j[w_j,w_j]&\ge -\Big(\alpha^2+\alpha H_j+\dfrac{H^2_j}{2}\Big) \|P_j u\|^2 + \Big(1 -\dfrac{a\log\alpha}{\alpha}\Big)\big\|(P_j u)'\big\|^2\\
&\quad-\dfrac{a\log\alpha}{\alpha}\,\|P_j u\|^2.
\end{align*}
Consider the self-adjoint operators
\begin{align*}
B&:= R^\Omega_\alpha +\Big(\alpha^2+\alpha H_*+\dfrac{H^2_*}{2}\Big) +\dfrac{(M+a)\log\alpha}{\alpha}\\
&\quad \text{ viewed as an operator in $\cH$,}\\
B'&:=\boplus\nolimits_{j\in\cJ_*} D_{j,\delta} \text{ in } \cH',
\end{align*}
with
$\qdom(B)=H^1(\Omega)\mathop{\cap} \cH$ and $\qdom(B')=\boplus\nolimits_{j\in\cJ_*} H^1_0(I_{j,\delta})$.
By combining the preceding estimates as in the proof of Proposition~\ref{prop-low-rob1} one arrives at the estimate
\begin{equation}
    \label{eq-buu2a}
\begin{aligned}
B[u,u]&\ge \dfrac{\alpha^2}{2a\log^2\alpha}\sum\nolimits_{j=1}^M \|v_j\|^2 + 
\,\dfrac{\alpha}{2a\log^2\alpha} \sum\nolimits_{j=1}^M \int_{\partial_\ext V_{j,\delta}} v_j^2\dd s\\
&\quad +a_0\alpha \sum\nolimits_{j\notin\cJ_*} \|P_j u\|^2+ \Big(1-\dfrac{a \log\alpha}{\alpha}\Big)\sum\nolimits_{j\in\cJ_*}\big\|(P_j u)'\big\|^2\\
&\quad +\dfrac{\alpha^2}{2}\sum\nolimits_{j=1}^M\big( \|w_j\|^2-\|P_j u\|^2\big)+\dfrac{\alpha^2}{2} \|u_c\|^2;
\end{aligned}
\end{equation}
the new summand $\sum\nolimits_{j\notin\cJ_*} \|P_j u\|^2$ is due to the modified definition of $B'$.
Each term on the right-hand side is non-negative and, hence, the left-hand side is an upper bound for each term on the right-hand side.
It also implies that $B$ is positive and then
\begin{equation*}
    %\label{e1b}
E_1(B)\equiv E_{K+1}(R^\Omega_\alpha)+\Big(\alpha^2+\alpha H_*+\dfrac{H^2_*}{2}\Big) +\dfrac{(M+a)\log\alpha}{\alpha}\ge 0.
\end{equation*}
By combining with Proposition~\ref{prop-up-rob} we see that
for any fixed $n\in\mathbb N$ one can choose $\lambda_n>0$  which is independent of $\alpha$ and such that
\begin{equation}
   \label{hyp1a}
0\le E_n(B)\le \lambda_n, \quad \big( 1+ E_n(B)\big)^{-1} \ge (1+\lambda_n)^{-1}.
\end{equation}

In order to construct a suitable identification map $J:\qdom(B)\to \qdom(B')$ we pick functions
$\rho^\pm_j\in C^1\big([0,\ell_j]\big)$ such that
\begin{gather*}
\rho^+_j=\begin{cases}
1 & \text{ in a neighborhood of $0$},\\
0 & \text{ in a neighborhood of $\ell_j$},
\end{cases}
\quad j\in \cJ_*,\\
\rho^-_j=\begin{cases}
0 & \text{ in a neighborhood of $0$},\\
1 & \text{ in a neighborhood of $\ell_j$},
\end{cases}
\quad j\in \cJ_*,
\end{gather*}
and then choose a constant $\rho_0>0$ such that
\begin{equation}
   \label{eq-rho0a}
\|\rho^\pm_j\|_{L^\infty(0,\ell_j)}+\|(\rho^\pm_j)'\|_{L^\infty(0,\ell_j)}\le \rho_0 \text{ for all }j\in \cJ_*.
\end{equation}
We have $I_{j,\delta}:=\big(\lambda^+_j(\delta),\ell_j-\lambda^-_{j+1}(\delta)\big)=:(\iota_j,\tau_j)$,
and that due to $\lambda^\pm_j(\delta)=\cO(\delta)$ we have $\iota_j=\cO(\delta)$ and $\tau_j=\ell_j+\cO(\delta)$, hence,
\[
\rho^+(\iota_j)=1, \quad \rho^+(\tau_j)=0, \quad
\rho^-(\iota_j)=0, \quad \rho^-(\tau_j)=1
\]
as $\alpha$ is sufficiently large. Therefore, the following map is well-defined:
\begin{gather*}
J:\qdom(B)\to \qdom(B')\equiv \boplus\nolimits_{j\in\cJ_*} H^1_0(I_{j,\delta}),
\quad J u= (J_j u),\\
(J_j u)(s):=(P_j u)(s)-(P_j u)(\iota_j)\rho^+_j(s) - (P_j u)(\tau_j)\rho^-_j(s)
\end{gather*}
For large $\alpha$ one has $1-H_j t\le 2$ for $t\in(0,\delta)$, therefore, by Cauchy-Schwarz,
\begin{multline}
   \label{pj002a}
\big|P_j u(\iota_j)\big|^2+\big|P_j u(\tau_j)\big|^2\\
\begin{aligned}
&=\Big( \int_0^\delta \psi_j(t) \sqrt{1-H_j t} \,w_j\big(\Phi_j(\iota_j,t)\big)\dd t\Big)^2\\
&\quad +\Big( \int_0^\delta \psi_j(t) \sqrt{1-H_j t} \,w_j\big(\Phi_j(\tau_j,t)\big)\dd t\Big)^2\\
&\le \int_0^\delta (1-H_j t)w_j\big(\Phi_j(\iota_j,t)\big)^2\, dt + \int_0^\delta (1-H_j t)w_j\big(\Phi_j(\tau_j,t)\big)^2\dd t\\
&\le 2\Big( \int_0^\delta w_j\big(\Phi_j(\iota_j,t)\big)^2\, dt+\int_0^\delta w_j\big(\Phi_j(\tau_j,t)\big)^2\dd t
\Big)\\
&\equiv 2\int_{\partial_\ext W_{j,\delta}} w_j^2\dd s.
\end{aligned}
\end{multline}
Using $(x+y)^2\ge(1-\varepsilon)x^2-y^2/\varepsilon$ for any $x,y\in\RR$ and $\varepsilon>0$ we estimate
\begin{align*}
\|J_j u\|^2=&\int_{I_{j,\delta}} \Big|(P_j u)(s)-(P_j u)(\iota_j)\rho^+_j(s) - (P_j u)(\tau_j)\rho^-_j(s)\Big|^2\dd s\\
&\ge (1-\varepsilon)\int_{I_{j,\delta}} \Big|(P_j u)(s)\Big|^2\dd s\\
&\quad-\dfrac{1}{\varepsilon}\int_{I_{j,\delta}}\Big|(P_j u)(\iota_j)\rho^+_j(s) + (P_j u)(\tau_j)\rho^-_j(s)\Big|^2\dd s.
\end{align*}
Using \eqref{pj002a} and the constant $\rho_0$ from \eqref{eq-rho0a} we have
\begin{gather*}
\int_{I_{j,\delta}}\Big|(P_j u)(\iota_j)\rho^+_j(s) + (P_j u)(\tau_j)\rho^-_j(s)\Big|^2\dd s\le 4 \ell_j \rho_0^2 \int_{\partial_\ext W_{j,\delta}} w_j^2\dd s,\\
\|J_j u\|^2\ge (1-\varepsilon)\|P_j u\|^2-\dfrac{4 \ell \rho_0^2}{\varepsilon}\int_{\partial_\ext W_{j,\delta}} w_j^2\dd s,
\quad \ell:=\max_{j\in\cJ_*} \ell_j.
\end{gather*}
Therefore, using \eqref{eq-buu2a},
\begin{align*}
\|u\|^2-\|Ju\|^2&\le \sum_{j=1}^M \|v_j\|^2+\sum_{j=1}^M\big( \|w_j\|^2- \|P_j u\|^2 \big)+
\sum_{j\notin \cJ_*}\|P_j u\|^2\\
&\quad + \varepsilon \|u\|^2 + \dfrac{4 \ell \rho_0^2}{\varepsilon}\sum_{j=1}^M\int_{\partial_\ext V_{j,\delta}} v_j^2\dd s+\|u_c\|^2\\
&\le \Big(\dfrac{2a\log^2\alpha}{\alpha^2}+\dfrac{4}{\alpha^2}+\dfrac{1}{a_0\alpha}+\dfrac{8 \ell \rho_0^2 a\log^2\alpha}{\varepsilon \alpha}\Big) B[u,u]+\varepsilon \|u\|^2.
\end{align*}
Taking $\varepsilon:=\log\alpha/\sqrt{\alpha}$ and choosing $c_1>0$ sufficiently large we obtain
\begin{equation}
    \label{eps1a}
\|u\|^2-\|Ju\|^2\le \dfrac{c_1 \log\alpha}{\sqrt{\alpha}}\Big( B[u,u]+\|u\|^2\Big).
\end{equation}

To study the difference $B'[Ju,Ju]-B[u,u]$ recall that $B'[Ju,Ju]=\sum_{j\in\cJ_*} \|(J_ju)'\|^2$.
Using the elementary inequality $(x+y)^2\le(1+\varepsilon)x^2+2y^2/\varepsilon$ valid for all $x,y\in\RR$ and $\varepsilon\in(0,1)$
we estimate
\begin{align*}
\big\|(J_j u)'\big\|^2=&\int_{I_{j,\delta}} \Big|(P_j u)'(s)-(P_j u)(\iota_j)(\rho^+_j)'(s) - (P_j u)(\tau_j)(\rho^-_j)'(s)\Big|^2\dd s\\
&\le (1+\varepsilon)\int_{I_{j,\delta}} \Big|(P_j u)'(s)\Big|^2\dd s\\
&\quad+ \dfrac{2}{\varepsilon}\int_{I_{j,\delta}}\Big|(P_j u)(\iota_j)(\rho^+_j)'(s) + (P_j u)(\tau_j)(\rho^-_j)'(s)\Big|^2\dd s.
\end{align*}
Using the estimate \eqref{pj002a} for the last term and the constant $\rho_0$ from \eqref{eq-rho0a}
we have
\begin{gather*}
\int_{I_{j,\delta}}\Big|(P_j u)(\iota_j)(\rho^+_j)'(s) + (P_j u)(\tau_j)(\rho^-_j)'(s)\Big|^2\dd s\le 4 \ell \rho_0^2 \int_{\partial_\ext W_{j,\delta}} w_j^2\dd s,\\
\begin{aligned}
B'[Ju,Ju]&\le (1+\varepsilon)\sum_{j\in\cJ_*} \big\|(P_ju)'\big\|^2 + \dfrac{8\ell\rho_0^2}{\varepsilon} \sum_{j\in\cJ_*} \int_{\partial_\ext W_{j,\delta}} w_j^2\dd s\\
&\le (1+\varepsilon)\sum_{j\in\cJ_*} \big\|(P_ju)'\big\|^2+\dfrac{8\ell\rho_0^2}{\varepsilon} \sum_{j=1}^M \int_{\partial_\ext W_{j,\delta}} w_j^2\dd s.
\end{aligned}
\end{gather*}
Recall that due to \eqref{eq-buu2a} we have
\begin{align*}
B[u,u]&\ge \Big(1-\dfrac{a\log\alpha}{\alpha}\Big)\sum_{j\in\cJ_*} \big\|(P_ju)'\big\|^2,\\
B[u,u]&\ge \dfrac{\alpha}{2a\log^2\alpha} \sum_{j=1}^M \int_{\partial_\ext V_{j,\delta}} v_j^2\dd s
\equiv \dfrac{\alpha}{2a\log^2\alpha} \sum_{j=1}^M \int_{\partial_\ext W_{j,\delta}} w_j^2\dd s,
\end{align*}
where we used \eqref{eq-uvw2} on the last step. Therefore,
\begin{align*}
B'[Ju,Ju]-B[u,u]&\le \Big(\varepsilon+\dfrac{a\log\alpha}{\alpha}\Big) \sum_{j\in\cJ_*} \big\|(P_ju)'\big\|^2\\
&\quad+\dfrac{8\ell\rho_0^2}{\varepsilon} \sum_{j=1}^M \int_{\partial_\ext W_{j,\delta}} w_j^2\dd s\\
&\le \left(\dfrac{\varepsilon+\dfrac{a\log\alpha}{\alpha}}{1-\dfrac{a\log\alpha}{\alpha}}+\dfrac{8\ell\rho_0^2}{\varepsilon}
\cdot \dfrac{2a\log^2\alpha}{\alpha}\right)\, B[u,u].
\end{align*}
By setting $\varepsilon=\log\alpha/\sqrt\alpha$ and choosing $c_2>0$ sufficiently large
we arrive at
\begin{equation}
    \label{eps2a}
B'[Ju,Ju]-B[u,u]\le \dfrac{c_2 \log\alpha}{\sqrt{\alpha}}\, B[u,u]\le \dfrac{c_2 \log\alpha}{\sqrt{\alpha}}\,\Big(B[u,u]+\|u\|^2\Big).
\end{equation}
In virtue of~\eqref{eps1a} and \eqref{eps2a} we can apply
Proposition~\ref{prop6}. Remark that for each fixed $n$ the assumption
$\varepsilon_1<1/\big(1+E_n(B)\big)$ is satisfied due to~\eqref{hyp1a}.
Hence, for each fixed $n$ there holds
\[
E_n\big(\boplus\nolimits_{j\in\cJ_*} D_{j,\delta}\big)\equiv E_n(B')\le E_n(B) + \dfrac{\log\alpha}{\sqrt{\alpha}}\,\dfrac{\big(c_1 E_n(B)+c_2\big)\big(1+E_n(B)\big)}{1-c_1\big(1+E_n(B)\big)\log\alpha/\sqrt{\alpha}}.
\]
By \eqref{hyp1a} we have $E_n(B)=\cO(1)$ for each fixed $n$,
and the preceding inequality implies
\[
E_{K+n}(R^\Omega_\alpha)\ge -\alpha^2-\alpha H_*-\dfrac{H^2_*}{2}+E_n\Big(\boplus\nolimits_{j\in\cJ_*} D_{j,\delta}\Big)+\cO\Big( \dfrac{\log\alpha}{\sqrt{\alpha}}\Big).
\]
It remains to remark that $E_n\big(\boplus\nolimits_{j\in\cJ_*} D_{j,\delta}\big)=E_n\big(\boplus\nolimits_{j\in\cJ_*} D_j\big)+\cO(\delta)$,
while $\delta=(c'\log\alpha)/\alpha=o\big( \tfrac{\log\alpha}{\sqrt{\alpha}}\big)$.
\end{proof}

The combination of Propositions \ref{prop-up-rob} and \ref{prop-low-rob} gives Theorem~\ref{thm3}.

\section{Concluding remarks}\label{ssec55}

\subsection{Resonant angles: equilateral triangle}\label{rem-triangle}
As already mentioned in the introduction,
we are going to show that there are some angles which do not satisfy the non-resonance condition. This will be done
in an indirect way. First, remark that if $\Omega$ is a convex polygon (with straight sides) with non-resonant vertices, $K$
corner-induced eigenvalues, and side lengths $\ell_j$, then
\begin{equation}
   \label{posit}
\lim_{\alpha\to+\infty}\big(E_{K+1}(R^\Omega_\alpha)+\alpha^2\big)=E_1(\boplus\nolimits_{j=1}^M D_j)\equiv \pi^2/\ell^2>0,
\end{equation}
where $D_j$ is the Dirichlet Laplacian on $(0,\ell_j)$ and $\ell:=\max \ell_j$. Let us show that this can be violated
for some particular polygons $\Omega$ and lead to a different eigenvalue asymptotics.

The paper by McCartin \cite{mcc4} contains a detailed analysis of the operator $R^\Omega_\alpha$ for the case when $\Omega$
is an equilateral triangle using a separation of variables in a suitably
chosen coordinate system. To be more precise, we assume that the side length of the triangle is $1$. Let us give a short account of the results of \cite{mcc4} concerning the behavior
of the eigenvalues as $\alpha\to+\infty$ (which corresponds  to $\sigma\to-\infty$ in the reference).

One constructs first a complete orthogonal system of eigenfunctions, noted $T_s^{m,n}$ with $n\ge m\ge 0$
and $T_a^{m,n}$ with $n>m\ge 0$ and $m,n\in\NN\mathop{\cup}\{0\}$, and $R^\Omega_\alpha T_\circ^{m,n}=E_\alpha(m,n)T_\circ^{m,n}$
for $\circ\in\{s,a\}$, i.e. $T_s^{m,n}$ and $T_a^{m,n}$ share the same eigenvalue for $n>m$.
It is then shown that $E_\alpha(m,n)\ge 0$ for $m\ge 2$, therefore, only $m\in\{0,1\}$
contribute to the negative spectrum. One shows then the following asymptotics
for $\alpha\to+\infty$ (we cite the respective equations in Subsection 7.2 of \cite{mcc4}):
\begin{align*}
E_\alpha(0,0)&=-4\alpha^2+o(1), & \text{Eq. (37)},\\
E_\alpha(0,1)&=-4\alpha^2+o(1), & \text{Eq. (50)},\\
E_\alpha(0,n)&=-\alpha^2+\dfrac{4}{27}\Big[ \dfrac{\pi}{r}\Big(n-\dfrac{3}{2}\Big)\Big]^2+o(1) \text{ for } n\ge 2,& \text{ Eq. (53)},\\
E_\alpha(1,1)&=-\alpha^2+o(1),& \text{Eq. (67)},\\
E_\alpha(1,n)&=-\alpha^2+\dfrac{4}{27}\Big[ \dfrac{\pi}{r}(n-1)\Big]^2+o(1) \text{ for } n\ge 2,& \text{ Eq. (80)},
\end{align*}
where $r:=1/(2\sqrt{3})$ is the inradius. The eigenvalues $E_\alpha(0,0)$ and $E_\alpha(0,1)$ (twice)
are corner-induced: the half-angle at each corner is $\pi/6$, and $\kappa(\pi/6)=1$ (see Subsection~\ref{sec-sectors}), hence $K=3$ (we remark that a more precise remainder
for the first three eigenvalues was obtained in \cite{hp}). Furthermore, by inspecting the above expressions
and by taking into account the multiplicities one sees that for any fixed $n\in\NN$ one has
the asymptotics $E_{K+n}(R^\Omega_\alpha)=-\alpha^2+z_n +o(1)$, where $z_n$ is the $n$th element (when enumerated in the non-decreasing order) of the \emph{multiset}
$Z:=\big\{\big(2\pi m/3\big)^2: \, m\in\ZZ\big\}$.
In particular, one has $z_1=0$ and $E_{K+1}(R^\Omega)=-\alpha^2+o(1)$, which is in contradiction to \eqref{posit}. Hence, the half-angle $\pi/6$
is resonant. In fact,  in the above multiset $Z$ one easily recognizes the spectrum of the Laplacian on a circle of length
$3$, i.e. on the three sides of the triangles glued to each other without any obstacle at the vertices.
This operator can be then viewed as the effective operator on the boundary.

We remark that the text of the paper \cite{mcc4} is included into McCartin's book \cite{mcc-book} as Chapter~7,
but due to a typesetting error some of the important formulas are missing  on page 105 of~\cite{mcc-book},
which complicates the understanding of the eigenvalue asymptotics. An interested reader should better refer to the original paper \cite{mcc4} for full details.

\subsection{Variable curvature}\label{ssec-var}
By analogy with the works on smooth domains, see e.g. \cite{pp15b}
one might expect the following asymptotics to be valid for general curvilinear polygons (i.e. without assuming that $H_j$
are constant): if all corners are concave or convex non-resonant, then
$E_{K+n}(R^\Omega_\alpha)=-\alpha^2+E_n \big(\bigoplus\nolimits_j (D_j-\alpha H_j)\big)+r(\alpha)$
with a suitable error term $r(\alpha)$. Some steps of the above scheme are still easily transferable,
but the whole machinery appears to fail when trying to prove the lower bound.
The main obstacles, when projected to the proof of Proposition~\ref{prop-low-rob}, are
that the eigenvalues of the comparison operator $B'=\bigoplus_j (D_j-\alpha H_j)+\alpha H_*+\rho(\alpha)$ with suitably chosen constants $\rho(\alpha)$ and $H_*:=\max_j \max H_j$,
may become infinitely large for large $\alpha$, and much smaller value of $\varepsilon_j$ are needed to satisfy the initial assumption of Proposition~\ref{prop6}
and to have a non-trivial resulting estimate. In a sense, the machinery we use implicitly aims at showing that the eigenfunctions are suitably small near vertices
by controlling their norms and traces using the values in the rest of the domain (Lemma~\ref{vertices1} and~\ref{vertices2}).
For non-constant curvatures, the eigenfunctions are localized near the points of maximal curvature, similarly as in the smooth case \cite{HK}.  In particular, 
 if the curvature takes its maximum at one of the corners, then the respective eigenfunctions should be localized near the corner, so the strategy of showing that
it asymptotically vanishes at the corners (which then gives an effective operator
with the Dirichlet boundary conditions) becomes contradictory. One might expect that
a more precise analysis in this case can be done under explicit hypotheses on the curvatures
(e.g. an isolated maximum at a corner) by showing first some
semiclassical localization properties for the eigenfunctions,
which might be a task of a higher complexity.

\subsection{Resonance and non-resonance conditions}\label{ssec-nonres2}
Our non-resonance condition introduced in Definition~\ref{defnonres} and used in the proof
is a slightly naive adaptation of a condition appearing in the spectral analysis of Laplacians
on domain collapsing onto a graph. The topic is presented in a systematic way
e.g. in the papers by Grieser \cite{grieser}, Molchanov and Vainberg \cite{mv}, and in the monograph by Post \cite{postbook}.
Let us recall some basic notions of the theory, mostly following the short presentation given in the paper \cite{kp-jmaa} by Pankrashkin.

\begin{figure}[t]

\centering

\includegraphics[height=18mm]{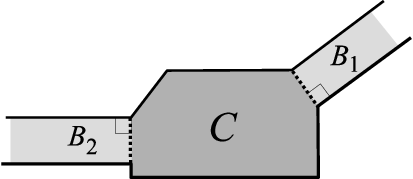}

\caption{Star waveguide $\Lambda$ with two branches and and a dark-shaded center. \label{fig1}}

\end{figure}

Let $d\ge 2$ and $\omega\subset \RR^{d-1}$ be a bounded connected Lipschitz domain. We denote by $\mu$ the first Dirichlet eigenvalue of $\omega$.
By a \emph{star waveguide} we mean a connected Lipschitz domain $\Lambda\subset \RR^d$
for which one can find $n$ non-intersecting half-infinite cylinders $B_1,\dots, B_n \subset \Lambda$,
all isometric to $(0,\infty)\times \omega$, such that $\Lambda$ coincides with the union $B_1\mathop{\cup}\dots \mathop{\cup} B_n$
outside a compact set, see Figure~\ref{fig1}. The cylinders $B_j$ will be called \emph{branches},
  the connected bounded domain
$C:=\Lambda\setminus \overline{\mathstrut B_1\mathop{\cup}\dots \mathop{\cup} B_n}$
will be called \emph{center}, which is also assumed Lipschitz.
We call such a domain $\Lambda$ a \emph{star waveguide}. Remark that centers of star waveguides are not defined uniquely:
one can attach finite pieces of $B_j$ to a given center to obtain a new center.

For small $\varepsilon>0$, let $\Omega_\varepsilon\subset \RR^d$ be a domain composed of finite cylinders  $B_{j,\varepsilon}$
isometric to $I_j\times (\varepsilon \omega)$ with $I_j:=(0,\ell_j)$, $\ell_j>0$, $j\in\{1,\dots,J\}$,
connected to each other through some bounded Lipschitz domains $C_{k,\varepsilon}$, see Figure~\ref{fig-net}(a). 
In the context of the problem, it is natural to refer to $B_{j,\varepsilon}$ as to \emph{edges}
and to $C_{k,\varepsilon}$ as to \emph{vertices}. We assume that the vertices $C_{k,\varepsilon}$
are isometric to $\varepsilon C_{k}$ with some $\varepsilon$-independent domains $C_k$, $k\in\{1,\dots,K\}$,
and that if one considers a vertex $C_{k,\varepsilon}$ and extends the attached cylindrical edges to infinity,
then one obtains a domain isometric to $\varepsilon \Lambda_k$ with some $\varepsilon$-independent
star waveguide $\Lambda_k$ having $C_k$ as its center.

\begin{figure}[b]
\centering
\centering
\begin{tabular}{cc}
\begin{minipage}[c]{52mm}
\begin{center}
\includegraphics[height=20mm]{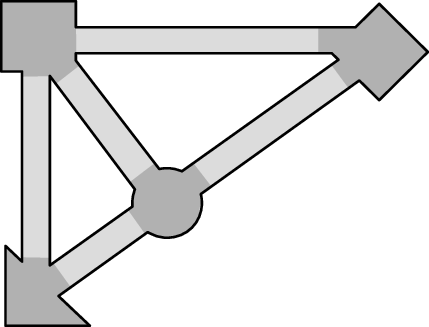}
\end{center}
\end{minipage}
&
\begin{minipage}[c]{52mm}
\begin{center}
\includegraphics[height=18mm]{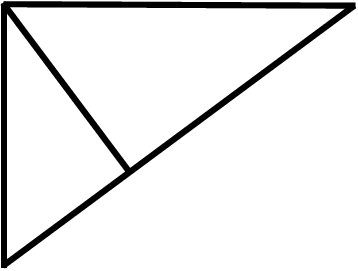}
\end{center}
\end{minipage}\\
(a) & (b)
\end{tabular}
\caption{ (a) An example of a domain $\Omega_\varepsilon$ with dark shaded vertices.
(b)~The associated one-dimensional skeleton $X$. \label{fig-net}}
\end{figure}

In various applications one is interested in the eigenvalues of the Dirichlet laplacian  $-\Delta^{\Omega_\varepsilon}_D$ in $\Omega_\varepsilon$
as $\varepsilon\to 0^+$. As the domain $\Omega_\varepsilon$ collapses onto its one-dimensional skeleton
$X$ composed from the segments $I_j$ coupled at the vertices, see
Figure~\ref{fig-net}(b), it is natural to expect that the behavior of the eigenvalues
should be determined by an effective operator associated with $X$.
The results of~\cite[Theorems 2 and 3]{grieser} can be summarized informally as follows.
Consider the Dirichlet Laplacians  $-\Delta^{\Lambda_k}_D$ in the star waveguides $\Lambda_k$ associated
with each vertex as described above: the spectrum consists of the essential part $[\mu,+\infty)$
and of discrete eigenvalues $E_j(-\Delta^{\Lambda_k}_D)$, $j\in \{1,\dots,N(\Lambda_k)\}$, $k\in\{1,\dots,K\}$.
Then with some $N\ge N(\Lambda_1)+\dots+N(\Lambda_K)$,  $a_n\in (0,\mu]$ and $b>0$ there holds, as $\varepsilon\to 0^+$:
\begin{itemize}
\item for $n\in \{1,\dots,N\}$ there holds $E_n(-\Delta^{\Omega_\varepsilon}_D)=a_n/\varepsilon^2+\cO(e^{-b/\varepsilon})$,
\item for any fixed $n\in\NN$ there holds $E_{N+n}(-\Delta^{\Omega_\varepsilon}_D)=\mu/\varepsilon^2 + E_n(L)+\cO(\varepsilon)$,
where $L$ is a self-adjoint operator in $L^2(X)\simeq \bigoplus_{j=1}^J L^2(0,\ell_j)$ acting as $(f_j)\mapsto (-f''_j)$
with suitable self-adjoint boundary conditions determined by the scattering matrices of $-\Delta^{\Lambda_k}_D$ at the threshold energy~$\mu$ (see e.g. the paper \cite{guip}
by Guilopp\'e for the definition and properties of the scattering matrices).
\end{itemize}
The operator $L$, which is the so-called quantum graph laplacian on $X$ (see the monograph~\cite{bk} by Berkolaiko and Kuchment for an introduction and a review),
represents the sought ''effective operator'' on $X$, and the associated boundary conditions
describe the way how the branches of the network interact through the vertices
in the limit $\varepsilon\to 0$. At the same time, finding explicitly the boundary condition in the general case
represents a very difficult task.

The above general construction admits an important particular case, which can be formulated in simpler terms.
One says that a star waveguide $\Lambda$ admits a \emph{threshold resonance} if there exists a non-zero function $\Phi\in L^\infty(\Lambda)$
satisfying $-\Delta \Phi=\mu \Phi$ in $\Lambda$ and $\Phi=0$ at $\partial\Lambda$, then the following result holds \cite[Section 8]{grieser}:
\begin{prop}\label{prop1}
Assume that none of $\Lambda_k$ admits a threshold resonance, then for $\varepsilon\to 0^+$
the following asymptotics are valid:
\begin{itemize}
\item Denote $N:= N(\Lambda_1)+\dots+N(\Lambda_K)$ and let $a_1,\dots,a_N$ be the family
of the eigenvalues $E_j(-\Delta^{\Lambda_k}_D)$, $j\in\{1,\dots,N(\Lambda_k)\}$, $k\in\{1,\dots,K\}$,
enumerated in the non-decreasing order,
then for $n\in\{1,\dots,N\}$ one has $E_n(-\Delta^{\Omega_\varepsilon}_D)=a_n/\varepsilon^2+\cO(e^{-b/\varepsilon})$,
with some $b>0$,
\item For any fixed $n\ge 1$ there holds $E_{N+n}(-\Delta^{\Omega_\varepsilon}_D)=\mu/\varepsilon^2 + E_j(\oplus_{j=1}^J D_j)+\cO(\varepsilon)$
with $D_j$ being the Dirichlet Laplacians on $(0,\ell_j)$.
\end{itemize}
\end{prop}
In other word, in the absence of threshold resonances the effective operator $L$ is decoupled and corresponds to
the Dirichlet boundary conditions at the vertices. In view of this result, it is important to be able to identify if 
star waveguides admits no threshold resonance. The following sufficient condition was obtained in~\cite{kp-jmaa}, which was in turn motivated
by the analysis of particular configurations carried out by Bakharev, Nazarov, Matveenko \cite{bmn}, Nazarov \cite{naz-t,naz17}, Nazarov, Ruotsalainen, Uusitalo
\cite{naz-hex}. For a star waveguide $\Lambda$ with a center $C$ we denote by $-\Delta^C_{DN}$ the Laplacian in $C$
with the Dirichlet boundary condition of $\partial C\cap\partial\Lambda$ and the Neumann boundary condition at the remaining boundary,
then if for some center $C$ one has the strict inequality
\begin{equation}
   \label{dcn}
E_{N(\Lambda)+1}(-\Delta^C_{DN})>\mu,
\end{equation}
then $\Lambda$ has no threshold resonance. In the recent preprint \cite{bn} Bakharev and Nazarov prove that the condition \eqref{dcn}
for some center $C$ is also necessary for the absence of threshold resonance (hence, it is a necessary and sufficient condition).

By comparing Proposition~\ref{prop1} with our main Theorem~\ref{thm-poly} one sees
that that role of the star waveguides attached to the vertices is quite similar to the role
of the infinite sectors for the Robin laplacians.
In fact our condition of non-resonance (Definition~\ref{defnonres}) is a translation of the condition \eqref{dcn} into the framework of Robin sectors.
Namely, one may rewrite \eqref{dcn}  using the center $\varepsilon C$ of the scaled  waveguide $\varepsilon \Lambda$ as
$E_{N(\Lambda)+1}(-\Delta^{\varepsilon C}_{DN})= \mu/\varepsilon^2+ c/\varepsilon^2$ with $c:=E_{N(\Lambda)+1}(-\Delta^C_{DN})-\mu>0$
and remark that $\mu/\varepsilon^2$ is the bottom of the essential spectrum of the Dirichlet laplacian on $\varepsilon \Lambda$.
This should be compared with the scaled form of the non-resonance condition $E_{\kappa(\theta)+1}(N^\delta_{\theta,\alpha})\ge -\alpha^2+  c/\delta^2$, $c>0$,
as $\alpha \delta$ is large, by noting that $-\alpha^2$ is the bottom of the spectrum of the $\alpha$-Robin laplacian in the infinite sector.
We also remark that the result of Proposition~\ref{prop1} was obtained earlier by Post \cite{post} under the assumption that each $\Lambda_k$
admits a center $C_k$ such that $E_1(-\Delta^{C_k}_{DN})>\mu$, which is exactly the condition \eqref{dcn} for $N(\Lambda)=0$.
In fact, the final steps of our proof (especially the construction of the identification map $J$) are an adaption of those
from \cite{post}. In view of the preceding analogies with the waveguides, it would be interesting to find alternative reformulations
of our non-resonance condition e.g. in terms of generalized eigenfunctions at the bottom of the essential spectrum, which might help
to extend our result to a larger range of angles. It  would also be of interest to understand the eigenvalue asymptotics for general angles
(i.e. without assuming that the angles are non-resonant), which might involve a development
of the scattering theory in infinite sectors similar to the one for waveguides.

\appendix

\section{Some geometric constructions in curvilinear sectors}\label{appa}

Let us introduce a geometric setting which will be used throughout the whole section.

Let $\Gamma_{\pm}$ be two $C^3$ curves meeting at a point at an angle $2\theta \in(0,\pi)$. In this section we would
like to construct some neighborhoods and cut-off functions near the intersection point.
More precisely, let $s_*>0$ and $\gamma_\pm : [-s_*,s_*] \to \RR^2$ be the arc length parametrizations of $\Gamma_\pm$, i.e. both $\gamma_\pm$ are injective $C^3$ functions
with $|\gamma'_\pm|=1$ and $\Gamma_\pm=\gamma_\pm\big([-s_*,s_*]\big)$. By applying suitable rotations and translations we assume without loss of generality that
\begin{equation}\label{eqn:hyprot}
\gamma_\pm(0) =(0,0), \quad
	\gamma_\pm'(0) = (\cos\theta, \pm \sin\theta), \quad
	\theta\in \big(0, \tfrac{\pi}{2}\big).
\end{equation}
In view of the above assumptions, near the point $(0,0)$ the curves $\Gamma_\pm$ are the graphs of $C^3$ functions $F_\pm$ with $\pm F_+(t)> \pm F_-(t)$ for $\pm t>0$, and we will be interested in some constructions in the curvilinear sector
\[
U:=\big\{ (x_1,x_2): \  0<x_1<b, \ F_-(x_1)<x_2<F_+(x_1)\big\}, \quad b>0,
\]
see Figure~\ref{vois12}. For subsequent use we also introduce unit normal vectors $n_\pm(s)$ to $\Gamma_\pm$ at $\gamma_\pm(s)$
which depend smoothly on $s$ and point to the outside of $U$ for small $s$.
In particular, one has then $n_\pm(0)= ( -\sin\theta, \pm \cos\theta)$.
As $n_\pm$ are unit vectors, one has $n_\pm'(s) = k_\pm(s) \gamma_\pm'(s)$, where $k_\pm$ are $C^1$ functions (which coincide up to the sign with the algebraic curvatures on $\Gamma_\pm$),
and $\gamma'_\pm(s)\wedge n_\pm(s)\equiv\pm 1$.

\begin{figure}

\centering

\includegraphics[height=40mm]{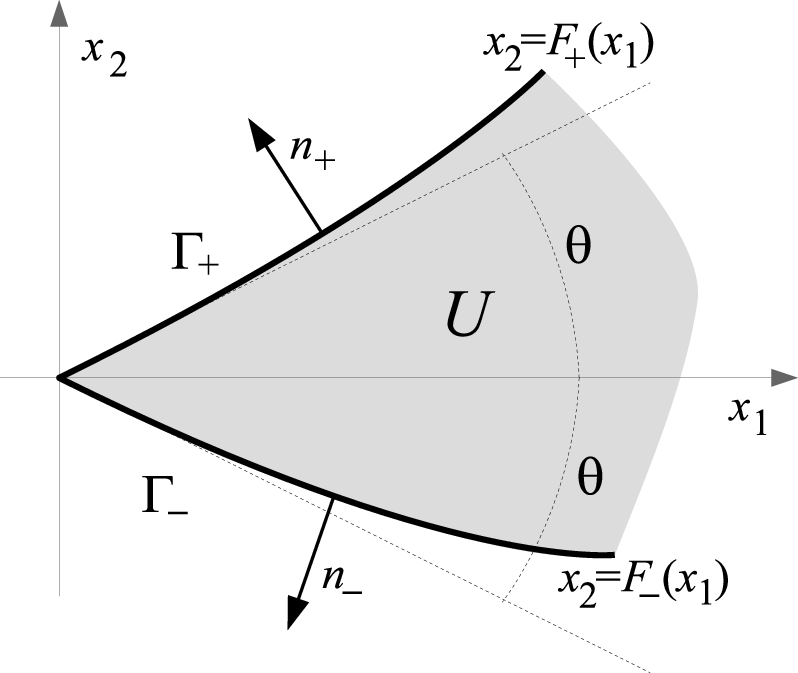}

\caption{The curves $\Gamma_\pm$ and the curvilinear sector $U$. The thin dotted lines correspond to the tangents to $\Gamma_\pm$ at the origin. \label{vois12}}

\end{figure}

%
%\begin{figure}
%\centering
%
%\includegraphics[height=45mm]{vois1.eps}
%
%\caption{The curves $\Gamma_\pm$ and the curvilinear sector $U$. The thin dotted lines correspond to the tangents to $\Gamma_\pm$ at the origin. \label{vois12}}
%\end{figure}

\begin{lemma}\label{prox}
There exist $t_1>0$ and a $C^2$ smooth function $Y:(-t_1,t_1)\to \RR^2$
such that for $t\in(0,t_1)$ the point $Y(t)$ is the unique point of $U$ which is at the distance $t$ from both $\Gamma_+$ and $\Gamma_-$,
and the points $A_\pm(t)\in \Gamma_\pm$ satisfying $\big| A_\pm(t) - Y(t)\big|=t$ are uniquely defined.
Furthermore, $A_\pm(t):=\gamma_\pm\big(\lambda_\pm(t)\big)$, where $\lambda_\pm$ are $C^2$ functions defined near $0$,
and
\[
\lambda_\pm(0)=0, \quad \lambda'_\pm(0)=\cotan\theta, \quad Y(0)=\begin{pmatrix} 0 \\ 0 \end{pmatrix},
\quad
Y'(0)=\dfrac{1}{\sin\theta}\begin{pmatrix} 1 \\ 0 \end{pmatrix}.
\]
\end{lemma}
\noindent The resulting curve 
\begin{align}
\label{anglebisector}
\Sigma:=\big\{ \big(t,Y(t)\big): t\in(-t_1,t_1)\big\},
\end{align}
can be viewed as the curvilinear angle bisector due to its geometric property: each point of $\Sigma$
is at equal distances from the curved sides $\Gamma_\pm$.

\begin{figure}[b]

\includegraphics[width=50mm]{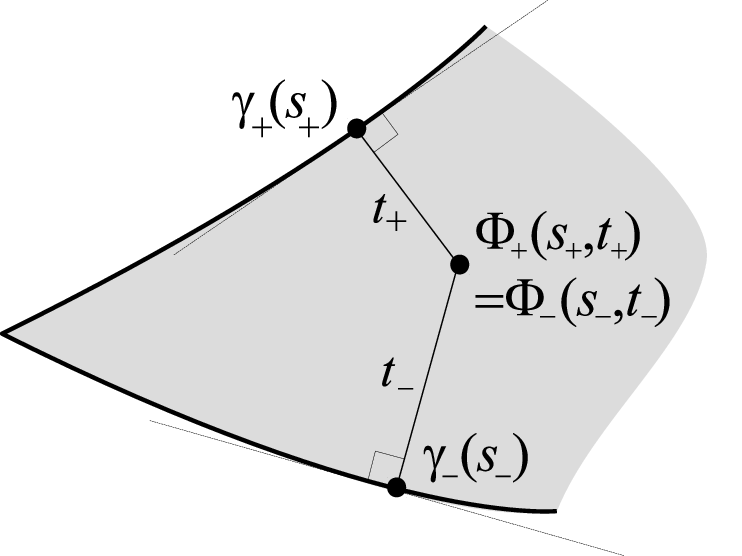}

\caption{The maps $\Phi_\pm$. \label{fig-phis}}

\end{figure}

%
%\begin{figure}
%\includegraphics[height=30mm]{vois4.eps}
%\caption{The maps $\Phi_\pm$. \label{fig-phis}}
%\end{figure}

\begin{proof}
For $t_0>0$ and $s_0\in(0,s_*)$ consider the maps (see Figure~\ref{fig-phis})
%\begin{equation}%\label{eqn:diffeotub}
\[
\Phi_\pm:(-s_0,s_0)\times (-t_0,t_0) \to \RR^2,
	\quad
	\Phi_\pm (s,t)=\gamma_\pm(s) - t n_\pm(s).
\]
%\end{equation}
It is a well known result from the differential geometry that $\Phi_\pm$ are injective for $t_0>0$ small enough, and that
$\dist\big(\Phi_\pm(s,t),\Gamma_\pm\big) = |t|$ and that they are $C^2$-diffeomorphisms from $(-s_0,s_0)\times (-t_0,t_0)$ to its images under $\Phi_\pm$.
One has
\[
\dfrac{\partial \Phi_\pm}{\partial s}(s,t)=\gamma'_\pm(s)- t n'_\pm(s)=\big(1 - t k_\pm(s)\big)\,\gamma'_\pm(s).
\]
Define $G:(-s_0,s_0) \times (-s_0,s_0) \times (-t_0,t_0)\to \RR^2$ by
\[
G(s_+,s_-,t) := \Phi_+(s_+,t) - \Phi_-(s_-,t),
\]
then $G(0,0,0) = \gamma_+(0) - \gamma_-(0) = (0,0)$ and $\partial G/\partial s_\pm(s_+,s_-,t) =	\pm\big(1 - tk_\pm(s_\pm)\big) \gamma_\pm'(s_\pm)$,
and the two vectors $\partial G/\partial s_\pm(0,0,0)=\pm\gamma'_\pm(0)$
are linearly independent. Hence, it follows by the implicit function theorem
that there exist $t_1>0$ and $s_1>0$ and $C^2$ functions $\lambda_\pm:(-t_1,t_1)\to (-s_1,s_1)$ with $\lambda_\pm(0)=0$
such that for $(s_+,s_-,t)\in (-s_1,s_1)\times (-s_1,s_1)\times (-t_1,t_1)$ one has the equivalence:
$G(s_+,s_-,t) = 0$ if and only if $s_\pm = \lambda_\pm(t)$.
If one defines a $C^2$ function $Y: (-t_1,t_1)\to \RR^2$ by $Y(t):=\Phi_\pm \big(\lambda_\pm(t), t\big)$,
then for any $t\in(0,t_1)$ the point $Y(t)$ is the unique point of $U$ satisfying $\dist \big( Y(t),\Gamma_\pm)=t$,
and the points $A_\pm(t)$ of $\Gamma_\pm$ which are the closest to $Y(t)$ are $A_\pm(t)=\gamma_\pm \big( \lambda_\pm(t)\big)$.
One differentiates $G\big(\lambda_+(t),\lambda_-(t),t\big)=0$ in $t$ to arrive at
\begin{multline*}
	\lambda_+'(t) \Big[1-t k_+\big(\lambda_+(t)\big)\Big] \gamma'_+\big(\lambda_+(t)\big)
	-
	\lambda_-'(t) \Big[1-t k_-\big(\lambda_-(t)\big)\Big] \gamma'_-\big(\lambda_-(t)\big)\\
	- \Big[n_+\big(\lambda_+(t)\big) -n_-\big(\lambda_-(t)\big)\Big]=0.
\end{multline*}
For $t=0$ one has $\lambda_+'(0) \gamma_+'(0) - \lambda_-'(0) \gamma_-'(0) = n_+(0) - n_-(0)$, i.e.
\begin{gather*}
\begin{pmatrix}
\cos \theta & - \cos\theta\\
\sin\theta & \sin \theta
\end{pmatrix}
\begin{pmatrix}
\lambda'_+(0)\\
\lambda'_-(0)
\end{pmatrix}=\begin{pmatrix} 0 \\ 2\cos\theta \end{pmatrix},
\end{gather*}
which gives
\begin{gather*}
\begin{pmatrix}
\lambda'_+(0)\\
\lambda'_-(0)
\end{pmatrix}= \dfrac{1}{2\sin\theta \cos\theta}
\begin{pmatrix}
\sin \theta &  \cos\theta\\
-\sin\theta &  \cos \theta
\end{pmatrix}
\begin{pmatrix} 0 \\ 2\cos\theta \end{pmatrix}
=
\begin{pmatrix}
\cotan \theta \\ \cotan \theta
\end{pmatrix}.
\end{gather*}
Then
\begin{align*}
Y'(t)&= \dfrac{d}{dt} \, \Phi_+\big(\lambda_+(t),t\big)\\
&=\lambda_+'(t) \Big[1-t k_+\big(\lambda_+(t)\big)\Big] \gamma'_+\big(\lambda_+(t)\big)-n_+\big(\lambda_+(t)\big),\\
Y'(0)&= \lambda'_+(0)\gamma'_+(0) - n_+(0)\\
&=\cotan\theta \begin{pmatrix} \cos\theta \\ \sin \theta \end{pmatrix} - \begin{pmatrix} - \sin\theta \\ \cos\theta\end{pmatrix}
=\dfrac{1}{\sin \theta} \begin{pmatrix} 1 \\ 0 \end{pmatrix}.  \qedhere
\end{align*}
\end{proof}

\begin{figure}[b]

\centering

\includegraphics[height=35mm]{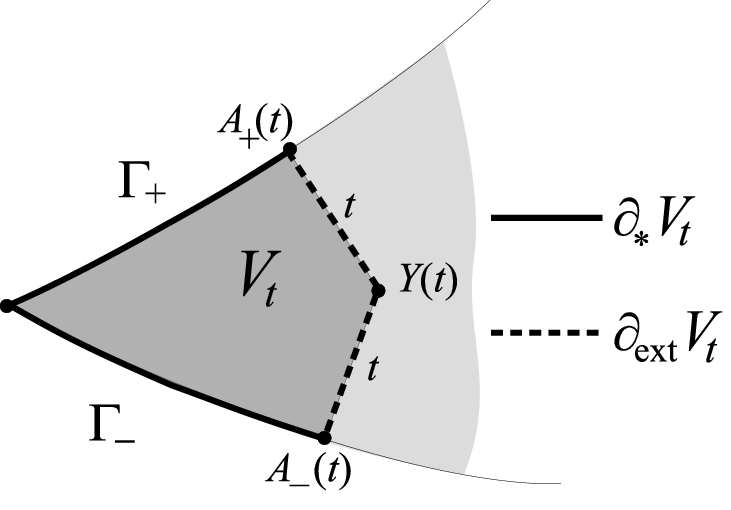}

\caption{Construction of the domain $V_t$ \label{vois12b}}

\end{figure}

%
%
%
%
%\begin{figure}
%\centering
%
%\includegraphics[height=35mm]{vois2.eps}
%
%\caption{Construction of the domain $V_t$ \label{vois12b}}
%\end{figure}

Using the objects defined in Lemma~\ref{prox} we introduce the following sets $V_t$:
\begin{defi}
For $t\in(0,t_1)$
denote by $V_t$ the interior of the curvilinear quadrangle bounded by the pieces of $\Gamma_\pm$
enclosed between the points $(0,0)$ and $A_\pm(t)$ and by the straight line segments connecting $Y(t)$ to $A_\pm(t)$. We refer to Figure~\ref{vois12b}
for an illustration. One will distinguish between two parts of its boundary, i.e.
one denotes 
\[
\partial_* V_t:=\partial V_t \mathop{\cap} \,(\Gamma_+\mathop{\cup}\Gamma_-), 
\text{ and } \partial_\ext V_t:=\partial V_t\setminus \partial_* V_t.
\]
\end{defi}
Then, we would like  to ``straighten'' $V_t$ in a controlable way in order to obtain a truncated curvilinear sector $\cS^r_{\theta}$ (see Definition~\ref{def25}).

\begin{lemma}\label{mapfi}
There is a bi-Lipschitz map $\Phi$ between two neighborhoods of the origin 
with $\Phi'(x)=I_2+\cO\big(|x|\big)$ for $x\to 0$
and a $C^2$ smooth function $r$ defined near $0$ with $r(0)=0$ and $r'(0)=\cotan\theta$
such that $\Phi(\cS_{\theta}^{r(t)})=V_t$, $\Phi(\partial_*\cS_{\theta}^{r(t)})=\partial_* V_t$
and $\Phi(\partial_\ext\cS_{\theta}^{r(t)})=\partial_\ext V_t$
for all sufficiently small $t>0$.
\end{lemma}

\begin{figure}[t]

\centering

\includegraphics[width=40mm]{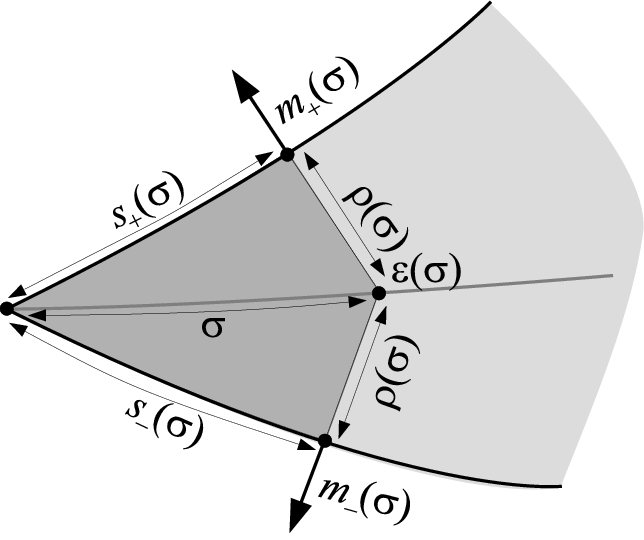}

\caption{Parametrization with the arc-length. The unit vectors $m_\pm$ are orthogonal to the boundary. The arrows indicate the length
of the corresponding arcs. For small $\sigma$
one has $\rho(\sigma)=\sigma \sin \theta +\cO(\sigma^2)$ and $s_\pm(\sigma)=\sigma\cos\theta+\cO(\sigma^2)$.\label{vois3}}

\end{figure}

\begin{proof}
Without loss of generality we may assume that $t_1>0$ is sufficiently small such that $Y'(t)\ne 0$ for $t\in[-t_1,t_1]$.
Let us introduce an arc-length parametrization of the curvilinear angle bisector $\Sigma$ introduced in \eqref{anglebisector}: consider
the function $\sigma$ with $\sigma(0)=0$ and $\sigma'=|Y'|$, i.e. $\sigma(t)$ is the length of $Y\big([0,t]\big)$.
One has $\sigma'(0)=\big|Y'(0)\big|=1/\sin\theta$ and 
 $\sigma'=|Y'|>0$ on $[-t_1,t_1]$. 
Hence, $\sigma:[-t_1,t_1]\to[-\sigma_-,\sigma_+]$ is a $C^2$ diffeomorphism for some $\sigma_\pm>0$.
Denote by $\rho :[-\sigma_-,\sigma_+]\to [-t_1,t_1]$ its inverse, which is then also $C^2$ and satisfies $\rho(0)=0$ and $\rho'(0)=1/\sigma'(0)=\sin\theta$.
Finally, let us pick a small $\delta>0$ and define $\varepsilon:=Y\circ\rho: (-\delta,\delta) \to \RR^2$,
then one has $|\varepsilon'|=1$, $\varepsilon'(0)=(1,0)^T$, and $Y\big([0,t]\big)=\varepsilon\big( \big[0,\sigma(t)\big]\big)$ for small $t>0$,
i.e. $\varepsilon$ is an arc-length parametrization of $\Sigma$ near the origin.
By construction, the point  $\varepsilon(\sigma)$  is then the unique point of $U$ with $\dist\big(\varepsilon(\sigma), \Gamma_\pm\big)=\rho(\sigma)$,
and for small $\sigma$ one has $\rho(\sigma)= \sigma \sin\theta+\cO(\sigma^2)$. Furthermore, if one sets
$s_\pm(\sigma)\coloneqq \lambda_\pm\big(\rho(\sigma)\big)$, then $s_\pm(\cdot)$ are $C^2$ functions
with $s'_\pm(0)=\lambda'_\pm(0)\rho'(0)=\cos\theta$, and the points $B_\pm(\sigma)\coloneqq\gamma_\pm \big(s_\pm(\sigma)\big)$
of $\Gamma_\pm$ are the closest to $\varepsilon(\sigma)$. We also reparametrize the normal vectors to $\Gamma_\pm$ by setting
$m_\pm(\sigma):= n_\pm\big(s_\pm(\sigma)\big)$, then one has $B_\pm(\sigma)=\varepsilon(\sigma)+ \rho(\sigma) m_\pm(\sigma)$.
The above constructions are illustrated in Figure~\ref{vois3}.

For the $C^2$ maps $\Psi_\pm: (\sigma,\tau) \mapsto \varepsilon(\sigma)+\tau m_\pm(\sigma)$ 
one has $\Psi_\pm(0,0)=(0,0)$ and
\[
\Psi'_\pm (0,0)= \begin{pmatrix}
\dfrac{\partial\Psi_\pm}{\partial\sigma}(0,0) & \dfrac{\partial\Psi_\pm}{\partial\tau}(0,0)\end{pmatrix}
=\begin{pmatrix} \varepsilon'(0) &  m_\pm (0)\end{pmatrix} =\begin{pmatrix} 1 & -\sin\theta \\ 0 &\pm \cos\theta\end{pmatrix},
\]
i.e. the Jacobian matrix $\Psi'_\pm (0,0)$ is invertible. Therefore, the maps $\Psi_\pm$ are diffeomorphisms between suitable neighborhoods
of the origin.
Furthermore, if for $t>0$ one introduces the curvilinear triangles $\Lambda_t:=\big\{ (\sigma,\tau): \quad 0<\sigma<\sigma(t), \quad 0<\tau <\rho(\sigma)\big\}$,
then the  image $\Psi_\pm( \overline{\Lambda_t} )$ is exactly the closure of the upper/lower $V^\pm_t$ part of $V_t$, i.e. of the part of $V_t$ lying above/below $\Sigma$,
and $\Psi_+(\cdot,0)=\Psi_-(\cdot,0)$. We now use this observation to construct a map $\Phi$ with the sought properties.
Namely, in addition to the above curvilinear triangles $\Lambda_t$ let us consider its ``straightened'' version
$L_t=\big\{ (\sigma,\tau): \quad 0<\sigma<\sigma(t), \quad 0<\tau < \sigma \sin\theta\big\}$.
obtained by replacing $\rho$ through its linear approximation at $0$.
The map $H: (\sigma,\tau) \mapsto\big( \sigma , \rho(\sigma) \tau/(\sigma\sin\theta)\big)$
satisfies then $H'(0,0)=I_2$, hence, it is a diffeomorphism between suitable neighborhoods of the origin,
and for sufficiently small $t>0$ it is bijective from $\overline{L_t}$ to $\overline{\Lambda_t}$.

\begin{figure}[b]
\includegraphics[width=110mm]{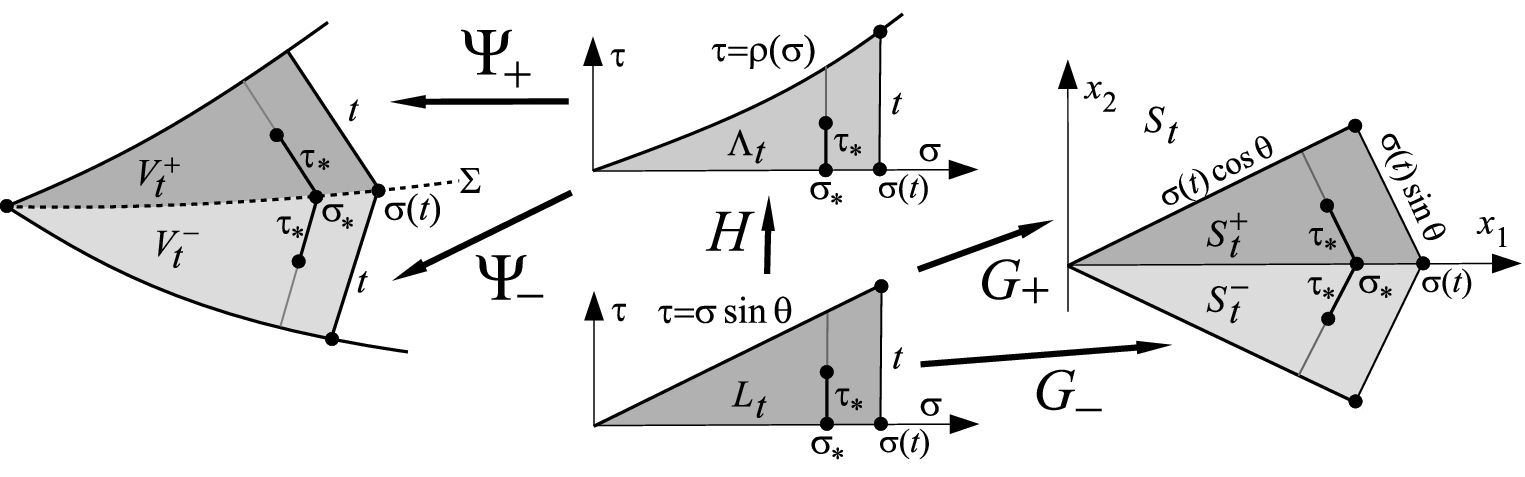}
\caption{The maps $\Psi_\pm$, $H$ and $G_\pm$ in the proof of Lemma~\ref{mapfi}.\label{diffeos}}
\end{figure}

Now let us consider the truncated sector $S_t:=\cS^{\sigma(t)\cos\theta}_{\theta}$
and their upper/lower parts $S_t^\pm:=S_t\mathop{\cap} \{(x_1,x_2):\, \pm x_2>0\}$.
One easily sees that the maps
\[
G_\pm: \RR^2\to \RR^2, \quad (\sigma,\tau)\mapsto \sigma \begin{pmatrix} 1 \\ 0 \end{pmatrix} + \tau \begin{pmatrix} - \sin\theta \\ \pm \cos\theta \end{pmatrix}
\]
are diffeomorphisms, and $\Bar S_t^\pm=G_\pm( \Bar L_t)$ for small $t>0$, and the inverses are given by
\[
G_\pm^{-1}(x_1,x_2)=\begin{pmatrix}
1 & \pm \tan\theta\\
0 & \pm\dfrac{1}{\cos\theta}
\end{pmatrix} \begin{pmatrix} x_1 \\ x_2 \end{pmatrix}.
\]
We refer to Figure~\ref{diffeos} for a graphical representation of the above maps.

Now let us define $\Phi$ by $\Phi(x_1,x_2)= \Psi_\pm \circ H \circ G_\pm^{-1} (x_1,x_2)$ for $\pm x_2>0$,
which then extends by continuity to $x_2=0$ due to
\[
\Psi_\pm \circ H \circ G_\pm^{-1} (x_1,0)= \Psi_\pm \circ H (x_1,0)=\Psi_\pm (x_1,0)= \varepsilon(x_1).
\]
By construction, the map $\Phi$ is $C^2$ on $\{\pm x_2\ge 0\}$ and continuous along $x_2=0$, hence it is Lipschitz.
Furthermore, by construction it defines bijections $S_t^\pm\to V^\pm_t$, $S_t\to V_t$ as well
as $\partial_* S_t\to \partial_* V_t$ and $\partial_\ext S_t\to \partial_\ext V_t$.
To estimate the Jacobian matrix $\Phi'$ we compute
\begin{align*}
(\Psi_\pm \circ H \circ G_\pm^{-1})'(0,0)&=\Psi'_\pm(0,0) H'(0,0) (G_\pm^{-1})'(0,0)\\
&=
\begin{pmatrix} 1 & -\sin\theta \\ 0 &\pm \cos\theta\end{pmatrix} \begin{pmatrix} 1 & 0 \\ 0 & 1\end{pmatrix}
\begin{pmatrix}
1 & \pm \tan\theta\\
0 & \pm\dfrac{1}{\cos\theta}
\end{pmatrix}=\begin{pmatrix} 1 & 0 \\ 0 & 1\end{pmatrix}.
\end{align*}
As $\Phi_\pm$ are $C^1$, it follows that $\Phi'_\pm=I_2+\cO(t)$ in $V_t$, which shows the requested property for $\Phi'$.
As $\Phi^{-1}_\pm$ are $C^1$ near the origin and $\Phi^{-1}$ is continuous by construction, it follows that $\Phi^{-1}$
is Lipschitz, therefore, the map $\Phi$ is bi-Lipschitz. Hence, we obtain the claim with $r(t)=\sigma(t)\cos\theta$, and $r'(0)=\sigma'(0)\cos\theta=\cotan\theta$.
\end{proof}

For later references we mention explicitly the following corollary, which is quite obvious from the geometric point of view:
\begin{coro}\label{lem-length}
There exist $0<a<b$ such that for all sufficiently small $t>0$ there holds
  $|x|<bt$ for $x\in V_t$, and $|x|>at$ for $x\in V_s\setminus \overline{V_t}$ and $s>t$.
\end{coro}

\begin{proof}
Let us use a map $\Phi$ and a function $r$ as in Lemma~\ref{mapfi}.
Remark first that 
\begin{equation}
  \label{yyy1}
\begin{aligned}
&\text{$|y|< \dfrac{r}{\cos\theta}$ for $y\in\cS^r_\theta$ and $r>0$},\\
&\text{$|y|> r$ for $y\in\cS_\theta^R\setminus\overline{\cS^r_\theta}$ and $R>r>0$.}
\end{aligned}
\end{equation}
As $v\in V_t$ iff $v=\Phi(y)$ with $y\in \cS^{r(t)}_\theta$ and $r(t)=\cO(t)$,
by applying the Taylor expansion of $\Phi$ near the origin one obtains $\frac{1}{2}|y|\le|v|\le 2|y|$.
Using the estimates \eqref{yyy1} one arrives at the result.
\end{proof}

We complete this subsection by a construction of cut-off functions with some special properties:

\begin{lemma}\label{cutoff}
Let $0<a<b$, then there exist $\delta_0>0$, $\eta>0$, $K>0$ and $C^2$ functions
$\varphi_\delta:\overline{V_{\eta}}\to \RR$ with $\delta\in(0,\delta_0)$ such that:
\begin{itemize}
\item[(a)] $0\le \varphi_\delta\le 1$, and for all $\beta\in\NN^2$ with $1\le |\beta|\le 2$
there holds $\|\partial^\beta \varphi_\delta\|_\infty\le K \delta^{-|\beta|}$,
\item[(b)] $\varphi_\delta=1$ in $V_{a\delta}$,
\item[(c)] $\varphi_\delta=0$ in $V_\eta\setminus \overline{V_{b\delta}}$,
\item[(d)] the normal derivative of $\varphi_\delta$ at $\Gamma_\pm$ is zero.
\end{itemize}
\end{lemma}

\begin{proof}

For small $t_0>0$ and $s_0>0$ consider the maps
%\begin{equation}\label{eqn:diffeotub}
\[
\Phi_\pm:(-s_0,s_0)\times (-t_0,t_0) \to \RR^2,
	\quad
	\Phi_\pm (s,t)=\gamma_\pm(s) - t n_\pm(s).
\]
%\end{equation}
It is a well known result of differential geometry that $\Phi_\pm$ are injective for $t_0>0$ small enough, with
$\dist\big(\Phi_\pm(\cdot,t),\Gamma_\pm\big) = |t|$ for $|t|<t_0$, and that they are $C^2$-diffeomorphisms from $(-s_0,s_0)\times (-t_0,t_0)$ to its images under $\Phi_\pm$.
Remark that one has
\[
\dfrac{\partial \Phi_\pm}{\partial s}(s,t)=\gamma'_\pm(s)- t n'_\pm(s)=\big(1 - t k_\pm(s)\big)\,\gamma'_\pm(s),
\quad
\dfrac{\partial \Phi_\pm}{\partial t}(s,t)=- n_\pm(s),
\]
i.e. if one writes $(\tau^\pm_1,\tau^\pm_2):=\gamma'_\pm$ and $(n^\pm_1,n^\pm_2):=n_\pm$, then
\[
\Phi'_\pm(s,t)=\begin{pmatrix}
\big(1 - t k_\pm(s)\big) \tau^\pm_1(s) & -n^\pm_1(s) \\
\big(1 - t k_\pm(s)\big) \tau^\pm_2(s) & -n^\pm_2(s)
\end{pmatrix}.
\]
By choosing $\eta>0$ sufficiently small one can then invert the maps $(s,t)\mapsto \Phi_\pm(s,t)$ near the origin
in order to obtain $C^2$ functions $s_\pm$ and $t_\pm$ on $V_{\eta}$.
The inverse function theorem gives
\begin{gather}
   \label{grads}
\nabla s_\pm (x)=\pm\dfrac{1}{1-t_\pm(x)K_\pm(x)}\Big( N^\pm_2(x), -N^\pm_1(x)\Big), \\
K_\pm:= k_\pm\circ s_\pm, \quad
N^\pm_j:=n^\pm_j\circ s_\pm.\nonumber
\end{gather}
In particular, $s_\pm(0,0)=0$ and $\nabla s_\pm (0,0)=(\cos\theta,\pm\sin\theta)$, therefore,
\[
s_\pm(x_1,x_2)=(\cos\theta,\pm\sin\theta)\cdot(x_1,x_2) + \cO(x_1^2+x_2^2) \text{ for } (x_1,x_2)\to (0,0).
\]
We further remark that for small $s$ one has obviously $s_\pm\big(\gamma_\pm(s)\big)=s$,
while
\begin{equation}
     \label{tays}
s_\pm\big(\gamma_\mp(s)\big)=(\cos\theta,\pm\sin\theta)\cdot \gamma'_\mp(0) s + \cO(s^2)\equiv
\cos (2\theta) \, s + \cO(s^2) \text{ for } s\to 0.
\end{equation}
Let us pick some $c\in(a\cotan\theta,b\cotan\theta)$ and then a sufficiently small $\varepsilon>0$ satisfying
\begin{equation}
   \label{eq-abc}
	[c-\varepsilon,c+\varepsilon]\subset (a\cotan\theta,b\cotan\theta), \qquad \cos (2\theta) (c+\varepsilon)< c-\varepsilon.
\end{equation}
We remark that the second condition follows from the first one for $\theta\ge\frac{\pi}{4}$.
Let $\psi:\RR\to[0,1]$ be a $C^\infty$ function with $\psi(s) = 1$ for $s<c-\varepsilon$ and $\psi(s) = 0$ for $s>c+\varepsilon$.
For small $\delta>0$ we define then $\varphi_\delta:\overline{V_\eta}\to \RR$ by $\varphi_\delta(x)= \psi\big(s_+(x)/\delta\big) \psi\big(s_-(x)/\delta\big)$.
Note that the property (a) is automatically satisfied due to the the $C^2$ smoothness of the functions $s_\pm$.

In order to see the properties (b) and (c) we first remark that due to Lemma~\ref{prox}
the definition of the domain $V_t$ for small $t$ can be reformulated as
$V_t:=\big\{ x\in V_\eta: \, s_\pm(x)<\lambda_\pm(t)\big\}$,
and for small $\delta$ and a fixed $A>0$ one has $\lambda_\pm(A\delta)=A\delta \cotan\theta +\cO(\delta^2)$.
In particular, for $x\in V_{a\delta}$ one has $s_\pm(x)\le a\delta\cotan\theta +\cO(\delta^2)< (c-\varepsilon) \delta$
as $\delta$ is small, which shows that $\varphi_\delta(x)=1$ and proves the claim (b).
Furthermore, for $x\notin V_{b\beta}$ one of the following two inequalities holds:
$s_\pm(x)>\lambda_\pm(b\delta)$. As $\lambda_\pm(b\delta)=b\delta\cotan\theta+\cO(\delta^2)>(c+\varepsilon)\delta$,
it follows that at least one of the terms $s_\pm(x)/\delta$ is greater than $c+\varepsilon$.
As $\psi$ vanishes in $(c+\varepsilon,+\infty)$, it follows that $\varphi_\delta(x)=0$. This proves the claim (c).

Let us finally show the property (d). For a better readability we give the computation of the normal derivative on $\Gamma_+$ only,
the case of $\Gamma_-$ is handled in a completely similar way. For $x= \gamma_+(s)\in\Gamma_+$ with $s>0$ one has
\begin{multline*}
\dfrac{\partial \varphi_\delta}{\partial n_+} (x)=
n_+(s)\cdot (\nabla \varphi_\delta)\big(\gamma_+(s)\big)\\
=\dfrac{1}{\delta}\, n_+(s)\cdot \Big[
(\nabla s_+)\big(\gamma_+(s)\big) \psi'\Big( \dfrac{s_+\big(\gamma_+(s)\big)}{\delta}\Big)\psi\Big( \dfrac{s_-\big(\gamma_+(s)\big)}{\delta}\Big)\\
+(\nabla s_-)\big(\gamma_+(s)\big) \psi'\Big( \dfrac{s_-\big(\gamma_+(s)\big)}{\delta}\Big)\psi\Big( \dfrac{s_+\big(\gamma_+(s)\big)}{\delta}\Big)
\Big].
\end{multline*}
By \eqref{grads} one has $(\nabla s_+)\big(\gamma_+(s)\big)=\big(n^+_2(s),-n^+_1(s)\big)$, which gives $n_+(s)\cdot(\nabla s_+)\big(\gamma_+(s)\big)=0$,
and the preceding expression simplifies to
\[
\dfrac{\partial \varphi_\delta}{\partial n_+}\big(\gamma_+(s)\big)
=\Big[\dfrac{1}{\delta}\,n_+(s)\cdot (\nabla s_-)\big(\gamma_+(s)\big) \Big]
\psi'\Big( \dfrac{s_-\big(\gamma_+(s)\big)}{\delta}\Big)\psi\Big(\dfrac{s}{\delta}\Big).
\]
Let us show that the product of the last two terms is zero for small $\delta$, i.e.
that $\psi'\big(\xi(s)\big)\psi(s/\delta)=0$ for $\xi(s):=s_-\big(\gamma_+(s)\big)/\delta$.
First, by construction of $\psi$ the second factor vanishes for $s\ge(c+\varepsilon)\delta$.
Therefore, one needs to show that $\psi'\big(\xi(s)\big)=0$ for all $0<s\le (c+\varepsilon)\delta$ as $\delta$ is sufficiently small.
Using the Taylor expansion \eqref{tays} for small~$\delta$ we have
$\xi(s)= \cos(2\theta) s/\delta + \cO(\delta)$. If $\theta\ge\frac{\pi}{4}$, then $\cos(2\theta)\le 0$,
and $\xi(s)\le\cO(\delta)<c-\varepsilon$. If $\theta <\frac{\pi}{4}$, then $\cos(2\theta)>0$,
and due to the choice of $\varepsilon$ made in \eqref{eq-abc} one obtains $\xi(s)\le \cos (2\theta)(c+\varepsilon)+\cO(\delta)<c-\varepsilon$.
Therefore, in both cases one has $\xi(s)<c-\varepsilon$ for all $0<s<(c+\varepsilon)\delta$ as $\delta$ is sufficiently small.
As $\psi$ was chosen constant on $(-\infty,c-\varepsilon)$, we have $\psi'\big(\xi(s)\big)=0$.
\end{proof}

\def\bysame{\leavevmode ---------\thinspace}
\makeatletter\if@francais\providecommand{\og}{<<~}\providecommand{\fg}{~>>}
\else\gdef\og{``}\gdef\fg{''}\fi\makeatother
\def\cdrandname{\&}
\providecommand\cdrnumero{no.~}
\providecommand{\cdredsname}{eds.}
\providecommand{\cdredname}{ed.}
\providecommand{\cdrchapname}{chap.}
\providecommand{\cdrmastersthesisname}{Memoir}
\providecommand{\cdrphdthesisname}{PhD Thesis}

\end{document}